\documentclass{amsart}
\usepackage[UKenglish]{babel}
\usepackage[a4paper, total={6in, 8in}]{geometry}
\usepackage{amsfonts, calligra, mathrsfs}
\usepackage[T1]{fontenc}
\usepackage{lmodern}
\usepackage{amsmath}
\usepackage{amssymb}
\usepackage{amsthm}
\usepackage{amscd}
\usepackage{mathtools}
\usepackage{tikz}
\usepackage{tikz-cd}
\usepackage{multirow}
\usepackage{graphicx}
\usepackage{microtype}
\usepackage{enumerate}
\usepackage{xcolor}
\usepackage{cancel}
\usepackage{url}
\usepackage{chngcntr}
\usepackage{apptools}
\usepackage{hyperref}
\usepackage[capitalise]{cleveref}
\usepackage{mathabx,epsfig}
\AtAppendix{\counterwithin{lem}{section}}

\numberwithin{equation}{section}

\theoremstyle{plain}
\newtheorem{lem}{Lemma}[section]

\newtheorem*{lem*}{Lemma}
\newtheorem*{claim*}{Claim}

\newtheorem{prop}[lem]{Proposition}
\newtheorem*{theor*}{Theorem}
\newtheorem{theor}[lem]{Theorem}
\newtheorem{cor}[lem]{Corollary}
\newtheorem{definition}[lem]{Definition}
\newtheorem{convention}[lem]{Convention}
\newtheorem{assumption}[lem]{Assumption}
\theoremstyle{remark}
\newtheorem{rem}[lem]{Remark}

\newcommand{\simto}{\,\widetilde{\to}\,}

\newcommand{\pa}{{\mathcal{P}}}
\newcommand{\oo}{{\mathcal{O}}}
\newcommand{\m}{{\mathcal{M}}}
\newcommand{\q}{{\mathbb{Q}}}
\newcommand{\qq}{{\mathcal{Q}}}
\newcommand{\sss}{{\mathcal{S}}}
\newcommand{\p}{{\mathbb{P}}}
\newcommand{\LL}{{\mathbb{L}}}

\newcommand{\aaa}{{\alpha}}
\newcommand{\ua}{{U_{\aaa}}}

\newcommand{\z}{{\mathbb{Z}}}
\newcommand{\n}{{\bf n}}

\newcommand{\uij}{{U_{ij}}}
\newcommand{\gt}{{\tilde{\gamma}}}
\newcommand{\ga}{\gamma}

\newcommand{\into}{{\hookrightarrow}}
\newcommand{\onto}{{\twoheadrightarrow}}
\newcommand{\aff}{{\mathbb{A}^3}}

\newcommand{\affor}{{\mathcal{O}_{\mathbb{A}^3}^r}}

\newcommand{\F}{{\mathcal{F}}}
\newcommand{\G}{{\mathcal{G}}}

\newcommand{\X}{{\mathscr{X}}}
\newcommand{\C}{{\mathscr{C}}}
\newcommand{\CC}{{\mathbb{C}}}
\newcommand{\pn}{{\mathscr{P}_n}}

\newcommand{\A}{{\mathscr{A}}}
\newcommand{\B}{{\mathscr{B}}}
\newcommand{\D}{{\mathscr{D}}}
\newcommand{\DD}{{\mathcal{D}}}
\newcommand{\U}{{\mathscr{U}}}

\newcommand{\T}{{\mathbb{T}}}

\newcommand*{\bigchi}{\mbox{\Large$\chi$}}

\newcommand*{\qn}{{\mathcal{Q}_\n}}
\newcommand*{\qa}{{\mathcal{Q}_{\mathcal{A}}}}
\newcommand*{\qb}{{\mathcal{Q}_{\mathcal{B}}}}

\newcommand{\disj}{%
	\mathrel{\rotatebox[origin=c]{180}{$\prod$}}}

\newcommand{\prodd}{%
	\mathrel{\rotatebox[origin=c]{180}{$\disj$}}}

\newcommand*{\sheafhom}{\mathcal{H}\kern -.5pt om}
\newcommand*{\sheafext}{\mathcal{E}\kern -.5pt xt}
\newcommand*{\quot}{\mathcal{Q}\kern -.5pt uot}
\newcommand*{\qxc}{\quot_{\mathscr{X}\kern -.5pt/\kern -.5pt\mathscr{C}}^{\F\kern -.5pt, n}}
\newcommand*{\qda}{\quot_{\mathscr{D}\kern -.5pt \subset \kern -.5pt\mathscr{A}}^{\F\kern -.5pt, n}}
\newcommand*{\qdaone}{\quot_{\mathscr{D}\kern -.5pt \subset \kern -.5pt\mathscr{A}}^{\F\kern -.5pt, n_1}}
\newcommand*{\qdbtwo}{\quot_{\mathscr{D}\kern -.5pt \subset \kern -.5pt\mathscr{B}}^{\F\kern -.5pt, n_2}}
\newcommand*{\qxcm}{\quot_{\mathscr{X}_0^+\kern -.5pt/\kern -.5pt\mathscr{C}_0^+}^{\F\kern -.5pt, \n}}
\newcommand*{\qxcz}{\quot_{\mathscr{X}_0\kern -.5pt/\kern -.5pt\mathscr{C}_0}^{\F\kern -.5pt, n}}

\newcommand*{\cz}{{\mathscr{C}_0^+}}

\newcommand*{\rpl}{{R\pi_*}}
\newcommand*{\rph}{{R\sheafhom_{\pi}}}
\newcommand*{\rpih}{{R\sheafhom_{\pi_i}}}

\newcommand*{\rphalpha}{{R\sheafhom_{\pi_{\alpha}}}}
\newcommand*{\rh}{{R\sheafhom}}
\newcommand*{\extpi}[1]{\sheafext_\pi^{#1}}

\newcommand*{\extpialpha}[1]{\sheafext_{\pi_\alpha}^{#1}}

\newcommand*{\op}{{\omega_\pi}}

\newcommand*{\ttt}{{{\bf t}}}

\newcommand*{\DT}{{\mathsf{Q}}}

\newcommand{\xrightarrowdbl}[2][]{%
	\xrightarrow[#1]{#2}\mathrel{\mkern-14mu}\rightarrow
}

\DeclareMathOperator{\Spec}{Spec}
\DeclareMathOperator{\supp}{supp}
\DeclareMathOperator{\Quot}{Quot}
\DeclareMathOperator{\Hilb}{Hilb}
\DeclareMathOperator{\Grass}{Grass}
\DeclareMathOperator{\Gr}{Gr}
\DeclareMathOperator{\rk}{rk}
\DeclareMathOperator{\DTT}{DT}
\DeclareMathOperator{\M}{M}
\DeclareMathOperator{\Bl}{Bl}

\DeclareMathOperator{\Aff}{Aff}
\DeclareMathOperator{\sheafHom}{\mathscr{H}\text{\kern -3pt {\calligra\large om}}\,}
\DeclareMathOperator{\Hom}{Hom}
\DeclareMathOperator{\RHom}{RHom}
\DeclareMathOperator{\Ext}{Ext}

\DeclareMathOperator{\vir}{vir}

\DeclareMathOperator{\HH}{H}
\DeclareMathOperator{\Cone}{Cone}

\DeclareMathOperator{\id}{id}
\DeclareMathOperator{\im}{Im}
\DeclareMathOperator{\der}{der}

\DeclareMathOperator{\vd}{vd}
\DeclareMathOperator{\fix}{fix}

\DeclareMathOperator{\Coh}{Coh}
\DeclareMathOperator{\pt}{pt}

\newcommand{\Q}{{\Quot_Y(F, n)}}

\usepackage{microtype}

\title{Virtual invariants of Quot schemes of points on threefolds}
\author{Solomiya Mizyuk}
\address{smizyuk@sissa.it}

\begin{document}
	\begin{abstract}
		We construct an almost perfect obstruction theory of virtual dimension zero on the Quot scheme 
		parametrizing zero-dimensional quotients of a locally free sheaf on a smooth projective $3$-fold.
		This gives a virtual class in degree zero and therefore allows one to define virtual invariants of the Quot scheme. We compute these invariants
		proving a conjecture
		by Ricolfi. 
		The computation is done by reducing to the toric case via cobordism theory and a degeneration argument. The toric case is solved by reducing to the computation on the Quot scheme of points on $\mathbb{A}^3$ via torus localization, the torus-equivariant Siebert's formula for almost perfect obstruction theories and the torus-equivariant Jouanolou trick.
	\end{abstract}
	\maketitle	
	
	\setcounter{tocdepth}{1}
	\tableofcontents
	
	\section{Introduction}
	Let $Y$ be a smooth projective $3$-fold over $\CC$. By the work of Thomas \cite{thomas}, a quasiprojective scheme $M$ parametrizing simple coherent sheaves $\F$ on $Y$ with fixed determinant admits a perfect obstruction theory with pointwise tangent and obstruction spaces 
	\begin{equation}
		\label{kertgobs}
		\Ext^1(\F, \F)_0 \quad \text{and} \quad \Ext^2(\F, \F)_0.
	\end{equation}
	This endows $M$ with a virtual fundamental class $[M]^{\vir} \in A_*(M)$ as constructed in \cite{litian,Behrend_1997}. 
	
	If $M$ is the Hilbert scheme of points $\Hilb_Y^n$ parametrizing ideal sheaves of colength $n$, the perfect obstruction theory is of virtual dimension zero. Therefore, 
	the Donaldson--Thomas invariants in this case are defined as 
	\begin{equation}
		\label{dtinvdef}
		\DTT^n_Y \coloneqq \deg [\Hilb_Y^n]^{\vir} \in \mathbb{Z}.
	\end{equation} 
	The generating series 
	$$\DTT_Y(q) \coloneqq 1 + \sum_{n = 1}^{\infty} q^n \DTT_Y^n$$
	were computed in \cite{LEPA, bfcyinv, junlicompinv} leading to the closed formula
	\begin{equation}
		\label{rkk1}
		\DTT_{Y}(q) = \M(-q)^{\int_Y c_3(T_Y \otimes \omega_Y)},
	\end{equation}
	where 
	$\M(q) = \prod_{n = 1}^{\infty}\frac{1}{(1-q^n)^n} $
	is the MacMahon function.
	
	For a locally free sheaf $F$ on $Y$  consider now the Quot scheme of points $$\Q$$ parametrizing zero-dimensional quotients of $F$ of length $n$. 
	At each point $[S \into F \onto Q]$ its tangent space is $\Hom(S, Q)$, and the standard obstruction space is $\Ext^1(S, Q)$. However, as $\dim \Ext^2(S, Q)$ varies in the higher rank case, the difference in dimensions between standard tangent and obstruction spaces jumps, so that the standard tangent-obstruction theory does not produce a global perfect obstruction theory. 
	Our starting point is the observation that
	there are natural linear maps 
	$$\Ext^1(S, Q) \simto \Ext^2(Q, Q) \into \Ext^3(Q, S) $$
	where the first map is an isomorphism and the second map is an injection. In particular, $\Ext^3(Q, S)$ is also an obstruction space. 
	Since $\dim \Ext^3(Q, S) = \dim \Hom(S, Q)$,
	the difference between the dimensions of the tangent and new obstruction space is constant, suggesting the presence of
	a perfect obstruction theory.
	The first main result in this paper shows that this is ``almost'' the case:
	one always has an almost perfect obstruction theory, a more flexible structure 
	introduced by Kiem and Savvas in \cite{kiemsavvasapot}.

	\begin{theor}(Theorem~\ref{quotsemipot})
		\label{theor1}
		For any smooth projective $3$-fold $Y$ and locally free sheaf $F$,
		the Quot scheme of points $\Q$
		has a canonical almost perfect obstruction theory of virtual dimension zero. 
	\end{theor}
	
	The main idea of the construction is to adjust the standard $3$-term obstruction theory on $\Q$ recalled in Section \ref{derobstr} to make it $2$-term. This adjustment relies on a certain vanishing which holds on an affine scheme and allows one to define perfect obstruction theories on $\Q$ étale locally. 
	The almost perfect obstruction theory on $\Q$ is then defined by 
	confirming the required gluing properties of these local perfect obstruction theories.
	The resulting almost perfect obstruction theory induces the pointwise obstruction spaces $\Ext^3(Q,S)$.
	\vspace{5pt}
	
	Since almost perfect obstruction theories still induces a virtual fundamental class,
	Theorem \ref{theor1} gives rise to virtual invariants of Quot schemes of points defined by
	$$\DT^n_{Y, F} \coloneqq \deg[\Q]^{\vir} \in \mathbb{Z}.$$

	Consider the generating series 
	$$\DT_{Y, F}(q) \coloneqq 1 + \sum_{n = 1}^{\infty} q^n \DT_{Y, F}^n.$$ 
	The second main result of this paper is the following theorem, which proves a conjecture by Ricolfi \cite[Conjecture 3.5]{ricolfivirtclasses} formulated whenever the virtual class on $\Q$ is defined.
	\begin{theor}(Theorem \ref{newgenser})
		\label{theor2}
		Let $Y$ be a smooth projective $3$-fold and $F$ a locally free sheaf of rank $r$ on $Y$. There is an identity
		\begin{equation}
			\DT_{Y, F}(q) = \M((-1)^rq)^{r\int_Y c_3(T_Y \otimes \omega_Y)}.
		\end{equation} 
	\end{theor}
	The proof of Theorem \ref{theor2} naturally splits into two parts. The first part is the case when $Y$ is a toric variety and $F$ a torus-equivariant locally free sheaf, see Theorem~\ref{comptoric}. The second part is the general case.

	The main idea of the proof in the toric case is to reduce the computation to the one on the local Quot scheme of points $\Quot_{\aff}({\affor}, n)$ using virtual torus localization and Siebert's formula for almost perfect obstruction theories. We have to reprove the Kiem--Savvas virtual localization and prove the torus-equivariant Siebert formula under suitable assumptions satisfied in our case. To this end, we exploit a torus-equivariant version of Jouanolou trick.
	We are inspired by \cite[Section 2]{kuhnliuthimm}.
	After reducing the computation to $\Quot_{\aff}({\affor}, n)$ we reach the same expression as in the computation in the toric case with $F$ rigid and simple in \cite{Fasola_2021}.
	
	The proof in the general case exploits the theory of double point cobordism \cite{Lee-Pandharipande-cobordism} extending the proof in the rank one case \cite{LEPA}. The computation boils down to proving the analog of the Li and Wu degeneration formula \cite{li2011good}, using the stacks of expanded degenerations. 
	
	\subsection{Related work} 
	The enumerative geometry of Quot schemes of points has been studied in great detail for lower dimesions. On a curve $C$, the Quot scheme $\Quot_C(\oo^r, n)$ is smooth,
	and the corresponding intersection theory is classical (see \cite{marian2024cohomologyquotschemesmooth} for a description of the cohomology and further references).
	If $X$ is a surface, the Quot scheme of points $\Quot_X(E, n)$ is singular, but due to the vanishing of higher Ext-groups, the standard obstruction theory is perfect and yields a virtual class in degree $rn$, where $r$ is the rank of $E$. In \cite{opquotcurvessurfaces} certain tautological integrals were computed for Quot schemes on curves and surfaces with $E = \oo^r$. In \cite{stark1, stark2} 
	the geometry and the intersection theory of the Quot scheme was considered for 
	an arbitrary locally free sheaf $E$ on $X$.

	This paper extends previous work in dimension $3$.
	If $F$ is a line bundle $L$, then ${\Quot_Y(L, n) \simeq \Hilb_Y^n}$, and we prove in Proposition~\ref{localcomparison} and Proposition~\ref{pglobal} that our construction in Theorem~\ref{theor1} recovers the perfect obstruction theory of Thomas in \cite{thomas}. In particular the invariants $\DT^n_{Y, L}$ recover the rank $1$ Donaldson-Thomas invariants~(\ref{dtinvdef}). 
	In the higher rank case, if $F$ is simple and rigid and either $Y$ is Calabi--Yau or $F$ is exceptional and the higher cohomology of $\oo_Y$ vanish, $\Q$ is the moduli space of the kernels, which are automatically simple, and therefore admits a perfect obstruction theory of virtual dimension zero with tangent and obstruction spaces (\ref{kertgobs}); in particular, the virtual dimension is zero. This was used by Ricolfi to define virtual invariants in this case, see \cite{ricolfivirtclasses}.They agree with $\DT^n_{Y, F}$ by Corollary~\ref{quotvirtclexplicit}.
	Theorem~\ref{theor2} was proven by Ricolfi in \cite{ricolfivirtclasses} when $Y$ is Calabi-Yau and $F$ is simple and rigid, and by Fasola--Monavari--Ricolfi in the toric case when $F$ is exceptional, see \cite[Theorem 9.7]{Fasola_2021}.

	\subsection{Plan of the paper} In Section~$\ref{s2}$ we recall some definitions and results concerning obstruction theories. Section $\ref{locally}$ is devoted to the construction of a perfect obstruction theory of virtual dimension zero on an arbitrary affine étale chart of $\Q$. We also show that if the $3$-fold $Y$ is Calabi--Yau, this perfect obstruction theory is symmetric, and if $\rk F = 1$, it agrees with the one in \cite{thomas}. In Section $\ref{globally}$ we show these local perfect obstruction theories in fact give a global almost perfect obstruction theory and thus prove Theorem~\ref{theor1}. We also show that in rank one case our construction gives a global perfect obstruction theory that agrees with the one constructed by Thomas in \cite{thomas}. Section~$\ref{s5}$ contains various tools to perform the computation in the toric case. We recall the notion of a pullback of an almost perfect obstruction theory under a smooth morphism and the torus-equivariant Jouanolou trick. Moreover, we prove the torus-equivariant Siebert formula and the virtual localization formula for almost perfect obstruction theories under assumptions that are satisfied by $\Q$. In Section~$\ref{s6}$ we check that the results from Section~$\ref{s5}$ can be applied on $\Q$ and finish the computation of the invariants in the toric case. Section~$\ref{s7}$ is devoted to the computation of the invariants in the general case. More precisely, we first review the theory of the double point cobordism and obtain the proof of Theorem~\ref{theor2} 
	granting the analog of the degeneration formula.
	We finish by defining relative virtual invariants for Quot schemes of points via the stack of expanded degenerations and proving the degeneration formula.
	
	\subsection{Conventions and notations} We work over $\mathbb{C}$. Schemes or Deligne--Mumford stacks are assumed to be separated and of finite type over $\CC$. For a Deligne--Mumford stack $X$, we denote by $\mathcal{D}(X) \coloneqq \mathcal{D}^b_{\Coh(X)}(\oo_X)$ the bounded derived category of $\oo_X$-modules on $X$ with coherent cohomology, and by $\mathcal{D}^{[a, b]}(X)$ the derived category of $\oo_X$-modules on $X$ parametrizing complexes of $\oo_X$-modules with coherent cohomology concentrated in $[a, b]$. For a complex $E^\bullet \in \DD(X)$, we shall denote by $E_\bullet$ its derived dual $E^{\bullet \vee}$.
	By $\LL_X \coloneqq \tau^{\geq -1} L_X^\bullet \in \DD^{[-1, 0]}(X)$ we denote the cut-off at $-1$ of the full cotangent complex $L_X^\bullet \in \mathcal{D}^{(-\infty, 0]}(X)$ defined by Illusie \cite{Illusie}.
	For a scheme $X$ acted on by an algebraic group $G$, we denote by $\mathcal{D}^G(X)$ the equivariant bounded derived category of $\oo_X$-modules with coherent cohomology. For a morphism $\pi \colon X \to Y$ between algebraic stacks and $\oo_X$-modules $\mathcal{F}, \mathcal{G}$ we denote by  
	$\rph(\mathcal{F}, \mathcal{G})$ the composite $R\pi_{*} R\sheafhom_X(\mathcal{F}, \mathcal{G})$ and by
	$\extpi{i}(\mathcal{F}, \mathcal{G})$ its $i$-th cohomology $h^i(\rph(\mathcal{F}, \mathcal{G})).$ 
	For a scheme $X$ we denote by $K^0(X)$ the K-group of vector bundles on $X$ and by $K_0(X)$ the K-group of coherent sheaves on $X$. If $X$ is acted on by an algebraic group $G$, we denote by $K^{0, G}(X)$ and by $K^G_0(X)$ the $G$-equivariant analogues.
	
	\subsection{Acknowledgements}
	I would like to thank my advisors Barbara Fantechi and Andrea Ricolfi for their support and guidance.
	%
	I am very grateful to Richard Thomas for sharing his insights which led to the construction of the perfect obstruction theory on the Quot scheme étale locally. I thank David Rydh for very useful discussions, in particular about the equivariant Jouanolou trick. 
	%
	%
	%
	
	\section{Background material} \label{s2}
	
	\subsection{Obstruction theories} We recall some definitions and results about obstruction theories following \cite{Behrend_1997, kiem}.
	\begin{definition}
		Let $X$ be a 
		Deligne--Mumford stack
		and $E^\bullet \in \mathcal{D}^{\leq 0}(X)$.
		A morphism ${\varphi \colon E^\bullet \to \LL_X}$ in $\mathcal{D}(X)$ to the truncated cotangent complex of $X$ is called an obstruction theory if $h^0(\varphi)$ is an isomorphism and  $h^{-1}(\varphi)$ is surjective.
	%
		\\If the object $E^\bullet \in \mathcal{D}(X)$ is perfect of perfect amplitude $[-1, 0]$ then the obstruction theory $\varphi \colon E^\bullet \to \LL_X$ is called perfect (POT, for short). 
	\end{definition}
	
	\begin{definition}
		A perfect obstruction theory $\varphi \colon E^{\bullet} \to \mathbb{L}_X$ 
		is called symmetric if there exists an isomorphism $\theta \colon E^{\bullet \vee}[1] \simto  E^{\bullet}$ in $\mathcal{D}(X)$ such that $\theta^{\vee}[1] = \theta$.
	\end{definition}	
	
	A perfect obstruction theory ${\varphi \colon E^\bullet \to \LL_X}$ defines a virtual fundamental class in the Chow group
	$$[X]^{\vir} \in A_{\vd}(X), $$
	where $\vd \coloneqq \rk(E^\bullet)$ is called the virtual dimension, and a virtual structure sheaf in K-theory
	$$\oo^{\vir}_X \in K_0(X).$$ 
	
	\begin{definition}
		Given a perfect obstruction theory ${\varphi \colon E^\bullet \to \LL_X}$, the sheaf $$Ob_X \coloneqq h^1(E^{\bullet \vee}) $$
		is called the obstruction sheaf.
	\end{definition}
	We briefly recall the construction of $[X]^{\vir}$ given a perfect obstruction theory ${\varphi \colon E^\bullet \to \LL_X}$. For any Deligne--Mumford stack $X$, Behrend--Fantechi construct a cone stack $\mathcal{C}_X$ which locally for an étale map $U \to X$ and a closed embedding $U \into M$ into a smooth scheme $M$ is the quotient stack $[C_{U/M}/T_M|_U].$ It is also shown that it admits an embedding $\mathcal{C}_X \subset \mathcal{E}_X $
	into a vector bundle stack constructed from $E^\bullet$.
	%
	The virtual fundamental class is then defined as 
	$$[X]^{\vir} \coloneqq 0^!_{\mathcal{E}_X}[\mathcal{C}_X],$$
	where $0^!_{\mathcal{E}_X}$ is the Gysin pullback corresponding to the zero section of the vector bundle stack \cite{kresch}.
	If the complex $E_\bullet \coloneqq E^{\bullet \vee}$ admits a locally free resolution $[E_0 \to E_1]$ by locally free sheaves, then $\mathcal{E}_X \simeq [E_1/ E_0]$, and putting $D^{\vir} \coloneqq \mathcal{C}_X \times_{\mathcal{E}_X} E_1$ yields
	$$[X]^{\vir} = 0^!_{E_1}[D^{\vir}].$$
	There is a fiber diagram
	\begin{center}
		\begin{equation}
			\label{diagrvirtclass}
			\begin{tikzcd}
				D^{\vir} \arrow[r, hook] \arrow[d] & E_1 \arrow[d]           \\
				\mathcal{C}_X \arrow[r, hook] \arrow[d]     & \mathcal{E}_X \arrow[d] \\
				c_X \arrow[r, hook]                & Ob_X,                   
			\end{tikzcd}
		\end{equation}
	\end{center}
	where $c_X$ is the coarse moduli sheaf of $\mathcal{C}_X$ and the embedding $c_X \subset Ob_X$ is viewed as an inclusion of sheaf stacks \cite{kiemsavvasapot}. Note that the obstruction sheaf $Ob_X$ is the coarse moduli sheaf of $\mathcal{E}_X$. We make use of the lower row in diagram~(\ref{diagrvirtclass}) in the next subsection when defining the virtual class corresponding to an almost perfect obstruction theory.  
	
	The following is a well-known result about obstruction theories.
	\begin{prop}
		\label{criterion}
		Let $X$ be a scheme.
		Consider a morphism $\varphi \colon E^\bullet \to \LL_X$ in $\mathcal{D}^{(- \infty, 0]}(X)$. 
		The following are equivalent:
		\begin{enumerate}
			\item $\varphi \colon E^\bullet \to \LL_X$ is an obstruction theory. 
			\item For every closed point $x \in X$, the induced maps on cohomology of $\varphi^\vee$ restricted to $x$ satisfies the following: the map
			$$h^0((\varphi|_x)^\vee) \colon h^0((\LL_X|_x)^\vee) \xrightarrow{\simeq} h^0((E^\bullet|_x)^{\vee})  $$ is an isomorphism and the map
			$$h^1((\varphi|_x)^\vee) \colon h^1((\LL_X|_x)^\vee) \into h^1((E^\bullet|_x)^{\vee}) $$ is an injection.
		\end{enumerate}
	\end{prop}
	
	\begin{definition}
		Let $X$ be a Deligne--Mumford stack acted on by an algebraic group $G$. A (perfect) obstruction theory $\varphi$ is called $G$-equivariant if the morphism $\varphi$ is in the equivariant derived category $\DD^G(X)$.
	\end{definition}
	For details on group actions on stacks we refer to \cite{romagny}.
	
	As before, a $G$-equivariant perfect obstruction theory on $X$ gives rise to an equivariant virtual class and an equivariant virtual structure sheaf
	$$[X]^{\vir} \in A_{\vd}^G(X) \quad \text{and} \quad \oo^{\vir}_X \in K^G_0(X).$$

	\subsection{Almost perfect obstruction theories.} We recall the notion of an almost perfect obstruction theory by Kiem--Savvas and the construction of the corresponding virtual class following \cite{semipot, kiem2020localizing, kiemsavvasapot}.
	\begin{definition}
		\label{APOT}
		Let $X$ be a Deligne-Mumford stack. An almost perfect obstruction theory (APOT, for short) on $X$ is the datum of
		\begin{enumerate}
			\item an étale cover $\normalfont{\{U_i \to X\}_i}$,
			\item a perfect obstruction theory $\normalfont{\varphi_i \colon E_i^\bullet \to \LL_{U_i}}$  on each $U_i$, 
		\end{enumerate}
		such that the following conditions are satisfied:
		\begin{enumerate}[(i)]
			\item There exists a sheaf $Ob_X \in \Coh(X)$, called the (global) obstruction sheaf, such that for each $i$ there is an isomorphism  $\psi_i \colon  Ob_{U_i} \simto Ob_X|_{U_i}.$ 
			\item For each $i, j$ there is an étale chart $V \to U_i \times_X U_j$ and isomorphisms $\eta_{ij} \colon E_i^\bullet|_V \simto E_j^\bullet|_V$ such that they are compatible with the maps to the cotangent complex and ${h^1(\eta_{ij}^\vee) = \psi_i^{-1}|_V \circ \psi_j|_V.}$
		\end{enumerate}
	\end{definition}
	
	Kiem and Savvas \cite{kiemsavvasapot} show that the inclusions $c_{U_i} \subset Ob_{U_i}$ of coarse moduli sheaves into the obstruction sheaves on each étale affine chart of the APOT glue to a global inclusion $c_X \subset Ob_X$. Viewing $c_X \subset Ob_X$ as an embedding of sheaf stacks \cite{kiemsavvasapot} induces a cycle class $[c_X] \in A_*(Ob_X)$, defined by Chang--Li \cite{semipot}. The virtual class of the almost perfect obstruction theory is then defined in \cite{semipot} as 
	$$[X]^{\vir} \coloneqq 0^!_{Ob_X}[c_X] \in A_{\vd}(X),$$
	where $0^!_{Ob_X} \colon A_*(Ob_X) \to A_*(X)$ is the Gysin map of Chang--Li initially constructed in \cite{gysinmapcohsheaves}. Above, vd denotes the virtual dimension of the APOT. 
	It is equal to the virtual dimension of each perfect obstruction theory~$\varphi_i$, as the restriction of the virtual class $[X]^{\vir}$ to each chart $U_i \to X$ recovers the virtual class $[U_i]^{\vir}$. 
	
	By \cite{kiemsavvasapot}, an APOT also gives rise to a virtual structure sheaf 	$$\oo_X^{\vir} \in K_0(X)$$ 
	such that its restriction to each étale chart $U_i \to X$ of the APOT is $\oo_{U_i}^{\vir}$.

	\begin{rem}
		\label{glkthcl}
		Note that the sheaves $h^0(E_{i \bullet}) \simeq T_{U_i}$ are glued to $T_X$ and hence the class 
		$$[T_X] - [Ob_X] \in K_0(X) $$
		plays the role of a global class in K-theory of the almost perfect obstruction theory.		
	\end{rem}
	\begin{convention}
		\label{K^0class}
		We will refer to every preimage of $[T_X] - [Ob_X] \in K_0(X)$ under the natural map $K^0(X) \to K_0(X)$ (which might not be surjective or injective) as to a global class in $K^0(X)$ of the APOT on $X$. We will also use equivariant analogues.
	\end{convention}
	
	%
	We will also need the notion of an equivariant almost perfect obstruction theory for schemes, see \cite[Definition 5.1]{kiem2020localizing}.
	\begin{definition}
		Let $X$ be a Deligne--Mumford stack acted on by an algebraic group $G$. An almost perfect obstruction theory $\varphi$ is called $G$-equivariant if the étale cover $\{U_i \to X\}_i$ is $G$-equivariant, the morphisms $\varphi_i$ are in $\DD^G(U_i)$,
		the isomorphisms $\psi_i$ are in $\Coh^G(U_i)$, the étale chart $V \to U_i \times_X U_j$ is $G$-equivariant
		and the isomorphisms $\eta_{ij}$ together with the compatibilities are in $\DD^G(V)$.
	\end{definition}
	A $G$-equivariant almost perfect obstruction theory on $X$ gives rise to an equivariant virtual class and an equivariant virtual structure sheaf:
	$$[X]^{\vir} \in A^G_{\vd}(X) \quad \text{and} \quad \oo_X^{\vir} \in K_0^G(X).$$
	Here vd is the virtual dimension of the APOT which coincides with the virtual dimension of each perfect obstruction theory $\varphi_i$, as in the non-equivariant case.
	
	\subsection{Equivalence of almost perfect obstruction theories}
	We need to make explicit the dependance of an APOT on the choice of étale cover.
	
	\begin{lem}
		\label{lemequivapot}
		Let $X$ be a Deligne--Mumford stack with an almost perfect obstruction theory 
		$$\varphi = \{U_i \to X, \varphi_i \colon E_i^\bullet \to \LL_X\}_{i \in I}. $$
		Consider $I^\prime \subset I$ such that $\{U_i \to X\}_{i \in I^\prime}$ is an étale cover of $X$. 
		Then the APOT $\varphi^\prime$ corresponding to the cover $\{U_i \to X\}_{i \in I^\prime}$ has the same virtual class, virtual structure sheaf and the class in $K$-theory from Remark~\ref{glkthcl}. We say that $\varphi$ refines $\varphi^\prime$.
	\end{lem}
	\begin{proof}
		This follows from the fact that the obstruction sheaf $Ob_X$ for the APOT $\varphi$ and the embedding $c_X \into Ob_X$ of the coarse moduli sheaf are already determined from the data on the cover $\{U_i \to X\}_{i \in I^\prime}$.
	\end{proof}
	
	\begin{definition}
		\label{sameAPOT}
		Two APOTs are equivalent if there exists an APOT refining both APOTs.
	\end{definition}
	
	\subsection{Standard Obstruction Theory  on \texorpdfstring{$\Q$}.} \label{derobstr}
	Let $Y$ be a smooth projective $3$-fold and $F$ a locally free sheaf on it. Fix an integer $n \in \mathbb{N}$. Consider the Grothendieck Quot scheme of points $\Q$ parametrizing isomorphism classes of quotients $[F \onto Q]$, with $Q$ having zero-dimensional support of length $n$. Let
	\begin{equation*}
		0 \to \sss \to q^*F \to \qq \to 0
	\end{equation*}
	be the universal sequence on $Y \times \Q$ with the two projections $\pi \colon Y \times \Q \to \Q$ and ${q \colon Y \times \Q \to Y}$. By \cite[Theorem 4.6]{gillam} there exists a canonical obstruction theory on $\Q$
	\begin{equation}
		\label{explderobstr}
		\varphi^{\der} \colon (\rpl \rh(\sss, \qq))^\vee \to \LL_\Q,
	\end{equation}
	corresponding to the reduced Atiyah class
	$$A(\qq/\sss) \in \Ext^0_{Y \times \Q}(\sss, \pi^*\LL_\Q \otimes \qq) $$
	under the chain of isomorphisms
	\begin{align*}
		\Ext^0_{Y \times \Q}(\sss, &\pi^*\LL_\Q \otimes \qq) \simeq 
		\Ext^0_{Y \times \Q}(\qq^{\vee} \otimes \sss, \pi^*\LL_\Q)	\\
		&\simeq \Ext^0_{Y \times \Q}(R\sheafhom(\qq, \sss), \pi^*\LL_\Q) \\
		&\simeq \Ext^0_{\Q}(\rpl (\rh(\qq, \sss) \otimes \omega_\pi)[3], \LL_\Q) \\
		&\simeq \Ext^0_{\Q}((\rpl \rh(\sss, \qq))^{\vee}, \LL_\Q) \\
		&= \Hom_{\DD(\Q)}((\rpl \rh (\sss, \qq))^{\vee}, \LL_\Q).
	\end{align*}
	The third and fourth isomorphisms are by Grothendieck duality (cf.~\cite[Theorem 3.34]{huybrechts2006fourier}) applied to the smooth proper $3$-dimensional morphism $\pi$.
	
	\begin{rem}
		The notation $\varphi^{\der}$ for the standard obstruction theory reflects the fact that it comes from derived algebraic geometry, see \cite{monavaripaviaricolfi2024derivedhyperquotschemes} and the references therein.
	\end{rem}
	
	\subsection{Perfect obstruction theory in the local case} \label{crit} Consider the ``local Quot scheme'' $$Q^r_n \coloneqq \Quot_{\mathbb{A}^3}(\mathcal{O}_{\mathbb{A}^3}^{r}, n)$$ of points on ${\mathbb{A}^3}$, 
	parametrizing isomorphism classes of quotients of the form $\oo^r_{\mathbb{A}^3} \onto T$ with $T$ zero-dimensional and $\chi(T) = n$.
	We briefly explain how a symmetric perfect obstruction theory on this Quot scheme
	is obtained. For more details see \cite{BR18, Fasola_2021}.
	\begin{prop}[{\cite[Theorem 2.6]{BR18}}]
		\label{quotcrit}
		The local $\Quot$ scheme $Q^r_n$
		admits a critical structure, i.$\,$e.$\,$it is isomorphic to the zero scheme $Z(df^r_n)$ for a regular function ${f^r_n \colon \Quot^r_n \to \mathbb{A}^1}$ on a smooth scheme (the non-commutative $\Quot$ scheme). 
	\end{prop}
	\begin{cor}
		\label{spot}
		The local Quot scheme $Q_n^r$ carries a symmetric perfect obstruction theory $E^{crit}_{r,n}$. We will call $E^{crit}_{r,n}$ the \emph{critical} perfect obstruction theory.
	\end{cor}

	\section{Perfect obstruction theory locally}	\label{locally}
	Let $Y$ be a smooth projective $3$-fold and $F$ a locally free sheaf of rank $r$ on it. For a fixed integer $n \in \mathbb{N}$ consider its Grothendieck Quot scheme 
	of points
	$$\Q.$$

	Let $U = \Spec R \xrightarrow{g} \Q$ be an étale morphism from an affine scheme. The aim of this section is to construct a perfect obstruction theory on $U$ of virtual dimension zero.
	Consider, on $Y \times U$, the pullback of the universal sequence of $\Q$ from $Y \times \Q$ via $(\id_Y \times g)$, and denote it by
	\begin{equation}
		\label{uunseq}
		0 \to \sss \to q^*F \to \qq \to 0.
	\end{equation}
	Here, ${q \colon Y \times U \to Y}$ and ${\pi \colon Y \times U \to U}$ are the two projections.
	
	\begin{lem}
		\label{perfect}
		The sheaves $\sss$ and $\qq$ are perfect objects in the derived category $\DD(Y \times U)$. Moreover, for $\mathcal{A}, \mathcal{B} \in \{\sss, q^*F, \qq\}$ the complexes $\rph(\mathcal{A}, \mathcal{B})$ 
		are perfect.
	\end{lem}
	\begin{proof}
		By Lemma $\ref{flatperfect}$, the sheaves $\sss$ and $\qq$ are perfect objects as they are flat over $U$. The second statement follows by \cite[\href{https://stacks.math.columbia.edu/tag/0DJR}{Tag 0DJR}]{stacks-project}.
	\end{proof}
	
	\begin{lem}
		\label{fqqf}
		For each closed point $x = [S \into F \onto Q] \in \Q$ the following holds:
		\begin{enumerate}
			\item $\Ext^i(F, Q) = 0$ for $i>0;$
			\item $\Ext^i(Q, F) = 0$ for $i<3$.
		\end{enumerate} 
		\begin{proof}
			We have ${\Ext^i(F, Q) \simeq \HH^i(Y, F^\vee \otimes Q) = 0}$ for $i>0$ as $Q$ has zero-dimensional support. 	
			Similarly, ${\Ext^i(Q, F) \simeq \Ext^{3-i}(F, Q)^\vee = 0}$ for $3-i >0$ using Serre Duality.
		\end{proof}
	\end{lem}
	\begin{lem}
		\label{locextfqqf}
		We have the following vanishings:
		\begin{enumerate}
			\item 	$\extpi{i}(q^*F, \qq) = 0 \quad \text{for} \quad i >0;$
			\item   $\extpi{i}(\qq, q^*F) = 0 \quad \text{for} \quad i <3.$
		\end{enumerate}
	\end{lem}
	\begin{proof}
		Consider a closed point  $x \in U$ mapping to $[S \into F \onto Q] \in \Q$. By \cite[\href{https://stacks.math.columbia.edu/tag/0A1D}{Tag 0A1D}]{stacks-project}, we have isomorphisms
		$$\rph(q^*F, \qq)|_x \simeq \RHom(F, Q) \quad \text{and} \quad 
		\rph(\qq, q^*F)|_x \simeq \RHom(Q, F).$$
		Hence by Lemma $\ref{huybrthom}$, Lemma $\ref{perfect}$ and Lemma $\ref{fqqf}$ the complexes $\rph(q^*F, \qq)$ and $\rph(\qq, q^*F)$ are vector bundles concentrated in degrees $0$ and $3$, respectively. Hence, the desired vanishing follows.
	\end{proof}
	Consider the diagram in the derived category $\DD(U)$
	\begin{center}
		\begin{equation}
			\label{diagr}
			\begin{tikzcd}
				& {\rph(\qq, q^*F)[1]} \arrow[d, "\xi_1"] \arrow[ld, "\exists! p"', dotted] &                      \\
				{\rph(\sss, \qq)} \arrow[r, "\xi_2"] & {\rph(\qq, \qq)[1]} \arrow[r, "\xi_3"]                           & {\rph(q^*F, \qq)[1],}
			\end{tikzcd}
		\end{equation}
	\end{center}
	where the maps $\xi_i$ are induced by the universal sequence $(\ref{uunseq})$ so that the lower row is a distinguished triangle. 
	\begin{lem}
		\label{vanishing}
		The composition $\xi_3 \circ \xi_1$ is zero.
	\end{lem}	
	\begin{proof}
		By Lemma~\ref{locextfqqf}, the perfect complexes $\rph(q^*F, \qq)$ and $\rph(\qq, q^*F)$ are concentrated in degrees $0$ and $3$, respectively. Moreover, they are vector bundles
		$$\rph(q^*F, \qq) \simeq \extpi{0}(q^*F, \qq) \quad \text{and} \quad 
		\rph(\qq, q^*F) \simeq \extpi{3}(\qq, q^*F)[-3].$$
		Now
			\begin{align*}
				\Hom_{\DD(U)}&(\rph(\qq, q^*F)[1], \rph(q^*F, \qq)[1]) \\
				& \simeq	\Hom_{\DD(U)}(\extpi{3}(\qq, q^*F)[-3], \extpi{0}(q^*F, \qq)) \\
				&\simeq	\Ext^3_U(\extpi{3}(\qq, q^*F), \extpi{0}(q^*F, \qq)) \\
				&\simeq	\HH^3(U, (\extpi{3}(\qq, q^*F))^\vee \otimes \extpi{0}(q^*F, \qq)) = 0.
			\end{align*}
		The last vanishing holds as $U$ was chosen to be affine.
		Hence, as there are no morphisms between the complexes of interest, the morphism $\xi_3 \circ \xi_1$ also vanishes.
	\end{proof}
	
	By Lemma $\ref{vanishing}$ together with Lemma $\ref{exun}$ there exist a morphism 
	\begin{equation}
		\label{pconstr}
		p \colon \rph(\qq, q^*F)[1] \to \rph(\sss, \qq)
	\end{equation}
	making the diagram $\ref{diagr}$ commute.
	By the same reasoning as in the proof of Lemma~$\ref{vanishing}$ one has the vanishing 
	$$\Hom_{\DD(U)}(\rph(\qq, q^*F)[1], \rph(q^*F, \qq)) = 0$$ and hence by Lemma $\ref{exun}$ the morphism $p$  making the diagram $(\ref{diagr})$ commute is unique.
	
	Put $$E_\bullet \coloneqq \Cone(p).$$ We hence have a distinguished triangle
	\begin{equation}
		\label{thetr}
		\rph(\qq, q^*F)[1] \xrightarrow{p} \rph(\sss, \qq) \xrightarrow{\varepsilon} E_\bullet \xrightarrow{\rho} \rph(\qq, q^*F)[2].
	\end{equation}
	Put $$E^\bullet \coloneqq (E_\bullet)^\vee.$$
	We next define the map 
	\begin{equation}
		\label{defmorph}
		{\varphi \colon E^\bullet \to \LL_U}
	\end{equation}
	to the (truncated) cotangent complex as the composition
	
	$$\varphi \colon E^\bullet \xrightarrow{\varepsilon^\vee} (\rph(\sss, \qq))^\vee \xrightarrow{\varphi^{\der}|_U} \LL_U, $$
	where the second map comes from the obstruction theory on $\Q$ discussed in Section $\ref{derobstr}$.
	
	\begin{prop}
		\label{potlocally}
		The map $\varphi \colon E^\bullet \to \LL_U$ defines a perfect obstruction theory on $U$ of virtual dimension zero.
	\end{prop}
	\begin{proof} Note that $E_\bullet$ is perfect being the Cone of a morphism between perfect objects by Lemma~$\ref{perfect}$. Hence $E^\bullet$ is also perfect.
		It remains to prove that $E_\bullet$ is of perfect amplitude $[0, 1]$ and that $\varphi$ defines an obstruction theory. We check both statements by restricting to a closed point using Proposition $\ref{criterion}$ and Lemma $\ref{huybrthom}$.

		Consider a closed point $x \in U$ mapping to $[S \into F \onto Q] \in \Q$ and set ${E \coloneqq E_\bullet|_x}$. The (derived) restriction of the distinguished triangle $(\ref{thetr})$ to $x \in U$ yields the distinguished triangle 
		$$R\Hom(Q, F)[1] \xrightarrow{p|_x} R\Hom(S, Q) \xrightarrow{\varepsilon|_x} E $$ in $\DD(\pt)$.
		Consider the corresponding long exact sequence in cohomology
		\begin{flalign*}
			0 \to \Hom(S, Q) \xrightarrow{h^0(\varepsilon|_x)} \HH^0(E) \to 0  \to \Ext^1(S, Q) \xhookrightarrow{h^1(\varepsilon|_x)} \HH^1(E) \to \\
			\Ext^3(Q, F) \xrightarrow{h^2(p|_x)} \Ext^2(S, Q) \to \HH^2(E) \to 0 \to 
			\Ext^3(S, Q) \to \HH^3(E) \to 0.
		\end{flalign*}
		We have used the vanishings from Lemma $\ref{fqqf}$.
		Note that $\Ext^3(S, Q)$ sits in an exact sequence
		$$\Ext^3(F, Q) \to \Ext^3(S, Q) \to \Ext^4(Q, Q) $$
		where the first term is zero by Lemma $\ref{fqqf}$ and the last term is zero as we are on a smooth $3$-fold $Y$. Therefore $\Ext^3(S, Q) = 0$. Hence also $\HH^3(E) = 0$. 
		
		We next show that $h^2(p|_x)$ is surjective. For this consider the restriction to the closed point $x \in U$ of the commutative triangle in diagram $(\ref{diagr})$. We hence see that $h^2(p|_x)$ is the composition of 
		$${h^2(\xi_1|_x) \colon \Ext^3(Q, F) \onto \Ext^3(Q, Q)}$$
		and the inverse of 
		$$h^2(\xi_2|_x) \colon \Ext^2(S, Q) \xrightarrow{\sim} \Ext^3(Q, Q). $$
		The latter is indeed an isomorphism as it sits in the exact sequence
		$$\Ext^2(F, Q) \to \Ext^2(S, Q) \to \Ext^3(Q, Q) \to \Ext^3(F, Q) $$ 
		where the first and the last terms vanish by Lemma $\ref{fqqf}$.
		Note that $h^2(\xi_1|_x)$ is surjective as it sits in the exact sequence
		$$\Ext^3(Q, F) \to \Ext^3(Q, Q) \to \Ext^4(Q, S) = 0. $$
		Therefore $h^2(p|_x)$ is also surjective. Hence $\HH^2(E)=0$. 
		
		Summarizing, we have shown that the complex $E$ is concentrated in degrees $0$ and $1$. Moreover, $h^0(\varepsilon|_x)$ is an isomorphism and $h^1(\varepsilon|_x)$ is injective. Hence by Lemma $\ref{huybrthom}$ the complex $E$ is of perfect amplitude $[0,1]$ and by Proposition $\ref{criterion}$ the morphism $\varphi$ defines an obstruction theory on $U$. 
		
		We next show it has virtual dimension zero. Consider a closed point $${[S \into F \into Q] \in \Q}.$$ Then, by (\ref{thetr}), the virtual dimension is
		\begin{align*}
			\chi(S, Q) + \chi(Q, F) &= 
			\chi(S, F) - \chi(S, S) + \chi(F, F) - \chi(S, F)  \\
			&=\chi(F, F) - \chi(S, S) = 0.		
		\end{align*}
		The last equality follows from the fact that $Q$ has zero-dimensional support and can be computed directly via Hirzebruch--Riemann--Roch. 
	\end{proof}
	
	Consider the following commutative diagram with rows and columns being distinguished triangles
	\begin{center}
		\begin{equation} \label{twodisttr}
			\begin{tikzcd}
				{\rph(\qq, q^*F)[1]} \arrow[d, equals] \arrow[r, "{p}"] & {\rph(\sss, \qq)} \arrow[d, "{\xi_2}"] \arrow[r, "{\varepsilon}"] & {E_\bullet} \arrow[d, "\gamma", dashed] \\
				{\rph(\qq, q^*F)[1]} \arrow[r, "{\xi_1}"]           & {\rph(\qq, \qq)[1]} \arrow[r, "{\xi_4[1]}"]  \arrow[d, "{\xi_3}"]              & {\rph(\qq, \sss)[2]} \arrow[d, dashed, "z"] \\
				& {\rph(q^*F, \qq)[1]}  \arrow[r, equals] & {\rph(q^*F, \qq)[1],}  
			\end{tikzcd}
		\end{equation}
	\end{center}
	where the upper triangle is $(\ref{thetr})$ and other distinguished triangles are induced by the universal sequence $(\ref{uunseq})$.
	By the octahedral axiom, there exist dashed arrows $\gamma$ and $z$ such that the right vertical column is also a distinguished triangle and the whole diagram commutes. 
	In fact, since $\xi_3 \circ \xi_1 = 0$ (Lemma~\ref{vanishing}), by an analog of Lemma~\ref{exun}, there is a unique morphism $z$ such that the lower square in diagram~(\ref{twodisttr}) is commutative.
	
	We next determine the obstruction sheaf $Ob_U = h^1(E_\bullet)$.
	\begin{prop} \label{obstrsheaf}
		There is an isomorphism
		$$Ob_U \simto \extpi{3} (\qq, \sss).$$
	\end{prop}
	\begin{proof}
		The long exact sequence in cohomology applied to the right vertical distinguished triangle in diagram~(\ref{twodisttr}) together with the vanishing $(2)$ in Lemma~\ref{locextfqqf} yields the desired isomorphism
		\begin{equation}
			h^1(\gamma) \colon h^1(E_\bullet) \simto \extpi{3}(\qq, \sss).  \qedhere
		\end{equation}
	\end{proof}			
	We next show a compatibility between choices of 
	the Cone in $(\ref{thetr})$ together with maps $\varepsilon, \rho$  and $\gamma$ in diagram~(\ref{twodisttr}) that will be used in the next section to show that the local construction glue to an almost perfect obstruction theory. The proof is close to that of \cite[Lemma~2.3.7]{kuhnliuthimm}.
	
	\begin{prop}
		\label{conechoices}
		The choice of $E_\bullet = \Cone(p)$, of morphisms $\varepsilon, \rho$ in $(\ref{thetr})$ and of morphism $\gamma$ in diagram (\ref{twodisttr}) is unique up to a (not necessarily unique) isomorphism, i.$\,$e.$\,$if
		$\tilde{E}_\bullet$,
		$\tilde{\varepsilon}, \tilde{\rho}, \tilde{\gamma}$ is a different choice, then there exists an isomorphism
		$$\eta \colon \tilde{E}_\bullet \simto E_\bullet $$
		such that $\eta \circ \tilde{\varepsilon} = \varepsilon$, $\rho \circ \eta = \tilde{\rho} $ and $\gamma \circ \eta = \tilde{\gamma}$.
	\end{prop}
	\begin{proof}
		As a Cone of a morphism is defined up to an isomorphism, we can without loss of generality assume that $\tilde{E}_\bullet = E_\bullet$, $\varepsilon = \tilde{\varepsilon}$ and $\rho = \tilde{\rho}$. We need to find an automorphism $\eta$ such that the diagram
		\begin{center}
			\begin{tikzcd} [row sep = 8]
				& E_\bullet \arrow[rd, "\rho"] \arrow[dd, "\eta", "\simeq"'] &                      \\
				{\rph(\sss, \qq)} \arrow[ru, "\varepsilon"] \arrow[rd, "\varepsilon"'] &                                                 & {\rph(\qq, q^*F)[2]} \\
				& E_\bullet \arrow[ru, "\rho"']                   &                     
			\end{tikzcd}
		\end{center}
		is commutative and $\gamma \circ \eta = \tilde{\gamma}$. Consider now the top two rows of diagram $(\ref{twodisttr})$ for the morphisms $\gamma$ and $\tilde{\gamma}$.
		\begin{center}
			\begin{equation}
				\label{adjustchoices}
				\begin{tikzcd}
					{\rph(\sss, \qq)} \arrow[d] \arrow[r, "\varepsilon"] & E_\bullet \arrow[d, "\gamma"', shift right] \arrow[d, "\tilde{\gamma}", shift left] \arrow[r, "\rho"] & {\rph(\qq, q^*F)[2]} \arrow[d, equals] \arrow[ld, "u_0", dashed] \arrow[l, "u_1"', bend right, dashed] \\
					{\rph(\qq, \qq)[1]} \arrow[r]                        & {\rph(\qq, \sss)[2]} \arrow[r, " {\xi_7}"']                                                                        & {\rph(\qq, q^*F)[2]}                               
				\end{tikzcd}
			\end{equation}
		\end{center}
		By construction, all squares in diagram~(\ref{adjustchoices}) are commutative. Therefore putting $u \coloneqq \gt - \ga$ gives $u \circ \varepsilon = 0$. Hence, there exists a map $u_0$ such that $u_0 \circ \rho = u$. Recall from diagram~(\ref{twodisttr}) that $\Cone(\ga) = \rph(q^*F, \qq)[1]$ is a vector bundle in degree $-1$. As $\rph(\qq, q^*F)[2]$ is a vector bundle in degree $1$ and we are working on an affine $U$, there is the vanishing
		$$\Hom_{\DD(U)}(\rph(\qq, q^*F)[2], \rph(q^*F, \qq)[d]) = 0$$
		for $d = 0, 1$.
		Hence, applying $\Hom_{\DD(U)}(\rph(\qq, q^*F)[2], -)$ to the distinguished triangle corresponding to the map $\gamma$ yields that post-composing with $\ga$ induces an isomorphism
		$$\gamma \circ \colon \Hom_{\DD(U)}(\rph(\qq, q^*F)[2], E_\bullet) \simto \Hom_{\DD(U)}(\rph(\qq, q^*F)[2], \rph(\qq, \sss)[2]).$$
		Therefore, $u_0$ factors as $u_0 = \ga \circ u_1$ for a unique $u_1$. Put $\eta \coloneqq \id + u_1 \circ \rho$. We now check that $\eta$ satisfies all required properties.
		By construction, $\ga \circ \eta = \ga + u = \gt.$ This also implies that $h^1(\eta)$ is an isomorphism. Together with $h^0(\eta) = h^0(\id) = \id$ it implies that the map $\eta$ is an automorphism. 
		We have $\eta \circ \varepsilon = \varepsilon + u_1 \circ \rho \circ \varepsilon = \varepsilon$. Finally, $\rho \circ \eta = \rho + \rho \circ u_1 \circ \rho = \rho$ because
		$\rho \circ u_1 \circ \rho = \xi_7 \circ \ga \circ u_1 \circ \rho = \xi_7 \circ u_0 \circ \rho = \xi_7 \circ u = 0$
	\end{proof}
	
\subsection{Calabi--Yau case}
		\begin{prop}
		\label{cyafflocal}
		If the $3$-fold $Y$ is Calabi-Yau, then the perfect obstruction theory constructed in Proposition $\ref{potlocally}$ is symmetric.
	\end{prop}
	\begin{proof}
		We first note that as $Y$ is Calabi--Yau and hence $\omega_\pi = q^* \omega_Y$ is trivial, Grothendieck duality with respect to the smooth $3$-dimensional morphism $\pi$ gives
		\begin{equation}
			\label{qfgd}
			(\rph(q^*F, \qq)[1])^\vee [-1] \simeq \rph(\qq, q^*F \otimes \op[2])[-1] \simeq \rph(\qq, q^*F)[1].
		\end{equation}
		Analogously, applying $R\sheafhom(-, \oo)$ and then shift by $[-1]$ to all entries and morphisms in diagram~\ref{twodisttr}, by functoriality of Grothendieck duality we obtain the same diagram.
		In particular,
		$$p^\vee[-1] = z \quad \text{and} \quad z^\vee[-1] = p. $$ Hence,
		$E^\vee_\bullet[-1] = E_\bullet$, and if we denote this equality by $\theta$, we have the canonical identification $\theta^\vee[-1] = \theta.$
	\end{proof}
\subsection{Rank one case.}
	In this subsection we show that in the rank one case the perfect obstruction theory $(\ref{defmorph})$ agrees with the restriction of the perfect obstruction theory on the Hilbert scheme. 
	If the locally free sheaf $F$ is of rank one, then $\Q \simeq \Hilb^n_Y$ admits a perfect obstruction theory \cite{thomas}. Its pullback via the étale map $U \to \Q$ with $U$ affine is 
	\begin{equation}
		\label{pothilb}
		(\rph(\sss, \sss)_0[1])^{\vee} \to \LL_U,
	\end{equation}
	where $\rph(\sss, \sss)_0$ is the pushforward via $\pi$ of $\rh(\sss, \sss)_0 \coloneqq \Cone(tr)[-1]$, which is the shifted Cone of the trace map $tr \colon \rh(\sss, \sss) \to \oo_{Y \times U}$ \cite{huybrechts_lehn}. 
	The trace map has a natural splitting 
	\begin{equation}
		\label{id}
		\id \colon  \oo_{Y \times U} \to \rh(\sss, \sss),
	\end{equation}
	and $\rh(\sss, \sss)_0$ is naturally isomorphic to $\Cone(\id)$. We will abuse notation, and denoted by $\id$ also its pushforward via $\pi$.
	
	\begin{lem}
		\label{dualext}
		Let $\pi \colon X \to Y$ be a smooth proper morphism of relative dimension $d$. Let $\F$ and $\G$ be coherent sheaves on $X$. Then 
		$$\sheafhom(\extpi{d}(\F, \G \otimes \op), \oo_Y) \simeq \extpi{0}(\G, \F).$$
	\end{lem}
	\begin{proof}
		By Grothendieck duality applied to the smooth proper morphism $\pi$ of relative dimension $d$,
		\begin{equation}
			\label{grgu}
			R\sheafhom(\rph(\F, \G \otimes \op), \oo_Y[-d]) \simeq \rph(\G, \F).
		\end{equation}
		Applying $h^0(-)$ to the right hand side yields $\extpi{0}(\G, \F)$. Applying $h^0(-)$ to the left hand side of (\ref{grgu}) yields
		\begin{align*}
			\sheafhom^\bullet(\rph(\F, \G \otimes \op), \oo_Y[-d]) &= \sheafhom^0(\rph(\F, \G \otimes \op), \oo_Y[-d]) \\  &\simeq \sheafhom(\extpi{d}(\F, \G \otimes \op), \oo_Y),
		\end{align*}
		where we have used that $\rph(\F, \G \otimes \op)$ is supported in degrees $[0, d]$.
	\end{proof}
	
	To compare the perfect obstruction theories $(\ref{defmorph})$ and $(\ref{pothilb})$ in Proposition~\ref{localcomparison} we from now on without loss of generality assume that $F = \oo_Y$. We will need several preparatory Lemmas.
	\begin{lem}
		\label{sstep1}
	The universal sequence induce isomorphisms
	\begin{align}
		\label{iso}
		\extpi{0}(\qq, \qq) \simto \extpi{0}(\oo, \qq); \\
		\label{isoo2}
		\extpi{3}(\qq, \oo) \simto \extpi{3}(\qq, \qq).
	\end{align}
	\end{lem}
	\begin{proof}
		We first note that $$\extpi{0}(\qq, \qq) \simeq \pi_* \sheafhom(\qq, \qq) \quad \text{and} \quad \extpi{0}(\oo, \qq) \simeq \pi_* \sheafhom(\oo, \qq).$$
		Denote by  ${\iota \colon Z_U \into Y \times U}$ the inclusion of the (pullback of the) universal scheme, then $\qq = \iota_*\oo_{Z_U}$.
		Hence,
		\begin{align*}
			\sheafhom(\qq, \qq) &\simeq \sheafhom(\iota_*\oo_{Z_U}, \iota_*\oo_{Z_U}) \simeq \iota_* \sheafhom(\iota^* \iota_* \oo_{Z_U}, \oo_{Z_U}) \simeq  \iota_* \sheafhom( \oo_{Z_U}, \oo_{Z_U}) \\ &\simeq \iota_* \sheafhom( \iota^* \oo, \oo_{Z_U}) \simeq \sheafhom( \oo, \iota_* \oo_{Z_U}) \simeq \sheafhom( \oo, \qq).	
		\end{align*}
		Therefore, $\pi_*\sheafhom(\qq, \qq) \simto \pi_*\sheafhom( \oo, \qq)$, and $(\ref{iso})$ is an isomorphism. Similarly, 
		\begin{equation}
			\label{iso2}
			\extpi{0}(\qq, \qq \otimes \op) \simto \extpi{0}(\oo, \qq \otimes \op).
		\end{equation}
		We next show that (\ref{isoo2}) is an isomorphism.
		It is surjective, because $\extpi{4}(\qq, \sss) = 0$.
		By Lemma~\ref{dualext}, 
		$$\extpi{3}(\qq, \qq)^\vee \simeq \extpi{0}(\qq, \qq \otimes \op) \quad \text{and} \quad \extpi{3}(\qq, \oo)^\vee \simeq \extpi{0}(\oo, \qq \otimes \op). $$	
		Therefore, together with (\ref{iso2}) it implies that in the commutative diagram 
		\begin{center}
			\begin{tikzcd}
				{\extpi{3}(\qq, \oo)^{\vee \vee}} \arrow[r, "\simeq"]        & {\extpi{3}(\qq, \qq)^{\vee \vee}} \\
				{\extpi{3}(\qq, \oo)} \arrow[r, two heads, "h^2(\xi_1)"] \arrow[u, "\simeq"] & {\extpi{3}(\qq, \qq)} \arrow[u]    
			\end{tikzcd}
		\end{center}
		the upper horizontal map is an isomorphism. The left vertical map is an isomorphism because $\extpi{3}(\qq, \oo)$ is locally free. Hence, by the commutativity of the diagram, the map $h^2(\xi_1)$ is injective. Therefore, $h^2(\xi_1)$ is an isomorphism.
	\end{proof}
	
%
	
	\begin{lem}
		\label{sstep2}
		The composition
		\begin{equation}
			\label{part1}
			\tau^{\leq 1}\rph(\sss, \qq) \to \rph(\sss, \qq) \xrightarrow{\varepsilon} E_\bullet
		\end{equation}
		of $\varepsilon$ and the natural truncation map is an isomorphism.
	\end{lem}
	\begin{proof}
		By (\ref{isoo2}) the exact sequence
		$$0 = \extpi{2}(\qq, \oo) \to \extpi{2}(\qq, \qq) \to \extpi{3}(\qq, \sss) \xrightarrow{0} \extpi{3}(\qq, \oo) \simto \extpi{3}(\qq, \qq) \to 0$$
		induced by the universal sequence yields a natural isomorphism
		$$\extpi{2}(\qq, \qq) \simto \extpi{3}(\qq, \sss).$$
		Note also that the vanishing from Lemma~\ref{locextfqqf} together with the long sequence in cohomology applied to the left vertical distinguished triangle in diagram~(\ref{twodisttr}) yields the isomorphism
		$$h^2(\xi_2) \colon \extpi{1}(\sss, \qq) \simto \extpi{2}(\qq, \qq).$$
		Therefore, applying $h^2(-)$ to the upper right square in the commutative diagram~(\ref{twodisttr}) yields
		$$h^1(\varepsilon) \colon \extpi{1}(\sss, \qq) \simto h^1(E_\bullet). $$ 
		Lemma~\ref{locextfqqf} together with $(\ref{thetr})$ also implies that $h^0(\varepsilon)$ is an isomorphism.
	\end{proof}
	We next check that a similar identification holds for the dual complex of the perfect obstruction theory~(\ref{pothilb}).
	\begin{lem}
		\label{sstep3}
		The composition 
		\begin{equation}
		\label{natcomp}
		\tau^{\leq1}\rph(\sss, \qq) \to	\rph(\sss, \qq) \xrightarrow{\xi_5} \rph(\sss, \sss)[1] \to \rph(\sss, \sss)_0[1],
		\end{equation}
		where the first map is the natural truncation, is an isomorphism.
	\end{lem}
	\begin{proof}
		By \cite{thomas}, from the comparison of the tangent and obstruction sheaves for the standard obstruction theory on $\Hilb^n_Y$ and the one from the moduli space of simple sheaves,
		the composition $(\ref{natcomp})$
		induces an isomorphism on $h^0$ and an injection on $h^1$. It hence suffices to show that $(\ref{natcomp})$ induces an isomorphism on $h^1$. 
		We first note that the diagram
		\begin{center}
			\begin{equation}
				\label{tracecomm}
				\begin{tikzcd}
					& {\rph(\sss, \sss)} \arrow[d] & {} \\
					{\rph(\oo, \oo)} \arrow[r] \arrow[ru, "\id"] \arrow[rru, "(***)", near start, phantom] & {\rph(\sss, \oo),}            &   
				\end{tikzcd}
			\end{equation}
		\end{center}
		is commutative. Here, we abuse notation denoting by $\id$ also the pushforward of (\ref{id}) via $\pi$. The other morphisms are induced by the universal sequence. The digram~(\ref{tracecomm}) is obtained by applying $R\pi_*(-)$ to a commutative diagram on $Y \times U$.
		We next consider the commutative diagram
		\begin{center}
			\begin{equation}
				\begin{tikzcd}
					{\extpi{1}(\oo, \qq) = 0} \arrow[d] \arrow[r]         & {\extpi{1}(\sss, \qq)} \arrow[d] \arrow[r, "\sim"', "{h^2(\xi_2)}"] & {\extpi{2}(\qq, \qq)} \arrow[d, "\simeq"] \\
					{\extpi{2}(\oo, \sss)} \arrow[d, "\simeq"'] \arrow[r] & {\extpi{2}(\sss, \sss)} \arrow[r, "b"]               & {\extpi{3}(\qq, \sss)}                    \\
					{\extpi{2}(\oo, \oo).} \arrow[ru, "h^2(\id)"']           &                                                      &                                          
				\end{tikzcd}
			\end{equation}
		\end{center}
		 with exact rows induced by the universal sequence and by applying $h^2(-)$ to $(\ref{tracecomm})$. By the proof of Lemma~\ref{sstep2}, the maps $h^2(\xi_2)$ and the right vertical map are isomorphisms.
		By the commutativity of the right square, the map $b$ is surjective. Hence, as by exactness $\ker(b) \simeq \im(h^2(\id))$, the morphism $b$ induces an isomorphism
		$\extpi{2}(\sss, \sss)_0 \simto \extpi{3}(\qq, \sss)$. Therefore by commutativity also
		\begin{equation}
		\extpi{1}(\sss, \qq) \simto \extpi{2}(\sss, \sss)_0.
		\qedhere
		\end{equation}
	\end{proof}
	\begin{prop}
		\label{localcomparison}
		If the locally free sheaf $F$ is of rank one, the perfect obstruction theories  $(\ref{defmorph})$ and $(\ref{pothilb})$ on $U$ coincide.
	\end{prop}
	\begin{proof} 
		Without loss of generality we can assume that $F = \oo_Y$. Note that Lemma~\ref{sstep2} and Lemma~\ref{sstep3} already show that the complexes of the two perfect obstruction theories $(\ref{defmorph})$ and $(\ref{pothilb})$ are isomorphic. However, we still need to check that the maps to the cotangent complex agree.

The long exact sequence in cohomology corresponding to the distinguished triangle in diagram~(\ref{diagr}) together with vanishing in Lemma~\ref{locextfqqf} and the isomorphism $(\ref{iso})$ gives 
		\begin{equation}
			\label{righttrunciso}
			\rph(\sss, \qq) \simto \tau^{\geq 0} (\rph(\qq, \qq)[1]).
		\end{equation}
		Consider the following diagram induced by the universal sequence (\ref{uunseq}) which is an extension of diagram (\ref{diagr})
		\begin{center}
			\begin{equation}
				\label{bigdiagr}
				\begin{tikzcd}
					& {\rph(\sss, \sss)} \arrow[d]                                    & {}                                                                 &                     \\
					{\rph(\oo, \oo)} \arrow[d] \arrow[r] \arrow[ru, "\id"] & {\rph(\sss, \oo)} \arrow[r] \arrow[d]                                     & {\rph(\qq, \oo)[1]} \arrow[d, "\xi_1"'] \arrow[ld, "p"'] \arrow[d] &                     \\
					{\rph(\oo, \qq)} \arrow[r] \arrow[ruu, "(***)", near end, phantom]           & {\rph(\sss, \qq)} \arrow[r, "\xi_2"] \arrow[ruu, "(**)", near start, phantom] & {\rph(\qq, \qq)[1]} \arrow[r, "\xi_3"] \arrow[lu, "(*)", near start, phantom]  & {\rph(\oo, \qq)[1],}
				\end{tikzcd}
			\end{equation}
		\end{center}
		where all the squares are commutative.
		By construction of $p$, the triangle $(*)$ is commutative. Applying $\tau^{\geq 0}$ to the square composed of triangles $(**)$ and $(*)$, together with (\ref{righttrunciso}) shows that the triangle $(**)$ is commutative as well. The triangle $(***)$ is also commutative by (\ref{tracecomm}).
		We next show that the following diagram coming from (\ref{bigdiagr}) 
		\begin{center}
			\begin{equation}
				\label{compare}
				\begin{tikzcd}
					{\rph(\oo, \oo)} \arrow[d, "\id"'] \arrow[r]              & {\rph(\sss, \oo)} \arrow[d, equals] \arrow[r] & {\rph(\qq, \oo)[1]} \arrow[d, "p"] \arrow[r, "{\xi_6}"]              & {\rph(\oo, \oo)[1]} \arrow[d, "{\id[1]}"] \\
					{\rph(\sss, \sss)} \arrow[r] \arrow[ru, "(***)", phantom] & {\rph(\sss, \oo)} \arrow[r] \arrow[ru, "(**)", phantom]    & {\rph(\sss, \qq)} \arrow[r, "{\xi_5}"'] \arrow[ru, "{(****)}", phantom] & {\rph(\sss, \sss)[1]}                    
				\end{tikzcd}
			\end{equation}
		\end{center}
		is a map between distinguished triangles. It suffices to check that the square $(****)$ is commutative. This follows from the commutativity of the diagram
		\begin{center}
			\begin{equation}
				\label{refine}
				\begin{tikzcd}
					{\rph(\qq, \oo)} \arrow[dd, "p"] \arrow[rr, , "{\xi_6[-1]}"] \arrow[rd, "{\xi_1[-1]}"]                       & {} \arrow[d, "{(1)}", phantom]         & {\rph(\oo, \oo)}                            \\
					& {\rph(\qq, \qq)} \arrow[ru, "tr_{\qq}"] &                                             \\
					{\rph(\sss, \qq)[-1]} \arrow[rr, "{\xi_5[-1]}"'] \arrow[ru, "{\xi_2[-1]}"] \arrow[ruu, "(*)", phantom] & {} \arrow[ruu, "{(2)}", near start, phantom]        & {\rph(\sss, \sss)} \arrow[uu, "tr_{\sss}"']
				\end{tikzcd}
			\end{equation}
		\end{center}
		together with $tr_{\sss} \circ \id$ being identity. The triangle $(*)$ is commutative by construction. Parts $(1)$ and $(2)$ in (\ref{refine}) are commutative as they are induced by applying $R\pi_*(-)$ to diagrams which are commutative by the construction of the trace map \cite{huybrechts_lehn}.
		
		Diagram $(\ref{compare})$ can be extended to 
		\begin{center}
			\begin{equation}
				\label{finaldiag}
				\begin{tikzcd}
					{\rph(\sss, \oo)} \arrow[d, equals] \arrow[r] & {\rph(\qq, \oo)[1]} \arrow[d, "p"] \arrow[r, "\xi_6"]                                        & {\rph(\oo, \oo)[1]} \arrow[d, "{\id[1]}"] \\
					{\rph(\sss, \oo)} \arrow[r] \arrow[ru, "(**)", phantom]    & {\rph(\sss, \qq)} \arrow[r, "\xi_5"'] \arrow[ru, "(****)", phantom] \arrow[d, "\varepsilon"] & {\rph(\sss, \sss)[1]} \arrow[d]           \\
					& E_\bullet \arrow[r, "\mu", equals]                                                                  & {\rph(\sss, \sss)_0[1],}                  
				\end{tikzcd}
			\end{equation}
		\end{center}	
		where the two right vertical columns are distinguished triangles. We obtain a map between these distinguished triangles as $(****)$ is commutative, with $$\mu \colon E_\bullet \to \rph(\sss, \sss)_0[1]$$ the induced morphism. The map $\mu$ is an isomorphism as the induced maps on cohomology are isomorphisms which follows from Lemma~\ref{sstep2} and Lemma~\ref{sstep3} by applying $\tau^{\leq 1}(-)$ to the lower square of diagram~(\ref{finaldiag}).
		As the Cone is defined up to an isomorphism, we can hence assume that $\mu$ is the identity. This identifies the complexes of the two perfect obstruction theories $(\ref{defmorph})$ and $(\ref{pothilb})$. The maps to the cotangent complex then coincide by the commutativity of the lower square in diagram~(\ref{finaldiag}).
	\end{proof}
	
	It will be shown in the next section (Proposition~\ref{pglobal}) that in the rank one case the map $p$ can be defined globally, and the induced perfect obstruction theory agrees with the one in \cite{thomas}.

	\section{Almost perfect obstruction theory} \label{globally}
	In this section we show that the collection of perfect obstruction theories constructed in the previous section (étale) locally on an affine glue in a suitable way to get an almost perfect obstruction theory on $\Q$. This allows one to define a virtual class of virtual dimension zero.

	We now denote by 
	\begin{equation}
		\label{unseq}
		0 \to \sss \to q^*F \to \qq \to 0
	\end{equation}
	the universal sequence on $Y \times \Q$ with the two projections 
	$$\pi \colon Y \times \Q \to \Q \quad \text{and} \quad {q \colon Y \times \Q \to Y}.$$ 
	For each affine étale chart $U \to \Q$ the previous section gives a map
	$$p_U \colon \rph(\qq, q^*F)[1]|_U \to \rph(\sss, \qq)|_U $$
	defining a perfect obstruction theory 
	$$\varphi_U \colon E^\bullet_U \coloneqq (\Cone(p_U)) ^\vee \to \LL_U$$ on $U$ with obstruction sheaf isomorphic to $\extpi{3}(\qq, \sss)|_U$, see Propositions~$\ref{potlocally}$ and $\ref{obstrsheaf}$.
	
	
	Consider any affine étale cover of $\Q$:
	\begin{equation}
		\label{cover}
		\{U_i \xrightarrow{g_i} \Q \mid U_i = \Spec R_i, \; g_i \; \text{étale} \}_{i \in I}
	\end{equation}
	with perfect obstruction theories 
	\begin{equation}
		\label{potoncover}
		\{\varphi_i \colon E_i^\bullet \to \LL_{U_i}\}_{i \in I}
	\end{equation}	
	constructed on each $U_i$ by $\varphi_{U_i}$ as above.
	For each $i$, there is an exact triangle
	\begin{equation}
		\label{idist}
		\rph(\qq, q^*F)[1]|_{U_i} \xrightarrow{p_i} \rph(\sss, \qq)|_{U_i} \xrightarrow{\varepsilon_i} E_{i \bullet} \xrightarrow{\rho_i} \rph(\qq, q^*F)[2]|_{U_i}.
	\end{equation}
	Here, $p_i \coloneqq p_{U_i}$, $E^\bullet_{i} \coloneqq E_{U_i}^\bullet$, and $E_{i \bullet} \coloneqq (E_i^{\bullet})^\vee.$
	\vspace{3pt}
	
	Note that as $\Q$ is separated, the fiber product 
	$U_{ij} \coloneqq U_i \times_\Q U_j$
	is affine.
	
	\begin{theor}
		\label{quotsemipot}
		The étale affine cover $(\ref{cover})$ together with perfect obstruction theories $(\ref{potoncover})$ on it defines an almost perfect obstruction theory on $\Q$.
	\end{theor}
	\begin{proof}
		Put  
		$$Ob_{\Q} \coloneqq \extpi{3}(\qq, \sss).$$ 
		By Proposition $\ref{obstrsheaf}$, 
		restricted to each chart $U_i$, it recovers the obstruction sheaf of the corresponding perfect obstruction theory on $U_i$. For each $i \in I$ there is an isomorphism
		$$\psi_i \colon  Ob_{U_i} \simto Ob_{\Q}|_{U_i},$$
		with $\psi_i = h^1(\ga_i)$, see the notation in diagram~(\ref{twodisttr}) for $U = U_i$. Hence, the first condition in Definition~$\ref{APOT}$ of an almost perfect obstruction theory is satisfied. 
		It remains to check that for each $i, j$ there exists a map 
		$$\eta_{ij} \colon E_{i\bullet}|_{U_{ij}} \simto E_{j\bullet}|_{U_{ij}} $$
		compatible with maps to the cotangent complex and
		such that $h^1(\eta_{ij}) = \psi_j|_{\uij}^{-1} \circ \psi_i|_{\uij}$.
		Note that uniqueness of the map $p$ in diagram $(\ref{diagr})$ implies an equality of morphisms
		$$p_i|_{U_{ij}} = p_j|_{U_{ij}} \eqqcolon p_{ij}. $$
		Hence, both $E_{i\bullet}|_{U_{ij}}$ and $E_{j\bullet}|_{U_{ij}}$ are choices of $\Cone(p_{ij})$ and hence by Proposition~\ref{conechoices} applied to ${U = \uij}$, there exists the desired $\eta_{ij}$. Indeed, by Proposition~\ref{conechoices}, such $\eta_{ij}$ satisfies $\eta_{ij} \circ \varepsilon_i|_{\uij} = \varepsilon_j|_{\uij}$ which ensures compatibility with the maps to the cotangent complex. It also satisfies $\ga_i|_{\uij} = \ga_j|_{\uij} \circ \eta_{ij}$ which, after taking first cohomology, yields $h^1(\eta_{ij}) = \psi_j|_{\uij}^{-1} \circ \psi_i|_{\uij}$.
	\end{proof}
	
	Theorem $\ref{quotsemipot}$ yields a virtual class 
	\begin{equation}
		\label{virtclass}
		[\Q]^{\vir} \in A_0(\Q)
	\end{equation}	
	of virtual dimension zero and a global virtual structure sheaf
	\begin{equation*}
		\oo_\Q^{\vir} \in K_0(\Q).
	\end{equation*}
	
	\begin{rem}
		In Theorem \ref{quotsemipot} we have considered an arbitrary affine étale cover (\ref{cover}). For any two affine étale covers the corresponding almost perfect obstruction theories are equivalent in the sense of Definition~\ref{sameAPOT} as there is a third APOT with respect to an affine étale cover refining the first two.
	\end{rem}

	\begin{rem}
		\label{quotglkthclass}
		According to Remark $\ref{glkthcl}$, the almost perfect obstruction theory from Theorem~$\ref{quotsemipot}$ on $\Q$ has a global class in K-theory
		$$[T_\Q] - [Ob_\Q] = [\extpi{0}(\sss, \qq)] -[\extpi{3}(\qq, \sss)] \in K_0(\Q)$$ which can be lifted to a class $E_\bullet \in K^0(\Q)$
		\begin{equation}
			\label{kth1}
			E_\bullet \coloneqq \rph(\qq, q^*F) + \rph(\sss, \qq) = \rph(q^*F, q^*F) - \rph(\sss, \sss)
		\end{equation}	
		where we omit the square brackets for appearance convenience. The equality is obtained by exploiting the universal sequence of $\Q$. The class above can also be rewritten as
		\begin{equation}
			\label{kth2}
			E_\bullet = \rph(q^*F, \qq) + \rph(\qq, q^*F) - \rph(\qq, \qq) \in K^0(\Q).
		\end{equation}
	\end{rem}
	
	\begin{rem}
		Note that in the case when the $3$-fold $Y$ is Calabi-Yau, the obstruction sheaf of the APOT from Theorem \ref{quotsemipot} is
		$$\extpi{3}(\qq, \sss) \simeq \extpi{3}(\qq, \sss \otimes \omega_\pi) \simeq \Omega_{\Q}, $$
		where the second isomorphism is by \cite[Equation (3)]{LehnQuot} and holds for an arbitrary $Y$. This is consistent with Proposition~\ref{cyafflocal}, stating that in this case the local perfect obstruction theories are symmetric. 
	\end{rem}
	We end this section by showing that in rank one case the maps $p_i$ are induced from a global map $p$ which yields a global perfect obstruction theory that agrees with the one in \cite{thomas}. Recall that the perfect obstruction theory on $\Hilb_Y^n$ in \cite{thomas} is 
	\begin{equation}
		\label{thomaspot}
		(\rph(\sss, \sss)_0[1])^\vee \to \LL_{\Hilb_Y^n},
	\end{equation}
	constructed via the Atiyah class. Here, $\rph(\sss, \sss)_0$ is the pushforward of the Cone of
	$$\id \colon \rh(\oo, \oo) \to \rh(\sss, \sss), $$
	which is the splitting of the trace map \cite{huybrechts2006fourier}. We will abuse notation and denote by $\id$ also its pushforward via $\pi$.
	Without loss of generality we assume that $F = \oo_Y$.
	\begin{prop}
		\label{pglobal}
		There exists a global map 
		$$p \colon \rph(\qq, \oo)[1] \to \rph(\sss, \qq),$$
		such that for each $i$, $p|_{U_i} = p_i$. Moreover, the composition
		\begin{equation}
			\label{globalpotoy}
			(\Cone(p))^\vee \to (\rph(\sss, \qq))^\vee \to \LL_{\Hilb_Y^n} 
		\end{equation}
		is a perfect obstruction theory which can be identified with (\ref{thomaspot}).
	\end{prop}
	\begin{proof}
		Consider the map between distinguished triangles (compare it with diagrams (\ref{compare}) and (\ref{finaldiag})  in the local case)
		\begin{center}
			\begin{equation}
				\label{globfindiag}
				\begin{tikzcd}
					{\rph(\oo, \oo)} \arrow[d, "\id"'] \arrow[r] & {\rph(\sss, \oo)} \arrow[d, equals] \arrow[r] & {\rph(\qq, \oo)[1]} \arrow[d, "\tilde{p}", dotted] \arrow[r] & {\rph(\oo, \oo)[1]} \arrow[d, "{\id[1]}"] \\
					{\rph(\sss, \sss)} \arrow[r]                 & {\rph(\sss, \oo)} \arrow[r]           & {\rph(\sss, \qq)} \arrow[d, dotted, "{\tilde{\varepsilon}}"] \arrow[r]                & {\rph(\sss, \sss)[1]} \arrow[d]           \\
					&                                       & \Cone(\tilde{p}) \arrow[r, equals]                                   & {\rph(\sss, \sss)_0[1]}                  
				\end{tikzcd}
			\end{equation}
		\end{center}
		By the octahedral axiom, there exists a map $\tilde{p}$ such that $\Cone(\tilde{p})$ can be identified with $\rph(\sss, \sss)_0[1]$. We denote the map to $\Cone(\tilde{p})$ by $\tilde{\varepsilon}$.
		By commutativity, $\tilde{\varepsilon}$ coincides with the composition map
		\begin{equation}
			\label{globcomp}
			\rph(\sss, \qq) \to \rph(\sss, \sss)[1] \to \rph(\sss, \sss)_0[1].
		\end{equation}
		In particular, for an affine chart $U_i \to \Hilb_Y$, the restriction $\tilde{\varepsilon}|_{U_i}$ coincides with the restriction of (\ref{globcomp}) to $U_i$. 
		By the commutativity of the lower square of diagram~(\ref{finaldiag}) in the proof of Proposition~\ref{localcomparison} in the local case for  $U_i$, this restriction also coincides with $\varepsilon_i$.  
		Therefore, for each $i$ the maps $p_i$ and $\tilde{p}|_{U_i}$ are maps corresponding to shifted Cones of $\tilde{\varepsilon}|_{U_i} = \varepsilon_i$. As two choices of a Cone differ by an isomorphism, there exists an automorphism 
		$$b_i \colon \rph(\qq, \oo)[1]|_{U_i} \simto \rph(\qq, \oo)[1]|_{U_i},$$
		such that $\tilde{p}|_{U_i} = p_i \circ b_i$.
		As $\rph(\qq, \oo)[1] \simeq \extpi{3}(\qq, \oo)[-2]$ is a shifted vector bundle, $b_i$ can be viewed as an element in $$\Hom(\extpi{3}(\qq, \oo)|_{U_i}, \extpi{3}(\qq, \oo)|_{U_i}) = \sheafhom(\extpi{3}(\qq, \oo), \extpi{3}(\qq, \oo))(U_i).$$
		Moreover, $b_i = h^2(b_i) = h^2(p_i)^{-1} \circ h^2(\tilde{p}|_{U_i})$. Therefore, as by the proof of Proposition~\ref{quotsemipot} for every $i$ and $j$ there is an identification $p_i|_{\uij} = p_j|_{\uij}$, we also get 
		$b_i|_{\uij} = b_j|_{\uij}$ for each $i$ and $j$. Therefore, the maps $\{b_i\}_i$ can be glued to a global map $b$. Put
		$$p \coloneqq \tilde{p} \circ b^{-1}.$$ Then $(\ref{globalpotoy})$ is a perfect obstruction theory, as it is a perfect obstruction theory locally. $\Cone(p)$ can be identified with $\Cone(\tilde{p})$ via $b$ and hence also with $\rph(\sss, \sss)_0$. The maps to the cotangent complex agree by the commutativity of the lower square in diagram~(\ref{globfindiag}).
	\end{proof}
	\begin{rem}
		The proof of Lemma \ref{sstep3} also work globally, and hence the complexes of the perfect obstruction theories $(\ref{thomaspot})$ and $(\ref{globalpotoy})$ on $\Hilb_Y^n$ can be identfied with $\tau^{\leq 1}\rph(\sss, \qq)$.
	\end{rem}
	\section{Smooth pullbacks, equivariant Jouanolou trick, Siebert's formula and virtual torus localization for almost perfect obstruction theories} \label{s5}
	
	
	\subsection{Smooth pullbacks of (almost) perfect obstruction theories}
	We recall following \cite[Section 2.2]{kuhnliuthimm} how to pull back an (almost) perfect obstruction theory along a smooth morphism.
	\begin{lem}{\cite[Lemma 2.3.3]{kuhnliuthimm}}
		\label{kltvanishing}
		Let $A$ be an affine scheme acted on by a torus $T$. Let $\mathcal{V}$ be a $T$-equivariant vector bundle on $A$ and $E \in \mathcal{D}^{T, \leq a}(A)$ for some $a \in \mathbb{Z}$. Then $\Ext^i(\mathcal{V}, E) = 0$ for all $i > a$.
	\end{lem}
	
	\begin{cor}
		\label{splitting}
		Let $f \colon A \to X$ be a smooth morphism between schemes with $A$ affine. Then the set of splittings of the natural morphism $\LL_A \to \Omega_f$ is nonempty. Moreover, it is a torsor under the finite-dimensional vector space $\Hom(\Omega_f, f^*\LL_X)$.
	\end{cor}
	\begin{proof}
		Note that the sheaf of relative differentials $\Omega_f$ is a vector bundle as $f$ is smooth. Consider the distinguished triangle
		\begin{equation}
			\label{disttr}
			f^*\LL_X \to \LL_A \to \Omega_f. 
		\end{equation}
		By Lemma $\ref{kltvanishing}$,  $\Hom(\Omega_f, f^*\LL_X[1]) =0.$ In particular, the connecting homomorphism in (\ref{disttr}) vanishes. Also, $\Hom(\Omega_f, \Omega_f[-1]) = 0.$
		\begin{center}
			\begin{tikzcd} [row sep = 6pt] 
				& \Omega_f \arrow[d, equals] \arrow[ld, "\exists"', dotted] &          \\
				\LL_A \arrow[r] & \Omega_f \arrow[r, "0"]                           & f^*\LL_X[1]
			\end{tikzcd}
		\end{center}
		The statement follows now from the long exact sequence obtained by applying $\Hom_{\DD(A)}(\Omega_f, -)$ to the distinguished triangle (\ref{disttr}), which gives
		\begin{equation}
			0  \to \Hom(\Omega_f, f^*\LL_X) \to \Hom(\Omega_f, \LL_A) \to \Hom(\Omega_f, \Omega_f) \to \Hom(\Omega_f, f^*\LL_X[1]) = 0. {\qedhere} 
		\end{equation}
	\end{proof}
	\begin{prop}
		\label{potpullback}
		Let $f \colon A \to X$ be a smooth morphism between schemes with $A$ affine and let $\varphi \colon E^\bullet \to \LL_X$ be a perfect obstruction theory on $X$. Then to each splitting of the natural map $\LL_A \to \Omega_f$ as in Corollary \ref{splitting} we can associate a perfect obstruction theory $\tilde{\varphi} \colon \tilde{E^\bullet} \to \LL_A$ on $A$ such that $f^*[X]^{\vir} = [A]^{\vir}$.
	\end{prop}
	\begin{proof}
		Consider the distinguished triangle (\ref{disttr}) and note that a choice of a splitting of the natural map $\LL_A \to \Omega_f$ yields an isomorphism
		$\LL_A \simeq f^*\LL_X \oplus \Omega_f.$
		Define the desired perfect obstruction theory by putting $$\tilde{E}^\bullet \coloneqq f^*E^\bullet \oplus \Omega_f $$ and 
		\begin{equation*}
			\tilde{\varphi} \colon \tilde{E}^\bullet \xrightarrow{(f^*\varphi, \id)} f^*\LL_X \oplus \Omega_f \simeq \LL_A  
		\end{equation*}
		Note that $\tilde{\varphi}$ is a perfect obstruction theory as $\varphi$ is and
		$Ob_A = f^* Ob_X$. 
		By construction, the perfect obstruction theories $\varphi$ and $\tilde{\varphi}$ are compatible in the sense of \cite[Definition 5.8]{Behrend_1997}. Hence, by \cite[Corollary~4.9]{manolachevirtpullb}, we have the desired compatibility of virtual classes.
	\end{proof}
	
	\begin{convention}
		A perfect obstruction theory $\tilde{\varphi}$ on $A$ constructed as in the proof of Proposition~\ref{potpullback} will be referred to as a pullback obstruction theory of $\varphi$ via$~f$.
	\end{convention} 
	
	The following Proposition is an analog of \cite[Proposition 2.2.2]{kuhnliuthimm} in Chow.
	\begin{prop}
		\label{semipotpullback}
		Let $f \colon Y \to X$ be a smooth morphism between schemes and let $ \varphi = \{U_i, \varphi_i \colon E^\bullet_i \to \LL_{U_i} \}_{i \in I}$ be an almost perfect obstruction theory on $X$. Then there exists an almost perfect obstruction theory $$\tilde{\varphi} = \{\widetilde{U_i}, \tilde{\varphi_i} \colon E^\bullet_i \to \LL_{\widetilde{U}_i} \}_{i \in I}$$ on $Y$ such that ${f^*[X]^{\vir} = [Y]^{\vir}}$.
	\end{prop}
	More precisely, we can associate such an almost perfect obstruction theory to each set of splittings as in Corollary $\ref{splitting}$ for each chart.
	\begin{proof}
		Consider the étale cover $\{\widetilde{U_i} \to Y\}_{i \in I}$ of $Y$ with $\widetilde{U_i} \coloneqq U_i \times_X Y$.
		For each $i$ consider a pullback perfect obstruction theory of $\varphi_i$ via $f_i$ 
		$$\widetilde{\varphi_i} \colon \tilde{E_i^\bullet} \coloneqq f_i^*E_i^\bullet \oplus \Omega_{f_i} \xrightarrow{(f_i^* \varphi_i, \id)} f_i^*\LL_{U_i} \oplus \Omega_{f_i} \simeq \LL_{\widetilde{U_i}}$$
		on $\widetilde{U_i}$ constructed in the proof of Proposition $\ref{potpullback}$,
		with $Ob_{\widetilde{U_i}} \simeq f_i^*Ob_{U_i}$. This data forms an almost perfect obstruction theory on $Y$ with obstruction sheaf
		$Ob_{Y} = f^*Ob_{X}. $
		We have the following chain of equalities
		\begin{equation}
			\label{compvirtclass}
			f^*[X]^{\vir} = f^*0^!_{Ob_X}[c_X] = 0^!_{f^* Ob_X}f^*[c_X] = 0^!_{Ob_Y}[c_{Y}] = [Y]^{\vir}.
		\end{equation}
		The second equality holds as Gysin pullback is compatible with flat pullbacks \cite{kresch}. The third equality holds as by \cite[Proof 2.2.4]{kuhnliuthimm} there is an equality $c_{Y} = f^{-1}c_X$ of closed substacks in $Ob_{Y}$ and hence $f^*[c_X] = [c_{Y}] \in A_*(Ob_{Y})$.
	\end{proof}
	
	\begin{convention}
		We will refer to an almost perfect obstruction theory $\tilde{\varphi}$ constructed as in the proof of Proposition $\ref{semipotpullback}$ as a pullback almost perfect obstruction theory of $\varphi$ via$~f$.
	\end{convention}

	\begin{rem}
		Theorems $\ref{potpullback}$ and $\ref{semipotpullback}$ are also true in the equivariant setting, as the proofs work equivariantly.
	\end{rem}
	
	\subsection{Torus-equivariant Jouanolou trick} \label{jouanolou}
	
	The original result of Jouanolou \cite[Lemma 1.5]{jouanolou} states:
	
	\begin{theor}
		Let $X$ be a quasi-projective scheme over an affine scheme. Then there exists an affine bundle (i.$\:$e.$\:$a vector bundle torsor) $A \to X$ with $A$ affine.
	\end{theor}
	Thomason extended this result to an arbitrary quasicompact and quasiseparated scheme with an ample family of line bundles \cite[Proposition 4.4]{weibel}. Theorem $\ref{tequiv}$ below is an equivariant version of this result. We use the strategy from \cite[Proof 2.5.3]{kuhnliuthimm} presenting an elegant proof generously communicated to the author by D. Rydh.
	
	\begin{theor}{(Totaro, $GL_n$-equivariant Jouanolou trick)}
		\label{totaro}
		Let $X$ be a stack of the form $[\tilde{W} / GL_n]$ with $\tilde{W}$ quasi-affine. Then there exists an affine bundle $E = [W/ GL_n] \to X$ with $W$ affine.  
	\end{theor}
	\begin{proof}
		By the proof of implication $(2) \implies (3)$ of \cite[Theorem 1.1]{totaro_resolution}, there exists an $\Aff_{r-1}$-torsor $E = [W/ GL_n] \to X$ with $W$ affine for a certain $r \in \mathbb{N}$, where $\Aff_{r-1} \simeq GL(r-1) \ltimes (G_a)^{r-1}$. Hence, it is an affine bundle.
	\end{proof}
	
	\begin{theor}{\cite[Theorem A]{gross-stack}}
		\label{gross}
		Let $X$ be a quasi-compact quasi-separated algebraic stack. The following are equivalent:
		\begin{enumerate}
			\item $X$ has affine stabilizer groups at closed points and satisfies the resolution property.
			\item $X$ is isomorphic to a quotient stack $[W/ GL_n]$, where $W$ is a quasi-affine scheme with an action of $GL_n$ for some $n \geq 0$.
		\end{enumerate}
	\end{theor}
	
	\begin{theor}
		\label{tequiv}
		Let $X$ be a quasi-separated quasi-compact algebraic space acted on by a torus $T$ and satisfying the $T$-equivariant resolution property. Then there exists a $T$-equivariant affine bundle $A \to X$ with $A$ affine.
	\end{theor}
	
	\begin{proof} The stabilizer groups at closed points can be viewed as closed subgroups of $T$, so they are closed subvarieties in the affine variety $T$ and hence affine.
		Hence, by Theorem $\ref{gross}$ of Gross, there exists a quasi-affine scheme $W$ with a $GL_n$-action such that $[X/T] \simeq [W/ GL_n]$. Hence, by Theorem~$\ref{totaro}$ of Totaro 
		there exists an affine scheme $Q$ with a $GL_n$ action and an affine bundle 
		$$e \colon [Q/ GL_n] \to [X/T] \simeq [W/GL_n]. $$ Consider the following fibered diagram
		\begin{center}
			\begin{tikzcd}
				{[\tilde{Q}/ GL_n]} \arrow[d] \arrow[r] & X \arrow[d]             \\
				{[Q/ GL_n]} \arrow[r, "e"]              & {[X/T] \simeq [W/GL_n],}
			\end{tikzcd}
		\end{center}
		
		where, $\tilde{Q} \coloneqq Q \times_{[X/T]} X$ and $\tilde{Q} \to Q$ is a $T$-torsor. Then, $\tilde{Q}$ is also an affine scheme. By \cite[Remark (2), p.2]{totaro_resolution}, $[\tilde{Q}/GL_n]$ is an affine scheme as well. The horizontal arrows in the diagram above are affine bundles and the vertical arrows are $T$-torsors. Hence the upper horizontal morphism is the desired $T$-equivariant affine bundle with $A \coloneqq  [\tilde{Q}/GL_n]$ affine.
	\end{proof}
	
	\subsection{Siebert's formula for almost perfect obstruction theories} \label{siebertformula} The classical formula of Siebert \cite[Theorem 4.6]{siebert} states that given a perfect obstruction theory $\varphi \colon E^\bullet \to \LL_X$ the corresponding virtual fundamental class can be expressed as 
	\begin{equation}
		\label{siebform}
		[X]^{\vir} = [s(E_\bullet) \cap c_F(X)]_{\vd},
	\end{equation}
	where $E_\bullet = E^{\bullet \vee}$, $\vd$ is the virtual dimension of $\varphi$ and $$c_F(X) \coloneqq c(T_Y|_X) \cap s(C_XY) \in A_*(X),$$ for any embedding of $X$ into a smooth scheme $Y$, is the Fulton's canonical class (see \cite[Example 4.2.6]{Fulton_Intersection_Theory}). 
	
	We aim to prove a similar formula to $(\ref{siebform})$ for almost perfect obstruction theories under Assumption~$\ref{siebassumpt}$ below working equivariantly.

	Let $X$ be a scheme acted on by a torus $T$ and
	\begin{equation} 
		\label{oursemipot}
		\varphi = \{U_i, \varphi_i \colon E_i^\bullet \to \LL_{U_i}\}_{i \in I}
	\end{equation} 
	be a $T$-equivariant almost perfect obstruction theory on $X$. Consider the following assumption:
	\begin{assumption}
		\label{siebassumpt}
		\begin{enumerate}
			\item \label{assumpt1} The class $[T_X] - [Ob_X] \in K_0^T(X)$ lies in the image of the natural map ${K^{0, T}(X) \to K_0^T(X)}$ and there exists a preimage $$E_\bullet \in K^{0, T}(X)$$ with $E_\bullet|_{U_i} \simeq E_i^\bullet \in K^{0, T}(U_i)$ for each $i \in I$.
			\item \label{assumpt2} There exists a $T$-equivariant embedding $X \into Y$ into a smooth scheme $Y$ which admits a $T$-equivariant affine bundle $B \to Y$ with $B$ affine 
			\begin{center}
				\begin{tikzcd}
					A \arrow[r, hook] \arrow[d, "a"'] & B \arrow[d, "b"] \\
					X \arrow[r, hook]                 & Y               
				\end{tikzcd}
			\end{center}
			such that a pullback $a^* \varphi$ of the almost perfect obstruction theory $\varphi$ via the induced affine bundle $a \colon A \to X$ 
			(see Proposition $\ref{semipotpullback}$)
			gives a perfect obstruction theory on $A$ with K-theory class $$a^*E_\bullet + T_{A/X} \in K^{0, T}(A).$$
		\end{enumerate}
	\end{assumption}
	
	\begin{theor}
		\label{apotpotsiebert}
		Let $X$ be a scheme acted on by a torus $T$ and $\varphi$ a $T$-equivariant almost perfect obstruction theory $(\ref{oursemipot})$ on $X$ satisfying Assumption $\ref{siebassumpt}$. Then
		$$[X]^{\vir} = [s(E_\bullet) \cap c_F(X)]_{\vd} \in A_*^T(X), $$
		where $E_\bullet$ denotes the class from Assumption $\ref{siebassumpt}$. 
	\end{theor}
	
	\begin{proof}
		We use the notation from Assumptions $\ref{siebassumpt}$ above, throughout the proof we work equivariantly. Let $$\psi \colon \tilde{E}^\bullet \to \LL_A$$ be a pullback perfect obstruction theory of $\varphi$ via the affine bundle $a$. By Proposition~$\ref{semipotpullback}$,
		\begin{equation}
			\label{l1}
			a^*[X]^{\vir} = [A]^{\vir}.
		\end{equation} Siebert formula for $\psi$ on $A$ gives
		\begin{equation}
			\label{l2}
			[A]^{\vir} = [s(\tilde{E}_\bullet) \cap c_F(A)]_{\vd + r},
		\end{equation}
		where $r$ is the relative dimension of the smooth morphism $a$. As $a$ is an affine bundle, the pullback $$a^* \colon A_*^T(X) \to A_*^T(A)$$ is an isomorphism by an equivariant extension of \cite[\href{https://stacks.math.columbia.edu/tag/0GUB}{Tag 0GUB}]{stacks-project}. Hence, by $(\ref{l1})$ and $(\ref{l2})$, it suffices to check that
		\begin{equation}
			\label{suff}
			a^*[s(E_\bullet) \cap c_F(X)]_{\vd} = [s(\tilde{E}_\bullet) \cap c_F(A)]_{\vd + r} \in A_{\vd + r}^T(A).
		\end{equation}
		By assumption,
		$$[\tilde{E}_\bullet] = a^*E_\bullet + [T_{A/X}] \in K^{0,T}(X).$$
		Hence, 
		\begin{equation}
			\label{l3}
			s(\tilde{E}_\bullet) \cap c_F(A) = a^*s(E_\bullet)s(T_{A/X})c(T_B|_A) \cap s(C_AB)
		\end{equation}
		By \cite[Proposition 4.2 (b)]{Fulton_Intersection_Theory}, 
		\begin{equation}
			\label{conepullback}
			s(C_AB) = a^*s(C_XY).
		\end{equation}
		The short exact sequence of vector bundles
		$$0 \to T_{B/Y}|_A \to T_B|_A \to b^*T_Y|_A \to 0 $$ gives
		\begin{equation}
			\label{eqcherncl}
			c(T_B|_A) = c(T_{B/Y}|_A)c(b^*T_Y|_A) = c(T_{A/X})c(a^*T_Y|_X)
		\end{equation}
		Therefore in view of $(\ref{conepullback})$ and $(\ref{eqcherncl})$ the right hand side of $(\ref{l3})$ can be rewritten as 
		$$a^*s(E_\bullet)s(T_{A/X})c(T_B|_A) \cap s(C_AB) = a^*(s(E_\bullet)c(T_Y|_X) \cap s(C_XY)).$$
		Hence, we get
		$$s(\tilde{E}_\bullet) \cap c_F(A) = a^*(s(E_\bullet)c(T_Y|_X) \cap s(C_XY)) = a^*(s(E_\bullet) \cap c_F(X)), $$
		and therefore also $(\ref{suff})$. 
	\end{proof}
	\subsection{Virtual torus localization for almost perfect obstruction theories}
	We will prove the torus localization formula for almost perfect obstruction theories of Kiem \cite[Theorem 4.5]{kiem} (stated for semi-perfect obstruction theories) replacing \cite[Assumption 4.3]{kiem} by Assumption $\ref{toruslocassumpt}$ which will hold for the Quot scheme of points.
	
	More precisely, \cite[Assumption 4.3]{kiem} requires existence of a two-term complex on the fixed locus which plays a role of the virtual normal bundle. However, to state the formula itself, one only needs its class $N^{\vir}$ in $K$-theory which, as in the case of the Quot scheme of points, might be easier to construct. We are inspired by \cite[Theorem 2.4.3]{kuhnliuthimm} which is an analogous statement in K-theory.
	
	
	
	Let $X$ be a scheme acted on by a torus $T$ and
	\begin{equation} 
		\label{toursemipot}
		\varphi = \{U_i, \: \varphi_i \colon E_i \to \LL_{U_i}\}_{i \in I}
	\end{equation} 
	be a $T$-equivariant almost perfect obstruction theory on $X$. Recall from \cite{kiem2020localizing}, that the APOT on $X$ induces an APOT on the fixed locus $X^T$ 
	\begin{equation}
		\label{apotfixloc}
		\varphi|_{X^T}^{\fix} \coloneqq \{U_i^T = X^T \times_X U_i, \: \varphi_i|_{U_i^T}^{\fix} \colon E_i|_{U_i^T}^{\fix} \to \LL_{U_i^T}\}_{i \in I}
	\end{equation}
	obtained from $\varphi$ by taking fixed parts of the torus action of the restriction to the fixed locus on each étale chart. To prove the torus localization formula we need the following assumption.
	\begin{assumption}
		\label{toruslocassumpt}
		\begin{enumerate}
			\item \label{tassumpt1} The class $[T_X] - [Ob_X] \in K_0^T(X)$ lies in the image of the natural map ${K^{0, T}(X) \to K_0^T(X)}$ and there exists a preimage $$E_\bullet \in K^{0, T}(X)$$ with $E_\bullet|_{U_i} \simeq E_{U_i} \in K^{0, T}(U_i)$ for each $i \in I$. 
			
			\item \label{tassumpt2} There exists a $T$-equivariant affine bundle $a \colon A \to X$ with $A$ affine such that a pullback $a^* \varphi$ of the almost perfect obstruction theory $\varphi$ via $a$ (see Proposition $\ref{semipotpullback}$) gives a perfect obstruction theory  
			$$\varphi_A \colon E_A^{\bullet} \oplus \Omega_{A/X} \to \LL_A$$
			on $A$ such that 
			$$(E_A^{\bullet})^{\vee} = a^*E_{\bullet} \in K^{0, T}(A).$$
		\end{enumerate}
	\end{assumption}
	
	Under Assumption~\ref{toruslocassumpt}, we have the following fibered diagram
	\begin{center}
		\begin{tikzcd}
			A^T \arrow[r, hook] & F_A \arrow[r, hook] \arrow[d, "a|_{F_A}"'] & A \arrow[d, "a"] \\
			& X^T \arrow[r, hook]                          & X               
		\end{tikzcd}
	\end{center}
	where we have put $F_A \coloneqq a^{-1}(X^T)$. All horizontal arrows are closed embeddings.
	Before proving the torus localization formula, we check the following statement.
	\begin{lem}
		\label{fixpotpullback}
		The APOT obtained as a pullback of $\varphi|^{\fix}_{X^T}$ via $a|_{F_A}$ is a perfect obstruction theory.
	\end{lem}
	\begin{proof}
		This follows from the assumption that the pullback of our APOT $\varphi$ via $a$ is a perfect obstruction theory $E_A^\bullet \oplus \Omega_{A/X}$. Restricting it to $F_A$ and taking fixed part of the first direct summand will yield the APOT of interest, which is then also a POT. 
		
		More precisely, consider a $T$-equivariant étale affine chart $g \colon U \to X$ and the perfect obstruction theory $\varphi_U \colon E_U^\bullet \to \LL_U$ on $U$ from the definition of the APOT $\varphi$. Consider the following diagram where all the squares are fibered
		\begin{center}
			\begin{equation}
				\label{cubediagram}
				\begin{tikzcd}
					&  &                                                         & W_A \arrow[rr, hook] \arrow[ld, "g_W"] \arrow[dd, "a_W", near end] &                   & U_A \arrow[dd, "a_U"] \arrow[ld, "g_A"] \\
					A^T \arrow[rr, "\iota^\prime", hook] &  & F_A \arrow[rr, "\iota_F", near end, hook, crossing over] \arrow[dd, "a|_{F_A}"'] &                                            & A \arrow[dd, "a", near end, crossing over] &                                  \\
					&  &                                                         & U^T \arrow[rr, hook] \arrow[ld, "g^T"]                  &                   & U \arrow[ld, "g"]                     \\
					&  & X^T \arrow[rr, "\iota", hook]                           &                                            & X.                 &                                 
				\end{tikzcd}
			\end{equation}
		\end{center}
		
		Here, $U_A \coloneqq A \times_X U $ and $W_A \coloneqq F_A \times_A U_A$, both spaces are affine as $A$ and $U$ are affine. The fiber product $U^T \coloneqq X^T \times_X U$ is also the fixed locus under the torus action on $U$ as $g$ is étale. 
		
		By construction, the APOT $\varphi|_{X^T}^{\fix}$ on $X^T$ consists of an affine étale chart ${g^T \colon U^T \to X^T}$ and a POT 
		$$\varphi_{U^T} \coloneqq \varphi_U|_{U^T}^{\fix} \colon E_U^{\bullet}|_{U^T}^{\fix} \to \LL_{U^T}$$
		on $U^T$ for each affine étale chart $U \to X$ and POT $\varphi_U$. 
		
		A pullback APOT of $\varphi|_{X^T}^{\fix}$ via $a|_{F_A}$ to $F_A$ then consists of perfect obstruction theories 
		\begin{equation}
			\label{pot1}
			a_W^* \varphi_{U^T} \colon  a_W^*(E_U^\bullet|_{U^T}^{\fix}) \oplus \Omega_{W_A/U^T} \to \LL_{W_A}
		\end{equation}
		on $W_A$ for each affine étale chart $U^T \to X^T$ (see diagram $(\ref{cubediagram})$ for the notation).
		
		By commutativity of $(\ref{cubediagram})$, $$(a_U^* \varphi_U)|_{W_A} \simeq a_W^*(E_U^\bullet|_{U^T}).$$
		Hence, $(a_U^* \varphi_U)|_{W_A}$ is trivial on the torus orbit, an we can take fixed parts. Moreover,
		$$(a_U^* \varphi_U)|_{W_A}^{\fix} \simeq a_W^*(E_U^\bullet|_{U^T})^{\fix} \simeq a_W^*(E_U^\bullet|_{U^T}^{\fix}). $$ Therefore, the POT $a_W^* \varphi_{U^T}$ is isomorphic to the POT 
		\begin{equation}
			\label{pot2}
			(a_U^* \varphi_U)|_{W_A}^{\fix} \colon 	(a_U^* E^{\bullet}_U)|_{W_A}^{\fix} \oplus \Omega_{W_A/U^T} \to \LL_{W_A}.
		\end{equation}
		
		Recall that $\varphi_A$ is the POT obtained as a pullback of the APOT $\varphi$ via $a$. By assumption, it is of the form 
		$$\varphi_A \colon E_A^{\bullet} \oplus \Omega_{A/X} \to \LL_A.$$
		We also have isomorphisms $g_A^* E_{A}^{\bullet} \simeq a_U^* E_U^{\bullet}$ and  $g_A^* \varphi_A \simeq a_U^* \varphi_U$.
		
		Hence, the POT $(\ref{pot2})$ is isomorphic to the POT $(g_A^* \varphi_A)|_{W_A}^{\fix}$ which is isomorphic to the POT $$(g_W^* (\varphi_A|_{F_A}))^{\fix} \simeq g_W^*((\varphi_A|_{F_A})^{\fix}).$$
		
		So, we see that each étale local POT $(\ref{pot1})$ of the construction of the APOT $(a|_{F_A})^*\varphi|_{X^T}^{\fix}$ is the pullback to the corresponding chart of the (global) POT on $F_A$
		
		$$(\varphi_A|_{F_A})^{\fix} \colon (E_A^{\bullet})|_{F_A}^{\fix} \oplus \Omega_{F_A/ X^T} \to \LL_{F_A}.$$
		This exactly means that the APOT $(a|_{F_A})^*\varphi|_{X^T}^{\fix}$ coincides with the POT $(\varphi_A|_{F_A})^{\fix}$.
		%
		%
	\end{proof}
	We are now ready to prove the torus localization formula.
	
	\begin{theor}
		\label{localization}
		Let $X$ be a scheme acted on by a torus $T$, and $\iota \colon X^T \into X$ be the embedding of the fixed locus. Let $\varphi$ be a $T$-equivariant almost perfect obstruction theory $(\ref{toursemipot})$ satisfying Assumption~$\ref{toruslocassumpt}$. Then
		$$[X]^{\vir} = \iota_* \frac{[X^T]^{\vir}}{e(N^{\vir})}. $$
		Here, $[X^T]^{\vir}$ is obtained from the APOT (\ref{apotfixloc}) on the fixed locus and $N^{\vir} \coloneqq E_{\bullet}|_{X_T}^{mov} \in K^{0, T}(X). $ 
	\end{theor}
	
	\begin{proof}
		We first apply Graber-Pandharipande virtual torus localization formula \cite{graber1997localization} to the perfect obstruction theory $\varphi_A$ on $A$. Denote by $\iota_A \colon A^T \into A$ the embedding of the fixed locus and note that by assumption, 
		$$N_{A^T}^{\vir} = (a^*E_{\bullet} + T_{A/X})|_{A^T}^{mov} \in K^{0, T}(A).$$
		Therefore,
		\begin{equation}
			\label{agrploc}
			[A]^{\vir} = \iota_{A*} \frac{[A^T]^{\vir}}{e(N^{\vir}_{A^T})} = \iota_{A*} \frac{[A^T]^{\vir}}{e(T_{A/X}|_{A^T}^{mov})e(a^*E_\bullet|_{A^T}^{mov})}
		\end{equation}
		Put $F_A \coloneqq a^{-1}(X^T)$ so that $\iota_A$ factors as
		\begin{center}
			\begin{tikzcd}
				A^T \arrow[r, "\iota^\prime", hook] & F_A \arrow[r, "\iota_F", hook] \arrow[d, "a|_{F_A}"'] & A \arrow[d, "a"] \\
				& X^T \arrow[r, "\iota", hook]                          & X               
			\end{tikzcd}
		\end{center}
		with $\iota^\prime$ and $\iota_F$ being closed embeddings. Note that
		\begin{equation}
			a^*E_\bullet|_{A^T}^{mov} = \iota^{\prime*} (a|_{F_A})^*(E_\bullet|_{X^T}^{mov})
		\end{equation} 
		hence, the right hand side of $(\ref{agrploc})$ can be rewritten as 
		\begin{equation}
			\label{l4}
			(\iota_F)_*\left(\iota^\prime_* \frac{[A^T]^{\vir}}{e(T_{A/X}|_{A^T}^{mov})\iota^{\prime*} (a|_{F_A})^*e(E_\bullet|_{X^T}^{mov})}\right) =
			(\iota_F)_*\left(\frac{1}{(a|_{F_A})^*e(E_\bullet|_{X^T}^{mov})} \iota^\prime_* \frac{[A^T]^{\vir}}{e(T_{A/X}|_{A^T}^{mov})}\right)
		\end{equation}
		where we use the projection formula for $\iota^\prime$.
		
		We next apply the localization formula to the perfect obstruction theory on $F_A$ obtained as a pullback via $a|_{F_A}$ of the APOT $\varphi|_{X^T}^{\fix}$ on $X^T$. It is a perfect obstruction theory by Lemma~\ref{fixpotpullback}.
		Note that $N^{\vir}_{F_A} = T_{A/X}|_{A^T}^{mov}$ by construction of the APOT on $X^T$. Hence,
		\begin{equation}
			[F_A]^{\vir} = \iota^\prime_* \frac{[A^T]^{\vir}}{e(T_{A/X}|_{A^T}^{mov}).}
		\end{equation}
		Hence, we can rewrite the right hand side of $(\ref{l4})$ as
		\begin{equation}
			(\iota_F)_* \frac{[F_A]^{\vir}}{(a|_{F_A})^*e(E_\bullet|_{X^T}^{mov})} = 
			(\iota_F)_* \frac{(a|_{F_A})^*[X^T]^{\vir}}{(a|_{F_A})^*e(E_\bullet|_{X^T}^{mov})} = 
			a^* \iota_* \frac{[X^T]^{\vir}}{e(E_\bullet|_{X^T}^{mov})}.
		\end{equation}
		Putting all together, we get
		$$a^*[X]^{\vir} = [A]^{\vir} = a^* \iota_* \frac{[X^T]^{\vir}}{e(E_\bullet|_{X^T}^{mov})}.$$
		Now by an equivariant extension of \cite[\href{https://stacks.math.columbia.edu/tag/0GUB}{Tag 0GUB}]{stacks-project}, $a^* \colon A_*^T(X) \simto A_*^T(A)$ is an isomorphism. Hence, the desired formula follows.
	\end{proof}
	
	\section{Virtual invariants and their computation in the toric case} \label{s6}
	\subsection{Virtual invariants}
	The virtual cycle $(\ref{virtclass})$ on $\Q$ 
	$$[\Q]^{vir} \in A_0(\Q)$$ allows one to define virtual invariants of Quot schemes of points corresponding to a locally free sheaf by
	\begin{equation}
		\label{ninv}
		\DT^n_{Y, F} \coloneqq \deg [\Q]^{vir}. 
	\end{equation}
	We consider the generating series:
	\begin{equation} \label{genser}
		\DT_{Y, F}(q) \coloneqq 1+ \sum_{n \geq 1} q^n \DT^n_{F, Y}. 
	\end{equation}
	The goal of this section is to obtain an explicit formula for $(\ref{genser})$ in the case when $Y$ is toric and $F$ has a torus-equivariant structure using the virtual torus localization (Theorem~$\ref{localization}$) and Siebert's formula for almost perfect obstruction theories (Theorem~$\ref{apotpotsiebert})$.
	
	\subsection{Applying torus localization, Siebert formula and Jouanolou trick to $\Q$.} \label{readyforloc}
	
	We aim to show that in the case when the $3$-fold $Y$ is a toric variety and the locally free sheaf $F$ has a torus-equivariant structure, we can perform the torus-equivariant Jouanolou trick (Theorem~\ref{tequiv}) on $\Q$ and apply the virtual localization formula (Theorem $\ref{localization}$) and Siebert formula for APOT (Theorem ${\ref{apotpotsiebert}}$) to the almost perfect obstruction theory from Theorem $\ref{quotsemipot}$.
	\vspace{6pt}
	
	Let $Y$ be a smooth toric $3$-fold with $\T = (\mathbb{C}^*)^3$-action
	$\sigma_Y \colon Y \times \T \to Y $
	and $F$ a locally free sheaf with a torus-equivariant structure
	\begin{equation}
		\label{equivstr}
		\mu \colon \sigma^*F \simeq pr_Y^*F, 
	\end{equation}
	where $pr_Y \colon Y \times \T \to Y$ is the projection morphism.
	\vspace{5pt}
	
	By \cite{Fasola_2021, newatclricolfi}, the torus action on $Y$ induces an action on $\Q$. We will need the explicit construction. For a scheme $S$, denote by $p_S \colon S \times Y \to S$ the projection to $S$ and by $F_S$ the pullback of $F$ to $S \times Y$ via the second projection. For an element $t \in \T(S)$, which we view as a map $t \colon S \to \T$, denote by $\sigma_t \colon S \times Y \to Y$ the composition $\sigma_Y \circ (t \times \id_Y)$ and by
	$$\alpha_t \colon S \times Y \to S \times Y$$
	the isomorphism given by $(p_S, \sigma_t)$.
	We denote by $$\mu_t \colon \alpha_t^* F_S \simeq F_S$$ the pullback $(t \times \id)^*\mu$ of the equivariant structure (\ref{equivstr}).
	The map
	\begin{equation}
		\sigma_Q(S) \colon	\T(S) \times \Q(S) \to \Q(S)
	\end{equation}
	is defined by sending a pair $(t, [F_S \xrightarrowdbl{v} Q_S]) \in \Q(S)$ to the composition 
	$$[F_S \xrightarrow{\mu_t^{-1}} \alpha_t^*F_S \xrightarrowdbl{\aaa_t^*v} \alpha_t^*Q_S].$$

	As $Y$ is toric, there exists a $\T$-equivariant ample line bundle on $Y$. We will denote it by $\oo_Y(1)$. By the Proof of \cite[Theorem 2.2.4]{huybrechts_lehn}, for $m \gg 0$ there exists a closed embedding 
	\begin{equation}
		\label{embgr}
		\iota \colon \Q \into \Grass(n, \HH^0(F(m)))
	\end{equation}
	sending the quotient $[F \onto Q]$ to the quotient $[\HH^0(Y, F(m)) \onto \HH^0(Y, Q(m))]$. Over a scheme $S$, it maps a quotient $[F_S \xrightarrowdbl{v} Q_S]$ to the quotient $$[p_{S*}F_S(m) \xrightarrowdbl{p_{S*}v(m)} p_{S*}Q_S(m)],$$ where $p_S \colon S \times Y \to S$ is the projection morphism.
	
	We also have an induced $\T$-action on the Grassmannian $\Gr \coloneqq \Grass(n, \HH^0(F(m)))$. 
	For a scheme $S$ define the morphism
	$$\sigma_{\Gr}(S) \colon \T(S) \times \Gr(S) \to \Gr(S) $$
	sending a pair $(t, [p_{S*}F_S(m) \xrightarrowdbl{u} \mathcal{E}])$ to the composition  
	$$[p_{S*}F_S(m) \xrightarrow[\sim]{p_{S*}(\mu_t^{-1})} p_{S*}\aaa_t^* F_S(m) \simeq p_{S*}\aaa_{t *}\aaa_t^* F_S(m) \simeq p_{S*}F_S(m) \xrightarrowdbl{u} \mathcal{E}]. $$ 
	In the first isomorphism by abuse of notation the equivariant structure on $\oo_Y(1)$ is also involved. The second isomorphism is by definition of $\aaa_t$ as $p_S = p_S \circ \aaa_t$ and the third isomorphism is constructed using the inverse of the adjunction unit $F_S \simeq \aaa_{t *}\aaa_t^* F_S$.
	
	\begin{prop}
		\label{equivemb}
		The closed embedding $(\ref{embgr})$ is $\T$-equivariant. 
	\end{prop}
	\begin{proof}
		We need to check that for each scheme $S$, the following identity hold
		\begin{equation}
			\label{tocheckequivemb}
			\iota(S) \circ \sigma_Q(S) = \sigma_{\Gr}(S) \circ (\id_{\T(S)} \times \iota(S)).
		\end{equation}
		A pair $(t, [F_S \xrightarrowdbl{u} Q_S]) \in \T(S) \times \Q(S)$ is mapped via the left hand side of (\ref{tocheckequivemb}) to the composition quotient
		\begin{equation}
			\label{lhs}
			[p_{S*}F_S(m) \xrightarrow[\sim]{p_{S*}(\mu_t^{-1})} p_{S *} \aaa_t^* F_S(m) \xrightarrowdbl{p_{S*}\aaa_t^*(u(m))} p_{S*}\aaa_t^*Q_S(m) ]. 
		\end{equation}
		Via the right hand side of (\ref{tocheckequivemb}) the pair $(t, [F_S \xrightarrowdbl{u} Q_S])$ is mapped to
		\begin{align*}
			&\sigma_{\Gr}(t, [p_{S*}F_S(m) \onto p_{S*}Q_S(m)]) \\ 
			=& [p_{S*}F_S(m) \xrightarrow[\sim]{p_{S*}(\mu_t^{-1})} p_{S *} \aaa_t^* F_S(m) = p_{S *} \aaa_{t*} \aaa_t^* F_S(m) \simto p_{S*}F_S(m) \xrightarrowdbl{p_{S*}(u(m))} p_{S*}\aaa_t^*Q_S(m)].
		\end{align*}
		It agrees with (\ref{lhs}) as the diagram 
		
		\begin{center}
			\begin{tikzcd}
				p_{S*}F_S(m) \arrow[r, "p_{S*}(u(m))", two heads] \arrow[d, "adj_{F_S}"', "\simeq"]    & p_{S*}Q_S(m) \arrow[d, "\simeq"', "adj_{Q_S}"]                    \\
				p_{S*}\aaa_{t*}\aaa_{t}^*F_S(m) \arrow[r, two heads] \arrow[d, equals] & p_{S*}\aaa_{t*}\aaa_{t}^*Q_S(m) \arrow[d, equals] \\
				p_{S*}\aaa_{t}^*F_S(m) \arrow[r, two heads]                    & p_{S*}\aaa_{t}^*Q_S(m)                   
			\end{tikzcd}
		\end{center}
		is commutative by the functoriality of the adjunction morphism.
		%
		%
		%
		%
	\end{proof}
	
	\begin{rem}
		\label{grquotient}
		The torus action morphism $\sigma_{\Gr} \colon \T \times \Gr \to \Gr$ described in Proposition~\ref{equivemb} is given by an element in $\Gr(\T \times \Gr)$, which is the image of the identity map $\id \in \T \times \Gr(\T \times \Gr)$ under $\sigma_{\Gr}(\T \times \Gr)$. It is given by the quotient on $\T \times \Gr$ obtained as a composition
		\begin{equation}
			\label{sgasquotient}
			p_{*}F_{\T \times \Gr}(m) \xrightarrow[\sim]{p_{*}(\mu_{\id}^{-1})} p_{*}\aaa^* F_{\T \times \Gr}(m) \simeq p_{*}\aaa_{*}\aaa^* F_{\T \times \Gr}(m) \simeq p_{*}F_{\T \times \Gr}(m) \xrightarrowdbl{q^*u} q^* \qq_{\Gr} .
		\end{equation}
		Here, $\aaa \coloneqq \aaa_{\id}$ and $p \coloneqq p_{\T \times \Gr}$ in the notation of the proof of previous Proposition. The last morphism is the pullback of the universal quotient $u$ on $\Gr$ via the second projection $q \colon \T \times \Gr \to \Gr$.
		%
	\end{rem}
	\begin{lem}
		\label{grunivequiv}
		The universal sequence 
		$$0 \to \sss_{\Gr} \to H^0(Y, F(m)) \otimes \oo_{\Gr} \to \qq_{\Gr} \to 0  $$
		on $\Grass(n, \HH^0(F(m)))$ carries a $\T$-equivariant structure.
	\end{lem}
	\begin{proof}
		By Remark \ref{grquotient}, the action morphism $\sigma_{\Gr}$ is given via the quotient $(\ref{sgasquotient})$. On the other hand, by the universal property, it is given by the pullback of the universal quotient $u$ via $\sigma_{\Gr}$. Hence, these two quotients are isomorphic, and we have the commutative diagram where we use the notation from Remark~\ref{grquotient}
		\begin{center}
			\begin{tikzcd}
				p_*F_{\T \times \Gr} \arrow[r, "\sigma_{\Gr}^*u"] \arrow[d, "\simeq", "a"'] & \sigma_{\Gr}^*\qq_{\Gr} \arrow[d, "\simeq"] \\
				p_*F_{\T \times \Gr} \arrow[r, "q^*u"]                                & q^*\qq_{\Gr}                             
			\end{tikzcd}
		\end{center}
		giving the $\T$-equivariant structure on the universal quotient and hence then also on the whole universal sequence. Here, $a$ denotes the composition of the isomophisms from $(\ref{sgasquotient})$.
	\end{proof}
	
	\begin{prop}
		\label{resolutpropsatisf}
		The Quot scheme $\Q$ and the Grassmannian $\Gr \coloneqq \Grass(n, \HH^0(F(m)))$ satisfy the $\T$-equivariant resolution property.
	\end{prop}
	\begin{proof}
		We need to check that for each $\T$-equivariant coherent sheaf $\F$ on $\Q$ or on $\Grass(n, \HH^0(F(m)))$ there exists a $\T$-equivariant (possibly infinite) resolution of locally free sheaves. 
		
		By Lemma \ref{grunivequiv}, 
		$\oo_{\Gr}(1) = \det \qq_{\Gr}$ is $\T$-equivariant. Hence, its restriction to $\Q$ is also $\T$-equivariant and very ample on $\Q$. It will be denoted by $\oo_{\Q}(1)$.
		The proof now goes as in the non-equivariant case in \cite{griffiths1978principles} using that for $k \gg 0$ the sheaf $\F(k)$ is globally generated. It remains to check that the natural evaluation map
		\begin{equation}
			\label{evmap}
			\HH^0(\Q, \F(k)) \otimes \oo_{\Q} \to \F(k)
		\end{equation} 
		for $\Q$ and a similar map for $\Grass(n, \HH^0(F(m)))$ are $\T$-equivariant. As the natural evaluation map is the counit of the adjunction with respect to the map to the point, it follows from adjunction that the needed maps are $\T$-equivariant.
	\end{proof}
	
	\begin{prop}[{\cite[Section 9.1]{Fasola_2021}}]
		\label{torunivseq}
		A $\T$-equivariant structure on $F$ induces a canonical $\T$-equivariant structure on the universal short exact sequence 
		$$0 \to \sss \to q^*F \to \qq \to 0, $$ on $\Q \times Y$ where $q \colon \Q \times Y \to Y$ is the second projection.
	\end{prop}	
	
	\begin{prop} \label{equiv}
		The almost perfect obstruction theory $\varphi$ on $\Q$ constructed in Section~$\ref{globally}$ admits a canonical $\T$-equivariant structure.
	\end{prop}
	\begin{proof}
		We first note that the  $\T$-equivariant structure on the universal sequence of $\Q$ from Proposition $\ref{torunivseq}$ makes the standard obstruction theory on $\Q$
		\begin{equation}
			\label{equivderobstr}
			\varphi^{\der} \colon (\rpl \rh(\sss, \qq))^\vee \to \LL_\Q
		\end{equation}
		discussed in Section $\ref{derobstr}$  to be $\T$-equivariant. The argument is as in the proof of \cite[Proposition 9.2]{Fasola_2021} and follows as $\varphi^{\der}$ is constructed from the reduced Atiyah class which can be viewed as a $\T$-equivariant extension.
		
		By \cite{alperhallrydh}, there exists an affine étale $\T$-equivariant cover
		$$\{U_i \to \Q\}_{i \in I}.$$
		
		By construction in Section~\ref{locally} together with Proposition~\ref{torunivseq}, the diagram~(\ref{twodisttr}) for $U = U_i$ is in $\DD^{\T}(U_i)$. Therefore for each $i \in I$, $\psi_i = h^1(\gamma_i)$ is in $\Coh^{\T}(U_i)$ and the morphism $\varphi_i$ is in $\DD^{\T}(U_i)$, as $(\ref{equivderobstr})$ is equivariant. By construction, for each $i, j \in I$, the isomorphisms $\eta_{ij}$ on the double intersections and the corresponding compatibilities are in $\DD^{\T}(\uij)$
%
	\end{proof}
	
	\begin{prop}
		\label{pullbackispot}
		For any $\T$-equivariant smooth morphism $a \colon A \to \Q$ with $A$ affine, a pullback of the $\T$-equivariant almost perfect obstruction theory $\varphi$ discussed in Proposition $\ref{equiv}$ is a $\T$-equivariant perfect obstruction theory of the form 
		$${\varphi}_A \colon E^\bullet_A \oplus \Omega_{a} \to \LL_A$$
		with $E_A^\bullet \in \DD^{\T}(A)$.
	\end{prop}
	\begin{proof}
		A $\T$-equivariant pullback of the $\T$-equivariant almost perfect obstruction theory $$\varphi = \{U_i, \varphi \colon E^\bullet_i \to \LL_{U_i}\}$$ via $a$ consists of an étale affine $\T$-equivariant cover 
		$$\{\tilde{U}_i \to A\}_{i \in I} \quad \text{with} \quad \tilde{U}_i = U_i \times_\Q A$$
		and $\T$- equivariant perfect obstruction theories 
		$$\tilde{\varphi}_i \colon a_i^*E^\bullet_{U_i} \oplus \Omega_{a_i} \to \LL_{\tilde{U_i}}$$
		where $a_i \colon \tilde{U}_i \to U_i$ are the corresponding projections.
		
		Define the complex $E_A^\bullet \in \DD^{\T}(A)$ constructed on $A$ (which is affine) as in Section $\ref{locally}$. It comes with the composition 
		$$E^\bullet_A \to (R\sheafhom_{\pi_A}(\sss_A, \qq_A))^\vee \to a^*\LL_\Q$$
		where $\pi_A \colon Y \times A \to A$ is the second projection and $\sss_A$ and $\qq_A$ are the pullbacks of the universal kernel and cokernel, respectively.
		As in the proof of Proposition $\ref{potpullback}$, we have a $\T$-equivariant splitting $$\LL_A \simeq a^*\LL_\Q \oplus \Omega_a$$ and hence the $\T$-equivariant morphism
		$${\varphi}_A \colon E^\bullet_A \oplus \Omega_{a} \to \LL_A $$ can be defined.
		Note that by construction for each $i \in I$ we have $${\varphi}_A|_{\tilde{U_i}} = \tilde{\varphi}_i$$ which shows that the pullback of $\varphi$ via $a$ is a $\T$-equivariant perfect obstruction theory given by ${\varphi}_A$.
	\end{proof}
	
	We summarize this section with the following Corollary.
	
	\begin{cor}
		\label{summary}
		For $Y$ a toric $3$-fold and $F$ a torus-equivariant locally free sheaf on it, the torus-equivariant Jouanolou trick (Theorem $\ref{tequiv}$) can be performed on $\Q$ and on $\Grass(n, H^0(F(m))$.
		Moreover, the $\T$-equivariant Siebert formula for APOT (Theorem~${\ref{apotpotsiebert}}$) and the virtual localization formula (Theorem~$\ref{localization}$) can be applied to the APOT on $\Q$ from Theorem~$\ref{quotsemipot}$.
	\end{cor}
	\begin{proof}
		By Proposition \ref{resolutpropsatisf}, the $\T$-equivariant resolution property is satisfied for $\Q$ and $\Gr$, and hence the $\T$-equivariant Jouanolou trick (Therem~\ref{tequiv}) can be performed on these spaces.
		
		By Proposition $\ref{equiv}$, the APOT $\varphi$ on $\Q$ is $\T$-equivariant. We first check Assumption~$\ref{siebassumpt}$. 
		By Remark $\ref{quotglkthclass}$, $\varphi$ has a global class $E_\bullet$ in $K^{0, \T}(\Q)$ and hence the first condition in Assumption~$\ref{siebassumpt}$ is satisfied. We next check the second condition. By Proposition $\ref{equivemb}$, there is a $\T$-equivariant closed embedding $\Q \into \Gr$ into the Grassmanian. By the Jouanolou trick (Theorem~\ref{tequiv}) for $\Gr$, 
		there exists a $\T$-equivariant affine bundle $B \to \Gr$ with $B$ affine. Hence, the base change $A \coloneqq \Q \times_{\Gr} B \to \Q$ is also $\T$-equivariant affine with $A$ affine (being a closed subscheme of an affine scheme $B$). Therefore by Proposition~$\ref{pullbackispot}$, $a^* \varphi$ is a $\T$-equivariant perfect obstruction theory of the form $$\varphi_A \colon E^\bullet_A \oplus \Omega_{A/ \Q} \to \LL_A,$$ where $E^\bullet_A$ is constructed as in Section~\ref{locally}. By construction,  $(E^\bullet_A)^{\vee} = a^*E_\bullet$ in $K^{0, \T}(\Q)$. Hence, the second condition in Assumption~$\ref{siebassumpt}$ to apply the Siebert's formula for APOT is satisfied. We have actually also shown that the Assumption~\ref{toruslocassumpt} to apply the localization formula for APOT is also satisfied.
		%
		%
	\end{proof}
	\begin{cor}
		\label{quotvirtclexplicit}
		The virtual class on $\Q$ induced by the almost perfect obstruction theory $\varphi$ from Theorem~\ref{quotsemipot} satisfies the relation
		\begin{equation}
			[\Q]^{\vir} = [s(E_\bullet) \cap c_F(\Q)]_0 \in A_0(\Q),
		\end{equation}
		where $E_\bullet = \rph(q^*F, q^*F) - \rph(\sss, \sss) \in K^{0}(\Q)$ is the global class (\ref{kth1}). 
	\end{cor}
	\begin{proof}
		Apply Siebert's formula for APOT (Theorem~\ref{apotpotsiebert}) to the APOT on $\Q$ constructed in Theorem~\ref{quotsemipot}.
	\end{proof}
	\begin{rem}
		\label{virtclcoinc}
		By Corollary~\ref{quotvirtclexplicit}, $[\Q]^{\vir}$ corresponding to the APOT from Theorem~\ref{quotsemipot} agrees with the virtual class on $\Q$ in cases when it was already constructed.
	\end{rem}
	\subsection{Comparing two almost perfect obstruction theories on $\Q^{\T}$}
	
	Recall that the fixed locus $\Q^\T$ carries a canonical almost perfect obstruction theory $\varphi|_{\Q^{\T}}^{\fix}$ induced by $\varphi$, see~$(\ref{apotfixloc}).$
	%
	It gives rise to the virtual class
	\begin{equation}
		\label{fixlocvirtclsemipot}
		[\Q^\T]^{\vir} \in A_*^{\T}(\Q^\T).
	\end{equation}
	It has also a global class in $K^{0, \T}(\Q^\T)$ in the sense of Convention~\ref{K^0class}
	\begin{equation}
		\label{fixlocglkth}
		E_\bullet^\T \coloneqq E_\bullet|_{\Q^\T}^{\fix},
	\end{equation}
	where $E_\bullet$ is the $\T$-equivariant lift of the class $(\ref{kth2})$. 
	
	\begin{lem}
		\label{applysiebfixapot}
		The $\T$-equivariant Siebert formula for APOT (Theorem ${\ref{apotpotsiebert}}$) can be applied to the almost perfect obstruction theory $\varphi|_{\Q^{\T}}^{\fix}$ on $\Q^{\T}$.
	\end{lem}
	\begin{proof} 
		By $(\ref{fixlocglkth})$, the first condition in Assumption $\ref{siebassumpt}$ is satisfied. We use the notation from the proof of Corollary $\ref{summary}$ and the proof itself. We have a $\T$-equivariant embedding of $\Q^{\T}$ into the Grassmannian $\Gr$. Moreover, there is a $\T$-equivariant affine bundle $b \colon B \to \Gr$ with $B$ affine. We denote its restriction to $\Q^{\T}$ by $F_A$. Everything is depicted in the following diagram where all squares are fibered
		\begin{center}
			\begin{tikzcd}
				F_A \arrow[d, "a|_{F_A}"] \arrow[r, hook] & A \arrow[d, "a"] \arrow[r, hook] & B \arrow[d, "b"] \\
				\Q^{\T} \arrow[r, hook]                   & \Q \arrow[r, hook]               & \Gr.           
			\end{tikzcd}
		\end{center}
		As the diagram is fibered, $F_A \into A$ is a closed embedding. As $A$ is affine, $F_A$ is also affine. The restriction $a|_{F_A}$ is also affine. By Lemma $\ref{fixpotpullback}$, the pullback APOT of $\varphi|_{\Q^{\T}}^{\fix}$ via $a|_{F_A}$ is a perfect obstruction theory. By construction, the corresponding class in $K^{0, \T}(\Q^\T)$ is 
		$(a|_{F_A})^*E^{\T}_\bullet + T_{F_A/\Q^{\T}}$, and Assumption~\ref{siebassumpt} is satisfied.
	\end{proof}
	
	To describe the second almost perfect obstruction theory on $\Q^{\T}$ we first recall from \cite{Fasola_2021} the description of the fixed locus of the Quot scheme $\Q^\T$ under the action of the torus $\T \simeq (\mathbb{C}^*)^3$. Denote by ${\bf n}$ a generic tuple $\{n_{\alpha}| \alpha \in \Delta(Y)\}$ of non-negative integers, where $\Delta(Y)$ is the set of vertices of the polytope corresponding to the toric variety $Y$. Set $|{\bf n}| = \sum_{\alpha \in \Delta(Y)} n_{\alpha}$. By $$Y = \bigcup_{\aaa \in \Delta(Y)} U_\aaa$$ we denote the standard cover of affine charts of $Y$.
	
	\begin{lem}[{\cite[Lemma 9.4]{Fasola_2021}}]
		\label{fixed}
		There is a scheme-theoretic identity
		$$\Q^\T = \disj_{|{\bf n}| = n} {\prodd}_{\alpha \in \Delta(Y)} \Quot_{U_{\alpha}}(F|_{U_{\alpha}}, n_{\alpha})^\T . $$
	\end{lem}
	%

	
	For a fixed tuple $\n = \{ n_\aaa\}_{\aaa \in \Delta(Y)}$ and $\aaa \in \Delta(Y)$ the open embedding 
	\begin{equation}
		\label{indsemipot}
		\Quot_{\ua}(F|_{\ua}, n_\aaa) \into \Quot_Y(F, n_\aaa)
	\end{equation}
	equips $\Quot_{\ua}(F|_{\ua}, n_\aaa)$ and hence also $\Quot_{\ua}(F|_{\ua}, n_\aaa)^\T$ with a $\T$-equivariant almost perfect obstruction theory. By Lemma $\ref{fixed}$, it induces an almost perfect obstruction theory on $\Q^\T$ considering on each component the product of the almost perfect obstruction theories induced by $(\ref{indsemipot})$. We denote this APOT by $\varphi_{Q^\T}^\prime$. Denote the resulting virtual class by 
	\begin{equation}
		\label{primevirtclass}
		[\Q^{\T}]^{\vir \prime} \in A_*^{\T}(\Q^\T).
	\end{equation}
	Note that 
	\begin{equation}
		\label{primevc}
		[\Q^{\T}]^{\vir \prime} = \sum_{\n} \prod_{\aaa} [\Quot_{\ua}(F|_{\ua}, n_\aaa)^{\T}]^{\vir}
	\end{equation}
	Denote by 
	$$0 \to \sss_\aaa \to F_\aaa \to \qq_\aaa \to 0 $$
	the universal sequence of $\Quot_{\ua}(F|_{\ua}, n_\aaa)^\T$ and by 
	$$\pi_\alpha \colon \ua \times \Quot_{\ua}(F|_{\ua}, n_\aaa)^\T \to \Quot_{\ua}(F|_{\ua}, n_\aaa)^\T$$ 
	the second projection. Then as in Remark $\ref{quotglkthclass}$, the global class in $K^{0, \T}(\Quot_{\ua}(F|_{\ua}, n_\aaa)^\T)$ of the almost perfect obstruction theory on $\Quot_{\ua}(F|_{\ua}, n_\aaa)^\T$ induced by the embedding $(\ref{indsemipot})$ is the fixed part of
	\begin{equation}
		E_{\alpha \bullet} \coloneqq R \sheafhom_{\pi_\aaa}(F_\alpha, \qq_\alpha) + R \sheafhom_{\pi_\aaa}(\qq_\alpha, F_\alpha) - R \sheafhom_{\pi_\aaa}(\qq_\alpha, \qq_\alpha).
	\end{equation}
	Denote by 
	$$p_\aaa \colon \Q^\T_\n \to \Quot_{\ua}(F|_{\ua}, n_\aaa)^\T$$
	the projection onto the $\alpha$-part where $\Q^\T_\n$ denotes the component of the fixed locus corresponding to the tuple $\n$.
	Then the induced almost perfect obstruction theory $\varphi_{Q^\T}^\prime$ on $\Q^\T$ 
	has a global K-theory class
	\begin{equation}
		E_\bullet^\prime \coloneqq \sum_{\n} \left(\sum_{\aaa \in \Delta(Y)} p_\aaa^* E_{\alpha \bullet} \right)^{\fix} \in K^{0, \T}(\Q^\T)
	\end{equation}
	Its K-theoretic virtual normal bundle is
	\begin{equation}
		\label{primevnb}
		N^{\prime \vir} \coloneqq \sum_{\n} \left(\sum_{\aaa \in \Delta(Y)} p_\aaa^* E_{\alpha \bullet} \right)^{mov} = \sum_{\n} \sum_{\aaa \in \Delta(Y)} p_\aaa^* N_{\aaa}^{\vir}.
	\end{equation}
	Here, $N_{\aaa}^{\vir} \coloneqq E_{\alpha \bullet}^{mov} $ is the K-theoretic virtual normal bundle corresponding to the APOT on $\Quot_{\ua}(F|_{\ua}, n_\aaa)$.
	\begin{lem}
		\label{applytsiebapot}
		The $\T$-equivariant Siebert formula for APOT (Theorem~\ref{apotpotsiebert}) can be applied to the almost perfect obstruction theory $\varphi_{Q^\T}^\prime$ on $\Q^\T$.
	\end{lem}
	\begin{proof}
		By the discussion above, $\varphi_{Q^\T}^\prime$ has a global class in $K^{0, \T}(\Q^\T)$ and the first condition in Assumption \ref{siebassumpt} is satisfied. For the second condition, consider the $\T$-equivariant embedding given on the component $\n$ in the decomposition of the fixed locus from Lemma~$\ref{fixed}$ as a composition 
		$$\prod_{\aaa}\Quot_{\ua}(F|_{\ua}, n_\aaa)^\T \into \prod_{\aaa}\Quot_Y(F, n_\aaa)^\T \into \prod_{\aaa}\Quot_Y(F, n_\aaa) \into \prod_{\aaa}\Grass(\HH^0(F(m)), n_\aaa).$$
		We next perform the torus-equivariant Jouanolou trick (Theorem~\ref{tequiv}) componentwise on the product of the Grassmannians and complete the proof as in the one of Lemma \ref{applysiebfixapot}.  
	\end{proof}
	\begin{prop}
		\label{kupzeroeq1}
		There is an equality in $K^{0, \T}(\Q^\T_\n)$
		\begin{equation}
			E_\bullet|_{\Q^\T_\n} = \sum_{\aaa \in \Delta(Y)} p_\aaa^* E_{\alpha \bullet}.
		\end{equation}
	\end{prop}
	Recall that $E_\bullet$ is the $\T$-equivariant lift of the class defined in Remark $\ref{quotglkthclass}.$
	\begin{proof}
		Just in this proof for appearance convenience we will by abuse of notation denote by 
		$$0 \to \sss \to q^*F \to \qq \to 0 $$
		the universal sequence restricted to the $\n$-th component of the fixed locus $\Q^\T$ and by $q$ and $\pi$ the projections from $Y \times \Q^\T_\n$ to the first and the second factor, respectively.
		It then suffices to prove the following identities
		\begin{align}
			\rph(F, \qq) = \sum_{\aaa \in \Delta(Y)} p_\aaa^* R\sheafhom_{\pi_\aaa}(F_\aaa, \qq_\aaa),\\
			\rph(\qq, F) = \sum_{\aaa \in \Delta(Y)} p_\aaa^* R\sheafhom_{\pi_\aaa}(\qq_\aaa, F_\aaa),\\
			\rph(\qq, \qq) = \sum_{\aaa \in \Delta(Y)} p_\aaa^* R\sheafhom_{\pi_\aaa}(\qq_\aaa, \qq_\aaa).
		\end{align}
		Consider the following diagram where we also fix notation for some projections and inclusions
		\begin{center}
			\begin{tikzcd}
				\ua \times \Q^\T_\n \arrow[r] \arrow[d, "j_\n"', hook]             & \ua \times \Quot_\alpha^\T \arrow[d, "j_\aaa", hook] \arrow[dd, "\pi_\aaa", bend left=60] \\
				Y \times \Q^\T_\n \arrow[d, "\pi"'] \arrow[r, "\sigma_\aaa"] & Y \times \Quot_\alpha^\T \arrow[d, "\tilde{\pi}_\aaa"]                                    \\
				\Q^\T_\n \arrow[r, "p_\aaa"]                                 & \Quot_\alpha^\T.                                                                          
			\end{tikzcd}
		\end{center}
		
		Here, $\Quot_\aaa^\T \coloneqq \Quot_{\ua}(F|_{\ua}, n_\aaa)^\T $. We also denote by $q_\aaa \colon Y \times \Quot_\aaa^\T \to Y$ the first projection so that $q = q_\aaa \circ \sigma_\aaa$. Note that 
		$$\qq = \bigoplus_{\aaa \in \Delta(Y)} \sigma_\aaa^* j_{\aaa *} \qq_\aaa $$
		as the supports of $\qq_\aaa$'s do not intersect. Now
		
		\begin{align*}
			\rph(q^*F, \qq) &= \rph(q^*F, \bigoplus_{\aaa \in \Delta(Y)} \sigma_\aaa^* j_{\aaa *} \qq_\aaa) 
			= \rpl \left( \bigoplus_\aaa R\sheafhom(q^*F, \sigma_\aaa^* j_{\aaa *} \qq_\aaa)\right) \\
			&= \bigoplus_\aaa \rpl(q^*F^\vee \otimes \sigma_\aaa^* j_{\aaa *} \qq_\aaa) 
			= \bigoplus_\aaa \rpl(\sigma_\aaa^*q_\aaa^*F^\vee \otimes \sigma_\aaa^* j_{\aaa *} \qq_\aaa) \\
			&= \bigoplus_\aaa \rpl\sigma_\aaa^*(q_\aaa^*F^\vee \otimes  j_{\aaa *} \qq_\aaa) 
			= \bigoplus_\aaa p_\aaa^* R\tilde{\pi_\aaa}_*(q_\aaa^*F^\vee \otimes  j_{\aaa *} \qq_\aaa)\\
			&= \bigoplus_\aaa p_\aaa^* R\tilde{\pi_\aaa}_*j_{\aaa *}(j_\aaa^*q_\aaa^*F^\vee \otimes  \qq_\aaa) 
			= \bigoplus_{\aaa} p_\aaa^* R\pi_{\aaa *} R \sheafhom(F_\aaa, \qq_\aaa)
		\end{align*}
		We have used flat base change applied to $p_\aaa$.
		
		\begin{align*}
			\rph(\qq, \qq) &= \rpl \left( \bigoplus_{\aaa, \beta} R \sheafhom(\sigma_\aaa^* j_{\aaa *} \qq_\aaa, \sigma_\beta^* j_{\beta *} \qq_\beta)\right) 
			= \bigoplus_\aaa \rpl \sigma_\aaa^* j_{\aaa *} R \sheafhom(\qq_\aaa, \qq_\aaa) \\
			&= \bigoplus_\aaa p_\aaa^* R \tilde{\pi}_\aaa j_{\aaa *} R \sheafhom(\qq_\aaa, \qq_\aaa) 
			= \bigoplus_\aaa p_\aaa^* R \pi_{\aaa *}R\sheafhom(\qq_\aaa, \qq_\aaa)
		\end{align*}
		
		Here we use that 
		$$R \sheafhom(\sigma_\aaa^* j_{\aaa *} \qq_\aaa, \sigma_\beta^* j_{\beta *} \qq_\beta) = 0 $$ for $\aaa \neq \beta$ as the sheaves $\sigma_\aaa^* j_{\aaa *} \qq_\aaa$ and $\sigma_\beta^* j_{\beta *} \qq_\beta$ have disjoint support.
		
		The remaining identity is shown analogously.
	\end{proof}
	\begin{cor}
		\label{comparison1}
		There are equalities in $K^{0, \T}(\Q^\T)$
		$$E_\bullet^\T = E_\bullet^{\prime} \quad \text{and} \quad N^{\vir} = N^{\prime \vir} $$
		and an equality in $A_*^{\T}(\Q^\T)$
		$$[\Q^\T]^{\vir} = [\Q^\T]^{\vir \prime}.$$
		Here, $N^{\vir} \coloneqq E_\bullet|_{\Q^\T}^{mov}$.
	\end{cor}
	\begin{proof}
		The first two equalities follow by taking fixed and moving parts, respectively, in Proposition~$\ref{kupzeroeq1}$. By Lemmas $\ref{applysiebfixapot}$ and $\ref{applytsiebapot}$, Siebert's formula for APOT (Theorem~$\ref{apotpotsiebert}$) can be applied to $\Q^\T$ and the two almost perfect obstruction theories. Therefore to compare the virtual classes it suffices to compare the global classes in $K^{0, \T}(\Q^\T)$ corresponding to these almost perfect obstruction theories. They are equal by the first part of the Corollary.
	\end{proof}
	
	Therefore, by Corollary \ref{comparison1} together with $(\ref{primevc})$ and $(\ref{primevnb})$
	\begin{equation}
		\label{step2}
		\frac{[\Q^\T]^{\vir}}{e(N^{\vir})} = \frac{[\Q^\T]^{\vir \prime}}{e(N^{\prime \vir})} = \sum_{\bf n} \prod_\alpha \frac{[\Quot_{\ua}(F|_{\ua}, n_\aaa)^\T]^{\vir}}{e(N_\alpha^{\vir})}.
	\end{equation}

	
	\subsection{Computation on 
		the local Quot scheme} \label{complocquot}
	Take an affine chart $\ua \into Y$ from the standard atlas on the toric variety $Y$ and consider
	$$Q_{\aaa} \coloneqq \Quot_{U_{\alpha}}(F|_{U_{\alpha}}, n_\aaa). $$
	Consider the almost perfect obstruction theory on $Q_\aaa$ induced from the APOT on $\Quot_Y(F, n_\aaa)$ from Theorem~$\ref{quotsemipot}$ via the inclusion $Q_\aaa \into \Quot_Y(F, n_\aaa)$. Denote by 
	\begin{equation}
		\label{univaaa}
		\sss_\alpha \into F_\aaa \onto \qq_\aaa
	\end{equation}
	the universal sequence on $\ua \times Q_\aaa$ corresponding to $Q_\aaa$. From now on, $\sss$ and $\qq$ denote the universal kernel and cokernel, respectively, corresponding to $\Quot_Y(F, n_\aaa)$. 
	Consider the commutative diagram with the fibered square
	\begin{center}
		\begin{tikzcd}
			U_\alpha \times Q_\alpha \arrow[rd, "\pi_\alpha"'] \arrow[r, hook] & Y \times Q_\alpha \arrow[d] \arrow[r, hook] & Y \times \Quot_Y(F, n_\aaa) \arrow[d, "\pi"] \\
			& Q_\alpha \arrow[r, hook]                                 & \Quot_Y(F, n_\aaa).                         
		\end{tikzcd}
	\end{center}
	Note that $\qq|_{Y \times Q_\aaa} = \qq_\aaa$. 
	Also, $R\sheafhom(\qq, \sss)|_{Y \times \ua} \simeq R\sheafhom(\qq_\aaa, \sss_\aaa)$ and hence $$\rph(\qq, \sss)|_{Q_\aaa} \simeq \rphalpha(\qq_\alpha, \sss_\alpha).$$
	Similar isomorphisms hold whenever $\qq$ is one of the entries in $\rph(-, -)$. Therefore for the construction of our almost perfect obstruction theory on $Q_\aaa$ we can consider perfect obstruction theories on each affine étale chart constructed via the pullback of the universal sequence (\ref{univaaa}) as in Proposition \ref{potlocally}. In particular, the obstruction sheaf is 
	\begin{equation}
		\label{obua}
		Ob_{Q_\aaa} \simeq \extpialpha{3}(\qq_\alpha, \sss_\alpha).
	\end{equation}

	Similarly as in the previous section, there is a torus $\T = (\mathbb{C}^*)^3$ acting on $Q_\aaa$, and we can apply the torus localization formula (Theorem~\ref{localization}) to the APOT on $Q_\aaa$ induced by the inclusion $Q_\aaa \into \Quot_Y(F, n_\aaa)$.  
	
	There is also a $\T^\prime \coloneqq (\mathbb{C}^*)^r$ action $\sigma_{Q_\aaa}$ on $Q_\aaa$ obtained by scaling the fibers of $F_\aaa \coloneqq F|_{U_{\alpha}}$, and it commutes with the $\T$-action, see \cite{Fasola_2021}. We recall the construction of $\sigma_{Q_\aaa}$.
	First note that, as $\ua \simeq \aff$, there is a (non-canonical) isomorphism
	$$F|_{\ua} \simeq \bigoplus_{i = 1}^r \oo_\ua \otimes \lambda_i \in K^{\T \times \T^\prime}(Q_\aaa), $$
	where $\lambda_i \in K^{\T \times \T^\prime}(\pt)$ are the torus weights.
	Consider a scheme $S$ and ${s = (s_i)_{1 \leq i \leq r} \in \T^\prime(S) = Aut_{\oo_S}(\oo_S)^r}$. Define the action
	$$\sigma_{Q_\aaa}(S) \colon \T^\prime(S) \times Q_\aaa(S) \to Q_\aaa(S) $$
	sending the pair $(s, [\bigoplus_{i = 1}^r \oo_{\ua \times S}\otimes \lambda_i \xrightarrowdbl{\kappa} \qq_S]) \in \T^\prime(S) \times Q_\aaa(S)$ to the composition
	$$\bigoplus_{i = 1}^r \oo_{\ua} \boxtimes \oo_S \boxtimes \lambda_i \xrightarrow{\bigoplus_{i = 1}^r \id \boxtimes s_i^{-1} \boxtimes \id} \bigoplus_{i = 1}^r \oo_{\ua} \boxtimes \oo_S \boxtimes \lambda_i \xrightarrowdbl{\kappa} \qq_S.$$
	Note that the action morphism $\sigma_{Q_\aaa} \colon \T^\prime \times Q_\aaa \to Q_\aaa$ which can be considered as an element in $Q_\aaa(\T^\prime \times Q_\aaa)$ is equivalently defined via a quotient on $\ua \times \T^\prime \times Q_\aaa$ which is the composition
	\begin{equation}
		\label{sqquotient}
		\bigoplus_{i = 1}^r \oo_{\ua \times Q_\aaa} \boxtimes \oo_{\T^\prime} \boxtimes \lambda_i \xrightarrow[\sim]{\bigoplus_{i = 1}^r \id \boxtimes u_i^{-1} \boxtimes \id} \bigoplus_{i = 1}^r \oo_{\ua \times Q_\aaa} \boxtimes \oo_{\T^\prime} \boxtimes \lambda_i \xrightarrowdbl{u_{\T^\prime}} \qq_{\T^\prime}.
	\end{equation}
	Here, $(u_i)_{1 \geq i \geq r} \in \T^\prime(\T^\prime) = Aut_{\oo_{\T^\prime}}(\oo_{\T^\prime})^r$ is the universal automorphism and $u_{\T^\prime}$ is the pullback of the universal quotient to $\ua \times \T^\prime \times Q_\aaa$. 
	
	\begin{lem}
		The almost perfect obstruction theory on $Q_\aaa$ induced by the inclusion $Q_\aaa \into \Q$ is $\T \times \T^{\prime}$-equivariant.
	\end{lem}
	\begin{proof}
		By the discussion above, the perfect obstruction theories on the affine étale charts that yield our APOT can be constructed via pullbacks of the universal sequence (\ref{univaaa}) as in Proposition \ref{potlocally}. It is therefore sufficient to check that the universal sequence (\ref{univaaa}) on $Q_\aaa$ can be made $\T \times \T^\prime$ equivariant. It can be made $\T$-equivariant as in the case of $\Q$. Hence it suffices to check it can be made $\T^\prime$-equivariant.
		
		By the discussion above, the action morphism $\sigma_{Q_\aaa}$ is given via the quotient $(\ref{sqquotient})$. On the other hand, via the universal property, it is given via the pullback $\sigma_{Q_\aaa}^*u$ of the universal quotient $u$ on $\ua \times Q_\aaa$. Hence, the two quotients $(\ref{sqquotient})$ and $\sigma_{Q_\aaa}^*u$ are isomorhic, and this isomorphism provides the equivariant structure on the universal quotient and hence also on the whole universal sequence. 
	\end{proof}
	We next show that the APOT on $Q_\aaa$ discussed above is equivariantly symmetric, the statement is made precise in Lemma~\ref{locetalsym} below. 
	
	Note that as $\ua \simeq \aff$, we have $\omega_\ua \simeq {\bf t} \otimes \oo_\ua \in K_0^{\T}$, where ${\bf t} \coloneqq t_1 t_2 t_3$ with $K^{\T \times \T^\prime}(pt) = \mathbb{Z}[t_1^{\pm}, t_2^{\pm}, t_3^{\pm}, w_1^{\pm}, \dots, w_r^{\pm}].$ Hence, $$\omega_{\pi_\alpha} \simeq {\bf t} \otimes \oo_{\ua \times Q_\aaa} \in K_0^{\T \times \T^\prime}(\ua \times Q_\aaa).$$ By abuse of notation, we always denote by ${\bf t}$ the appropriate pullback.   	
	
	Consider a $\T \times \T^\prime$-equivariant étale morphism $a_i \colon W_i \to Q_\aaa $ with $W_i$ affine. Denote by ${\varphi_i \colon E^\bullet_i \to \LL_{W_i}}$ the $\T \times \T^\prime$-equivariant POT on $W_i$ constructed from the universal sequence as in Proposition $\ref{potlocally}$. Put $E_{i \bullet} \coloneqq (E_i^\bullet)^\vee$ and denote by $\pi_i \colon \ua \times W_i \to W_i$ the second projection. Note that 
	
	\begin{equation}\label{omegapi}
		\omega_{\pi_i} = (\id \times a_i)^* \omega_{\pi_\aaa} \simeq (\id \times a_i)^* ({\bf t} \otimes \oo_{\ua \times Q_\aaa}) \simeq {\bf t} \otimes \oo_{\ua \times W_i}. 
	\end{equation}
	
	
	\begin{lem}
		\label{locetalsym}
		For every i, there is the equivariant symmetry
		\begin{equation}
			E_{i \bullet} \simeq  {\bf t}^{-1} \otimes E^\bullet_i [-1].
		\end{equation}
	\end{lem}
	\begin{proof}
		The proof is a slight modification of the proof of Proposition \ref{cyafflocal} in the Calabi--Yau case, we explain the necessary modifications. Denote by 
		\begin{equation}
			\label{uunivseqi}
			0 \to \sss_i \to F_i \to \qq_i \to 0
		\end{equation}
		the pullback of the universal sequence (\ref{univaaa}) to $\ua \times W$. 
		Consider the commutative diagram obtained from (\ref{uunivseqi}) analogously to diagram~(\ref{twodisttr})
		\begin{center}
			\begin{equation} \label{itwodisttr}
				\begin{tikzcd}
					{\rpih(\qq_i, F_i)[1]} \arrow[d, equals] \arrow[r, "{p_i}"] & {\rpih(\sss_i, \qq_i)} \arrow[d] \arrow[r, "{\varepsilon_i}"] & {E_{i \bullet}} \arrow[d, "\gamma_i"] \\
					{\rpih(\qq_i, F_i)[1]} \arrow[r]           & {\rpih(\qq_i, \qq_i)[1]} \arrow[r]  \arrow[d]              & {\rpih(\qq_i, \sss_i)[2]} \arrow[d, "q_i"] \\
					& {\rpih(F_i, \qq_i)[1]}  \arrow[r, equals] & {\rpih(F_i, \qq_i)[1].}  
				\end{tikzcd}
			\end{equation}
		\end{center}
		%
		Similarly to the Calabi--Yau case, we have the duality
		\begin{align*}
			(\rpih(F_i, &\qq_i)[1])^{\vee}[-1] \otimes \ttt^{-1} \simeq \rpih(\qq_i, F_i \otimes \omega_{\pi_i}[2])[-1] \otimes \ttt^{-1} \\ &\simeq \rpih(\qq_i, F_i \otimes \ttt)[1] \otimes \ttt^{-1} \simeq  \rpih(\qq_i, F_i)[1] 
		\end{align*}
		and analogously for all other entries in diagram (\ref{itwodisttr}). Note that the first isomorphism in the chain of isomorphisms above is by Grothendieck duality with respect to the smooth $3$-dimensional morphism $\bar{\pi}_i \colon Y \times W_i \to W_i$ which is the projection to the second factor. The second isomorphism is due to~$(\ref{omegapi})$.
		
		Therefore applying dual, then shifting by $[-1]$ and tensoring by $\ttt^{-1}$ to all entries and morphisms in diagram $(\ref{itwodisttr})$, yields the initial diagram. Therefore, 
		\begin{equation}
			\label{eupdot}
			E_{i \bullet} = E^\bullet_i[-1] \otimes \ttt^{-1}. \qedhere
		\end{equation} 
	\end{proof}

	Consider now the critical perfect obstruction theory (see Section \ref{crit}) on $Q_\aaa$. 
	By \cite{Fasola_2021}, it is $\T \times \T^\prime$-equivariant, also the equivariant symmetry stated analogously to the one in Lemma~\ref{locetalsym} holds, see \cite[Diagram 3.15]{Fasola_2021}. This POT on $Q_\aaa$ induces a POT on the fixed locus $Q_\aaa^{\T \times \T^\prime}$. 
	
	The fixes locus $Q_\aaa^{\T \times \T^\prime}$ also carries an APOT induced from the APOT on $Q_\aaa$ discussed in the beginning of subsection~\ref{complocquot}.
	%

	
	\begin{prop}
		\label{comparewithcrit}
		The APOT and the 
		POT on $Q_\aaa^{\T \times \T^\prime}$ possess the same virtual class and the same K-theoretic virtual normal bundle.
	\end{prop}
	\begin{proof}
		First recall that by \cite[Lemma 3.5]{Fasola_2021} the fixed locus $Q_\aaa^{\T \times \T^\prime}$ is smooth and, moreover, consists of finite number reduced isolated points. Therefore, it is sufficient to check the needed equalities over one point $x \in Q_\aaa^{\T \times \T^\prime}$. 
		
		We consider an étale morphism $\epsilon_W \colon W \to Q_\aaa$ with $W$ affine, and $x_W \in W$ a preimage of~$x$. The pullback to $W$ via $\epsilon_W$ of both obstruction theories are perfect obstruction theories which by Lemma~\ref{locetalsym} and \cite{Fasola_2021} are equivariantly symmetric.
		If now $\varphi_{E^\bullet_W} \colon E^\bullet_W \to \LL_W$ is one of these perfect obstruction theories, then the following equalities in $K_0^{\T \times \T^\prime}(x_W)$ hold
		\begin{align*}
			&E^\bullet_W|_{x_W} = h^0(E^\bullet_W|_{x_W}) - h^{-1}(E^\bullet_W|_{x_W}) = h^0(E^\bullet_W)|_{x_W} - h^{-1}(((E^\bullet_W)^\vee[1] \otimes \ttt)|_{x_W}) \\  &= \Omega_W|_{x_W} - h^{-1}((E^\bullet_W)|_{x_W}^\vee[1]) \otimes \ttt = \Omega_W|_{x_W} - h^0((E^\bullet_W)|_{x_W})^\vee \otimes \ttt =  \Omega_W|_{x_W} - T_W|_{x_W} \otimes \ttt.
		\end{align*}
		We have used that $E^\bullet_W|_{x_W}$ is a complex of vector spaces and hence taking dual commutes with taking cohomology. Therefore both perfect obstruction theories restricted to $x_W$ have the same class in $K_0^{\T \times \T^\prime}(x_W) \simeq K^{0, \T \times \T^\prime}(x_W).$

		
		Hence, there is also equality in $K_0^{\T \times \T^\prime}(x) \simeq K^{0, \T \times \T^\prime}(x)$ of the restrictions to $x \in Q_\aaa^{\T \times \T^\prime}$ of the K-theory classes of the initial obstruction theories on $Q_\aaa^{\T \times \T^\prime}$.
		Taking the moving parts of these restrictions yields the equality of the K-theoretic virtual normal bundles of the APOT and the POT.
		Taking the fixed parts yields the equality of the K-theory classes of the APOT and the POT in $K^{0, \T \times \T^\prime}(x)$. Finally, by applying Siebert formula for APOT (Theorem~\ref{apotpotsiebert}) and the Siebert formula for POT \cite{siebert} gives the equality of virtual classes.    	
	\end{proof}
	By applying torus localization formula for APOT (Theorem~\ref{localization}) on $Q_\aaa$ with respect to $\T$ and $\T \times \T^\prime$ we get the following equalities which will be used in the computation of the invariants
	\begin{equation}
		\label{step3}
		\frac{[Q_\alpha^\T]^{\vir}}{e(N_\alpha^{\vir})} = [Q_\aaa]^{\vir} =  \frac{[Q_\alpha^{\T \times \T^\prime}]^{\vir}}{e(N_\alpha^{\prime \vir})}.
	\end{equation}
	Note that by Proposition \ref{comparewithcrit} we can consider the right hand side in the equality above coming from the critical perfect obstruction theory on $Q_\aaa$.
	
	\subsection{Computation of the invariants}
	We are ready to compute the generating series $(\ref{genser})$ in the toric case.
	\begin{theor} \label{comptoric}
		Let $(Y, F)$ be a pair consisting of a smooth projective toric 3-fold $Y$ and a $\T$-equivariant locally free sheaf $F$ of rank $r$. Then
		$$\DT_{Y, F}(q) = \M((-1)^rq)^{r \int_Y c_3(T_Y \otimes \omega_Y)}.$$
	\end{theor}
	\begin{proof} 
		Put $Q \coloneqq \Q$ and $Q_\alpha \coloneqq \Quot_{U_{\alpha}}(F|_{U_{\alpha}}, n_{\alpha}).$ Then
		\begin{equation}
			\label{final}
			[Q]^{\vir} = \frac{[Q^\T]^{\vir}}{e(N^{\vir})} = 
			\sum_{\bf n} \prod_\alpha \frac{[Q_\alpha^\T]^{\vir}}{e(N_\alpha^{\vir})} = \sum_{\bf n} \prod_\alpha \frac{[Q_\alpha^{\T \times \T^\prime}]^{\vir}}{e(N_\alpha^{\prime \vir})}. 
		\end{equation}
		Recall that $$N^{\vir} = E_\bullet|_{\Q^\T}^{mov}.$$ 
		The first equality is by virtual torus localization (Theorem~\ref{localization}) applied to APOT $\varphi$, see Corollary~$\ref{summary}$. The second equality is by $(\ref{step2})$, the third equality is by $(\ref{step3})$. In $(\ref{final})$ we have arrived at the same expression as in the computation in the case when $F$ is simple and rigid in the proof of \cite[Theorem 9.7]{Fasola_2021}. Hence, we get the desired formula. 
		%
	\end{proof}	
	\section{Computation of the virtual invariants in the general case} \label{s7}
	We compute the virtual invariants for $\Q$ in the general case for a pair $(Y, F)$, where $Y$ is a smooth projective $3$-fold and $F$ a locally free sheaf on it. This extends Theorem \ref{comptoric} completing the proof of \cite[Conjecture 3.5]{Fasola_2021}.
	
	\subsection{Double point cobordism} 
	We recall some facts from the theory of double point cobordism following \cite{Lee-Pandharipande-cobordism}. 
	
	\begin{definition}
		Let $Y$ be a smooth quasi-projective variety of pure dimension. A morphism $\pi \colon Y \to \mathbb{P}^1$ is a double point degeneration over $0 \in \p^1$ if $\pi^{-1}(0) = A \cup B$ with $A$ and $B$ smooth Cartier divisors intersecting transversally along a smooth divisor.
	\end{definition}
	Denote by $D$ the intersection $A \cap B$ in the definition above and by $\mathbb{P}(\pi) \to D$ the projective bundle $\mathbb{P}(\mathcal{O}_D \oplus {N}_{D/A}) \to D$. 
	By \cite{Lee-Pandharipande-cobordism}, $N_{D/A} \otimes N_{D/B} \simeq \oo_D$, hence, $\mathbb{P}(\pi)$ is also isomorphic to $\mathbb{P}(\mathcal{O}_D \oplus {N}_{D/B})$.
	
	\begin{definition}
		Let $\mathcal{M}_{n, r}$ be the set of classes of isomorphic pairs $[Y, E]$, where $Y$ is a smooth quasi-projective variety of dimension $n$ and $E$ is a vector bundle on $Y$ of rank $r$. Denote by $\mathcal{M}_{n, r}^+$ the free abelian group generated by these classes.
	\end{definition}
	
	\begin{definition}
		\label{dpr}
		Let $\pi \colon Y \to \mathbb{P}^1$ be a double point degeneration with $\pi^{-1}(0) = A \cup B$ and $E$ a vector bundle on $Y$ of rank $r$. Let $\xi \in \mathbb{P}^1$ be a regular value of $\pi$.  Then the associated double point relation is  
		$$[Y_\xi, E|_{Y_\xi}] - [A, E|_A] -[B, E|_B] + [\mathbb{P}(\pi), E|_{\mathbb{P}(\pi)}].$$
		Here, $Y_\xi = \pi^{-1}(\xi)$, the bundle $E|_{\mathbb{P}(\pi)}$ is the pullback of $E$ via the composition $\mathbb{P}(\pi) \to D \subset Y$.
	\end{definition}
	
	Denote by $\mathcal{R}_{n, r} \subset \mathcal{M}_{n, r}^+$ the subgroup of all double point relations.
	\begin{definition}
		The double point cobordism theory for bundles on varieties is defined by
		$$\omega_{n, r} \coloneqq \mathcal{M}_{n, r}^+ / \mathcal{R}_{n, r}.$$
	\end{definition}
	
	We next recall the construction of the $\mathbb{Q}$-basis of $\omega_{n, r}$, which is the main result of \cite{Lee-Pandharipande-cobordism}. 
	\begin{definition}
		A partition pair of size $n$ and type $r$ is a pair $(\lambda, \mu)$ where
		\begin{enumerate}[(i)]
			\item $\lambda$ is a partition of $n$,
			\item $\mu$ is a sub-partition of $\lambda$ of length $\ell(\mu) \leq r$.
		\end{enumerate}
		The set of all partition pairs of size $n$ and type $r$ is denoted by $\mathcal{P}_{n, r}$.
	\end{definition}	
	\begin{rem}
		The sub-partition condition means $\mu$ is obtained by deleting parts in $\lambda$. Sub-partitions $\mu, \mu^\prime$ are equivalent if they differ by permuting equal parts of $\lambda$.	
	\end{rem}	
	Define the map 
	$$\phi \colon \mathcal{P}_{n, r} \to \omega_{n, r}$$
	as follows. Fix $(\lambda, \mu) \in \mathcal{P}_{n, r}$ and consider $\p^\lambda \coloneqq \p^{\lambda_1} \times \dots \times \p^{\lambda_{\ell({\lambda})}}.$ For each part $m$ of $\mu$, let $L_m$ be the line bundle on $\p^\lambda$ obtained by pulling-back $\mathcal{O}_{\p^m}(1)$ via projection to the factor $\p^\lambda \to \p^m$ corresponding to the part $m$. Put
	$$\phi(\lambda, \mu) \coloneqq [\p^\lambda, \mathcal{O}_{\p^\lambda}^{r - \ell(\mu)} \oplus \bigoplus_{m \in \mu}L_m].$$
	
	\begin{theor}{\cite[Theorem 1]{Lee-Pandharipande-cobordism}}
		\label{basisvb}
		For $n, r \geq 0$, there is an isomorphism of $\mathbb{Q}$-vector spaces
		$$\omega_{n, r} \otimes_{\z} \mathbb{Q} = \bigoplus_{(\lambda, \mu) \in \mathcal{P}_{n, r}} \mathbb{Q} \cdot \phi(\lambda, \mu).$$	
	\end{theor}
	
	\subsection{The computation of the invariants via cobordism theory.}
	Consider $\q[[t]]^*$ the multiplicative group of power series with constant term 1. Define a group homomorphism
	\begin{equation}
		\label{dtinv}
		\DT \colon (\m_{3, r}^+, +) \to (\q[[t]]^*, \cdot)
	\end{equation}
	by
	\begin{equation*}
		\label{onbasis}
		\DT([Y, F]) \coloneqq \DT_{Y, F}(q) = 1 + \sum_{n \geq 1} \DT_{Y, F}^n q^n,
	\end{equation*}
	where $\DT_{Y, F}^n$ is the $n$-th virtual invariant of $\Q$, see (\ref{ninv}).
	
	For the rest of this subsection, we make the following assumption. 
	
	\begin{assumption}
		\label{assumption}
		The group homomorphism $\DT$ respects double point relations.
	\end{assumption} 
	This assumption will be shown to follow from the degeneration formula which will be proven later. By Assumption \ref{assumption}, the homomorphism $\DT$ descends to $\omega_{3, r}$ giving a homomorphism
	$$ \DT \colon \omega_{3, r} \to \q[[t]]^* .$$
	
	This allows us to extend Theorem~\ref{comptoric} to the general case. The proof is analogous to the one for $r = 1$, see \cite[Section 14]{LEPA}.
	
	\begin{theor}
		\label{newgenser}
		Let $(Y, F)$ be a pair consisting of a smooth projective $3$-fold $Y$ and a locally free sheaf $F$ of rank $r$ on $Y$. Then
		$$\DT_{Y, F}(q) = \M((-1)^rq)^{r \int_Y c_3(T_Y \otimes \omega_Y)}.$$
	\end{theor}
	As before, $\M(q) = \prodd_{m \geq 1} (1 - q^m)^{-m}$ is the MacMahon function, the generating function for the number of plane partitions of non-negative integers.
	
	\begin{proof}
		By Theorem $\ref{basisvb}$, we can write
		$$k[Y, F] = \sum_{(\lambda, \mu) \in \pa_{3, r}} m_{\lambda, \mu} [\p^\lambda,\phi(\lambda, \mu)],$$ 
		with $k \ne 0$ and $m_{\lambda, \mu}$ integers for all $\mu$ and $\lambda$.
		Therefore,
		\begin{equation*}
			\DT([Y, F])^k = \prod_{(\lambda, \mu) \in \pa_{3, r}} \DT([\p^{\lambda}, \phi(\lambda, \mu)])^{m_{\lambda, \mu}} = \prod_{(\lambda, \mu) \in \pa_{3, r}} \DT_{\p^{\lambda}, \phi(\lambda, \mu) }(q)^{m_{\lambda, \mu}}.
		\end{equation*}
		
		By the computation of the invariants in the toric case Theorem~\ref{comptoric},  the row of equalities above continues as:
		$$  \prod_{(\lambda, \mu) \in \pa_{3, r}} \DT_{\p^{\lambda}, \phi(\lambda, \mu) }(q)^{m_{\lambda, \mu}} = \prod_{(\lambda, \mu) \in \pa_{3, r}} \M((-1)^rq)^{r m_{\lambda, \mu} \int_{\p^\lambda} c_3(T_{\p^\lambda} \otimes \omega_{\p^\lambda})} $$ 
		$$ = 
		\M((-1)^rq)^{r \int_{\p^\lambda} \sum_{(\lambda, \mu) \in \pa_{3, r}} m_{\lambda, \mu} c_3(T_{\p^\lambda} \otimes \omega_{\p^\lambda})}.$$
		This equals 
		$$\M((-1)^rq)^{k r \int_Y c_3(T_Y \otimes \omega_Y)},$$
		as algebraic cobordism respects Chern numbers.
		
		Together with $\DT([Y, F])(0) = 1$ and $\M(0) = 1$, we get 
		$$\DT_{Y, F}(q) = \DT([Y, F]) = \M((-1)^rq)^{r \int_Y c_3(T_Y \otimes K_Y)}$$ which completes the proof.
	\end{proof}
	
	%
	
	\subsection{Double point relation and the degeneration formula}
	
	In this subsection, we show that Assumption~\ref{assumption} follows from the degeneration formula (Theorem \ref{degf}). The proof is analogous to the rank one case in \cite{LEPA}. Consider a double point degeneration $\pi \colon X \to C = \p^1$, where $X$ is a smooth projective $4$-fold, and a locally free sheaf $E$ on $X$. Let $A \cup_D B$ be the special fiber.
	Taking a regular value $\xi \in \p^1$ yields a double point relation
	$$[Y_\xi, E|_{Y_\xi}] - [A, E|_A] -[B, E|_B] + [\mathbb{P}(\pi), E|_{\mathbb{P}(\pi)}].$$
	We need to show that the virtual invariants of $\Q$ respect it, namely that
	\begin{equation}
		\label{respdpr}
		\DT([X_\xi, E|_{X_\xi}]) \cdot \DT([A, E|_A])^{-1} \cdot \DT([B, E|_B])^{-1} \cdot \DT([\mathbb{P}(\pi), E|_{\mathbb{P}(\pi)}]) = 1,
	\end{equation}
	This relation will follow from a version of the {\it degeneration formula}, which is the following statement.
	\begin{theor}
		\label{degf}
		Let $\pi \colon X \to  C = \p^1$ be a double point degeneration, where $X$ is smooth projective $4$-fold, and $E$ a locally free sheaf on $X$.
		Then 
		\begin{equation}
			\label{degfexplicit}
			\DT([X_\xi, E|_{X_\xi}]) = \DT([A/D, E|_A]) \cdot \DT([B/D, E|_B]).
		\end{equation}
	\end{theor}
	The relative virtual invariants of Quot schemes of points in the formulation of Theorem \ref{degf} above will be defined later. The analog of the following proposition is discussed in \cite[Section 0.9]{LEPA}, we will present the proof for the seek of completeness.
	
	\begin{prop}
		The degeneration formula Theorem \ref{degf} implies that the virtual invariants of Quot schemes of points respect the double point relation.
	\end{prop}
	\begin{proof}
		We use the notation introduced before, the goal is to check $(\ref{respdpr})$. For brevity we omit the locally free sheaf $E$ and its restrictions as well as the square brackets when considering the homomorphism $\DT$.
		
		First note that the deformation to the normal cone \cite{Fulton_Intersection_Theory} of $D \subset A$ is a double point degeneration. The special fiber is the union $A \cup \p(N_{D/A} \oplus \mathcal{O}_D)$ intersecting at $D \subset \p(N_{D/A} \oplus \mathcal{O}_D)$ included as $\p(N_{D/A})$ with normal bundle $N_{D/A}^\vee$. Applying Theorem \ref{degf} in this case gives
		\begin{equation}
			\label{eqa}
			\DT(A) = \DT(A/D) \cdot \DT(\p(\mathcal{O}_D \oplus N_{D/A})/D).
		\end{equation}
		Similarly, 
		\begin{equation}
			\label{eqb}
			\DT(B) = \DT(B/D) \cdot \DT(\p(\mathcal{O}_D \oplus N_{D/B})/D).
		\end{equation}
		Here, $D \subset \p(\mathcal{O}_D \oplus N_{D/B})$ is included as $\p({N_{D/B}})$ with normal bundle $N_{D/B}^\vee$.
		We next apply Theorem~\ref{degf} to the deformation to the normal cone of ${D \subset \p(\mathcal{O}_D \oplus N_{D/A})}$ where $D$ is included as $\p(N_{D/A})$ with normal bundle $N_{D/A}^\vee$. This yields
		\begin{equation}
			\label{eqpi}
			\DT(\p(\pi)) = \DT (\p(\mathcal{O}_D \oplus N_{D/A})/D) \cdot \DT (\p(\mathcal{O}_D \oplus N_{D/A}^\vee)/D)
		\end{equation}
		Note that $D$ is included into $\p(\mathcal{O}_D \oplus N_{D/A}^\vee) \simeq \p(\mathcal{O}_D \oplus N_{D/B})$ as $\p(N_{D/A}^\vee) \simeq \p(N_{D/B})$ with normal bundle $N_{D/A} \simeq N_{D/B}^\vee.$
		Hence,
		\begin{equation}
			\label{eab}
			\DT(\p(\pi)) = \DT (\p(\mathcal{O}_D \oplus N_{D/A})/D) \cdot \DT (\p(\mathcal{O}_D \oplus N_{D/B})/D).
		\end{equation}	
		Now, combining equations $(\ref{degfexplicit}), (\ref{eqa}), (\ref{eqb})$ and $(\ref{eab})$ gives the needed equality
		\begin{equation}
			\DT(Y_\xi) \cdot \DT(A)^{-1} \cdot \DT(B)^{-1} \cdot \DT(\mathbb{P}(\pi)) = 1. \qedhere
		\end{equation}
	\end{proof}
	
	To prove Theorem $\ref{degf}$ we work with the relative Quot scheme of points on the stacks of expanded degenerations
	constructed by Li and Wu in \cite{li2011good}. In the sequel we recall the needed constructions partly following \cite{oberdieck2021marked} for the presentation.  
	
	
	\subsection{The stack of expanded degenerations corresponding to a pair}
	Let $(A, D)$ be a pair with $A$ a projective variety and $D \subset A$ a smooth divisor. Consider its normal bundle $N_{D/A}$ and the projective bundle
	$$\Delta = \p(N_{D/A} \oplus \oo_D) .$$
	It has two canonical sections $D_0 \coloneqq \p(N_{D/A}) \subset \Delta$ 
	and $D_{\infty} \coloneqq \p(\oo_D) \subset \Delta$, both isomorphic to $D$. For $k \geq 1$ denote by 
	$$\p_k \coloneqq \Delta \cup \Delta \cup \dots \cup \Delta $$
	the chain of $k$ copies of $\Delta$ where $D_\infty \subset \Delta$ in the $m$-th copy is glued with $D_{0}$ in the $(m+1)$-th copy for $m = 1, \dots k-1$.
	\begin{definition}
		The $k$-step expanded degeneration of $(A, D)$ for $k \geq 1$ is the pair $(A[k], D[k])$ where
		$$A[k] \coloneqq A \cup \p_k, $$
		where $D \subset A$ is glued to $D_0$ in the first copy of $\Delta$ in $\p_k$ and 
		$$D[k] \coloneqq D_\infty \subset A[k]$$
		where $D_\infty$ from the last copy of $\Delta$ in $A[k]$ is taken. 
		We call $A[k]$ the $k$-step expansion of $A$ and $D[k] \subset A[k]$ the relative divisor of the expansion.
	\end{definition}
	The $1$-step expanded degeneration of $(A, D)$ arises naturally as a special fiber of the degeneration of $A$ to the normal cone of $D$. More generally, the $k$-step expanded degeneration arises naturally by iterated degeneration to the normal cone.
	A stack of expanded degenerations $\mathscr{U}$ whose closed points are $k$-step expanded degenerations is constructed in \cite{li2011good}. Denote its universal family by 
	$$(\A, \D) \to \U. $$
	We briefly recall the construction following \cite{oberdieck2021marked}. Let $(\A_0, \D_0) \coloneqq (A, D)$ and $$\A_1 \coloneqq \Bl_{D \times 0}(A \times \mathbb{A}^1)$$ with $\D_1$ defined as the proper transform of $D \times \mathbb{A}^1$. Inductively define
	$$\A_{n} \coloneqq \Bl_{\D_{n-1} \times 0}(\A_{n-1} \times \mathbb{A}^1)$$ and put $\D_{n}$ to be the proper transform of $\D_{n-1} \times \mathbb{A}^1$. For every $n \geq 0$ there is a natural morphism 
	$$p_n \colon \A_n \to \mathbb{A}^n $$ 
	given by the composition $\A_n \to \A_{n-1} \times \mathbb{A}^1 \xrightarrow{p_{n-1} \times \id} \mathbb{A}^n$.
	
	The stack $\U$ is then constructed as follows. For a scheme $S$ set $\U(S)$ to be a family of expanded degenerations of $(A, D)$ over a scheme $S$, i.$\:$e.$\:$a pair $(\A_S, \D_S)$ where
	\begin{itemize}
		\item $\A_S$ is a scheme over $S$ with an $S$-projection $\A_S \to A \times S$
		\item $\D_S \subset \A_S$ is a Cartier divisor,
	\end{itemize}
	such that there exists an étale cover $\{S_\aaa \to S\}_\aaa$ of $S$ such that $(S_\aaa \times_S \A_S, S_\aaa \times_S \A_S)$ is isomorphic to $(f_\aaa^* \A_n, f_\aaa^* \D_n)$ for some $n \geq 0$ and morphism $f_\aaa \to \mathbb{A}^n$, and the isomorphism is compatible with the projections to $A \times S$.
	%
	
	\subsection{Stack of expanded degenerations corresponding to a simple degeneration} Let $X$ be a smooth variety and $C$ a smooth curve with a distinguished point $0 \in C$. Consider a simple degeneration 
	$$\epsilon \colon  X \to  C,$$
	i.$\,$e.$\,$a projective morphism $\epsilon$ such that its fibers $X_c$ over $c \neq 0$ are smooth and the fiber over $0 \in C$ is $X_0 = A \cup B$, a union of two smooth Cartier divisors $A$ and $B$ intersecting transversally along a smooth divisor $D \coloneqq A \cap B$.
	
	\begin{definition}
		A $k$-step expanded degeneration of the central fiber $X_0$ is the variety 
		$$X_0[k] = A \cup \mathbb{P}_k \cup B, $$
		where $\mathbb{P}_k = \Delta \cup \dots \cup \Delta$ with $\Delta = \mathbb{P}(N_{D/A} \oplus \oo_D)$ as before. 
		The component $A$ is glued via $D$ to $D_0 \subset \Delta$ of the first component of $\mathbb{P}_k$ and $B$ is glued via $D$ to $D_{\infty} \subset \Delta$ of the last component of $\mathbb{P}_k$.
	\end{definition}
	By \cite{li2011good}, there exists a stack of target expansions 
	\begin{equation}
		\label{targexp}
		\C \to C
	\end{equation}
	together with a universal family $\X \to \C$ such that for a morphism $S \to \C$ the pullback family $X_S \to S$ has fibers $X_s$ if $s \in S$ does not lie over $0$ and $X_0[k]$ for some $k$ if it lies over $0$. The stack $\C$ is a smooth irreducible zero-dimensional Artin stack and the map $(\ref{targexp})$ is an isomorphism over $C \, \setminus \, 0$.
	
	\begin{definition}
		A weighted assignment of $X_0[k]$ is a function 
		$$w \colon \{\Delta_0 \coloneqq A, \Delta_1, \dots, \Delta_k, \Delta_{k+1} \coloneqq B\} \to \mathbb{N}, $$
		where $\Delta_i$ is the $i$-th copy of $\Delta$ in $\p_{k} \subset X_0[k]$. The sum $\sum_{l = 0}^{k+1}w(\Delta_l)$ is called the total weight of $w$. A weight assignment of $X_c$, $c \neq 0$ is a single value assignment $w_s(X_s) \in \mathbb{N}$.
	\end{definition}
	Similarly to $(\ref{targexp})$, there is a stack of weight $n$ target expansions $\C^n$ and a universal family $$\X^n \to \C^n$$ such that 
	%
	%
	for a morphism $S \to \C^n$ the fibers of the pullback family $X^n_S \to S$ are pairs $(X_s, w_s)$  if $s \in S$ does not lie over $0$ and $(X_0[k], w)$ for some $k$ if it lies over $0$. Here, $w_s$ and $w$ are weighted assignments of weight $n$ of $X_s$ and $X_0[k]$, respectively. By forgetting the weights, we obtain morphisms $\X^n \to \X$ and $\C^n \to \C$ which are étale, and the following natural diagram is Cartesian, see \cite[Section 2]{li2011good} for more details.
	\begin{center}
		\begin{tikzcd}
			\X^n \arrow[r] \arrow[d] & \X \arrow[d] \\
			\C^n \arrow[r]           & \C        
		\end{tikzcd}
	\end{center}
	
	
	Put $\C_0 \coloneqq \C \times_C 0 $ and $\X_0 \coloneqq \X \times_C 0$, similarly we also get $\C^n_0$ and $\X^n_0$ for the weighted version. Denote by $$\mathscr{P}_n \coloneqq \{\n = \{n_1, n_2\} \in \mathbb{N} \times \mathbb{N} \mid n_1 + n_2 = n\}$$
	the set of possible splittings of $n$ into two parts. Such a splitting $\n$ of $n$ gives rise to the stack
	$\C^{+, \n}_0 $
	of weighted marked expansions of the central fiber together with the universal family 
	$$\X^{+, \n}_0 \to \C^{+, \n}_0.$$
	The $\CC$-valued points of $\C^{+, \n}_0 $ are triples $(X_0[k], w, D_m)$ for some $m \leq k \in \mathbb{N}$, where $${D_m \coloneqq \Delta_m \cap \Delta_{m+1} \subset X_0[k]}$$ is the $m$-th copy of the divisor $D$ and $w$ is a weighted assignment of total weight $n$ such that $\sum_{l = 0}^m w(\Delta_l) = n_1 $ and $\sum_{l = m+1}^k w(\Delta_l) = n_2.$ 
	
	\begin{prop}[\cite{li2011good}]
		\label{lbs}
		There are canonical line bundles with sections $(L_\n, s_\n)$ on $\C^n$, indexed by $\mathscr{P}_n$, such that
		\begin{enumerate}
			\item let $t \in \Gamma(\mathcal{O}_{\mathbb{A}^1}) $ be the standard coordinate function and $p \colon \C^n \to \mathbb{A}^1$ the projection from the construction, then
			$$\bigotimes_{\n \in \pn} L_\n \simeq \mathcal{O}_{\C^n} \quad \text{and} \quad \prod_{\n \in \pn}s_\n = p^* t.$$
			\item the morphism $\Phi_\n$ factors through the zero section $Z(s_\n) \subset \C^n$ giving an isomorphism $$\C_0^{+, \n} \simeq Z(s_\n) \subset \C^n.$$
		\end{enumerate}	
	\end{prop}

	\subsection{The relative Quot scheme of points on the stacks of expanded degenerations}
	Let $X$ be a smooth variety and $C$ a smooth curve with a distinguished point $0 \in C$. As before, we consider $X \to C$ a simple degeneration such that $X_c$ is smooth for $c \neq 0$ and the central fiber $X_0 = A \cup B$ is a union of two smooth Cartier divisors intersecting transversally along a smooth Cartier divisor $D$.
	
	Let $F$ be a locally free sheaf on $X$. We denote by the same symbol its restrictions to $A$, to $B$ and its pullback to the $k$-step expanded degenerations $X_0[k]$ for $k \in \mathbb{N}$ via the natural projections. Denote by $\F$ the pullback of $F$ to the universal family $\X$ of the stack of expanded degenerations, to it's marked or weighted versions and also to the universal families $\A$ and $\mathscr{B}$ of expanded degenerations corresponding to pairs $(A, D)$ and $(B, D)$, respectively.  
	\begin{definition}
		\label{stadm}
		We call $Q \in Coh(A[k])$ with $\dim (\supp Q) = 0$ stable admissible if 
		\begin{enumerate}
			\item $\supp Q$ does not meet $D[k] \subset A[k]$ and the singular locus,
			\item for each $\Delta \subset A[k]$, its intersection with $\supp Q$ is nonempty.
		\end{enumerate} 
	\end{definition}
	
	Following \cite{li2011good}, we next define the relative Quot scheme for a fixed $n \in \mathbb{N}$ 
	\begin{equation}
		\label{relquotpairs}
		\qda \to \U
	\end{equation}
	on a stack of expanded degenerations corresponding to a pair $(A, D)$. For a scheme $S$, set $\qda(S)$ to be a family of relative quotients, i.$\,$e.$\,$a pair $(\mathscr{D}_S \subset \mathscr{A}_S, \psi_S)$ where
	\begin{itemize}
		\item $\mathscr{D}_S \subset \mathscr{A}_S \to S$ is a family of expanded degenerations of $(A, D)$ over $S$
		\item $\psi_S \colon \F_S \onto Q_S$ is a quotient where $\F_S$ is the pullback of $\F$ to $\mathscr{A}_S$ and $Q_S$ is $S$-flat such that for each $s \in S$ the restriction $\F_S|_s$ has zero-dimensional support of length $n$ and is stable admissible.
	\end{itemize} 
	By \cite{li2011good}, $\qda$ is a proper separated Deligne-Mumford stack.
	
	Similarly to Definition \ref{stadm}, there is also a notion of stable admissibe zero-dimensional coherent sheaves on $X_0[k]$ for $k \in \mathbb{N}$.
	For $n \in \mathbb{N}$, the relative Quot scheme 
	\begin{equation}
		\label{qxc}
		\qxc \to \C
	\end{equation}
	corresponding to the simple degeneration $X \to C$ is defined similarly to the case of pairs (\ref{relquotpairs}). Its $\CC$-valued points are zero-dimensional quotients of length $n$ on $X_0[k]$ for some $k \in \mathbb{N}$ which are stable admissible.
	By \cite{li2011good}, $\qxc$ is a proper separated Deligne-Mumford stack. Note that the morphism $(\ref{qxc})$ can be factored as 
	$$\qxc \xrightarrow{\pi_n} \C^n \to \C, $$ 
	where the second morphism is the forgetful morphism. We next define $\qxcm$ from the fibered diagram
	\begin{center}
		\begin{equation}
			\label{defqxcm}
			\begin{tikzcd}
				\qxcm \arrow[r] \arrow[d] & \qxc \arrow[d, "\pi_n"] \\
				{\C_0^{+, \n}} \arrow[r, "\tilde{\iota_\n}"]  & \C^n                   
			\end{tikzcd}
		\end{equation}
	\end{center} 
	Given a splitting $\n = (n_1, n_2) \in \pn$, there exists a ``gluing morphism''
	$$\Phi_{\n} \colon \qdaone \times \qdbtwo \to \qxcm.$$
	%
	%
	\begin{theor}\cite[Theorem 5.27]{li2011good}
		\label{lnsn}
		Let $(L_\n, s_\n)$ be line bundles with sections on $\C^n$ as in Proposition~\ref{lbs}. Then
		\begin{enumerate}
			\item $\otimes_\n \pi_n^* L_\n \simeq \mathcal{O}_{\qxc}$ and $\prod_\n \pi_n^*s_\n = \pi_n^* p^*t$
			\item $\qxcm = Z(\pi_n^* s_\n) $
			\item $\Phi_{\n}$ is an isomorphism of DM stacks.
		\end{enumerate}
	\end{theor}
	%
	
	\subsection{Virtual classes and the degeneration formula} From now on we assume that the smooth variety $X$ is projective of dimension $4$ so that the fiber $X_c$ over $c \neq 0$ of the simple degeneration $X \to C$ and the components $A$ and $B$ in the decomposition of the central fiber $X_0 = A \cup_D B$ are smooth projective $3$-dimensional varieties. 
	\begin{rem}
		Note that in Theorem \ref{degf} we aim to prove the degeneration formula corresponding to double point degenerations. However, without loss of generality it is sufficient to consider only simple degenerations as away from the special fiber we consider only fibers over regular values.
	\end{rem}
	\begin{theor}
		\label{qdarelapot}
		The Quot scheme $\qa \coloneqq \qda \to \U$ admits a relative almost perfect obstruction theory of relative dimension zero.
	\end{theor}
	\begin{proof}
		Let $$0 \to \sss_\A \to \F_\A \to \qq_\A \to 0 $$
		be the universal sequence of $\qa$. Let$\pi_\A \colon \A \times_\U \qa \to \qa$ be the projection to the second factor. There exists a relative obstruction theory on $\qa$
		\begin{equation}
			\label{relot}
			(R\sheafhom_{\pi_\A}(\sss_\A, \qq_\A))^\vee \to \LL_{\qa/ \U}
		\end{equation}
		defined via the reduced Atiyah class \cite{reducedatclstacks} similarly as in the absolute case recalled in Section \ref{derobstr}.
		The construction of the relative almost perfect obstruction theory on $\qa$ goes now as in the absolute case Theorem~\ref{quotsemipot}. Consider an affine étale cover $\{U_i \to \qa\}_i$ of $\qa$ and construct a perfect obstruction theory on each $U_i$ adjusting $(\ref{relot})$ to make it two-term.
		As in the absolute case, to show that we get a perfect obstruction theory on each $U_i$ we apply Proposition~\ref{crit}. We use that $\qa$ parametrizes quotients with support on a non-singular locus of $k$-step expansions $A[k]$.
		The local construction glues to a global almost perfect obstruction theory by an analogous argument as in the proof of Theorem~\ref{quotsemipot} for the absolute case.
	\end{proof}
	Note that as $\U$ is of dimension zero, we obtain a virtual cycle $$[\qda]^{\vir} \in A_0(\qda)$$ in degree zero. The  relative virtual invariants of Quot scheme of points corresponding to a pair $(A, D)$ are now defined as 
	$$\DT^n_{A/D, F} \coloneqq \deg[\qda]^{\vir}. $$
	We also consider the generating series 
	\begin{equation}
		\label{reldt}
		\DT_{A/D, F}(q) \coloneqq 1 + \sum_{n \geq 1} q^n \DT^n_{A/D, F} \eqqcolon \DT([A/D, F]).
	\end{equation}
	The right hand side in the formula above appears in formulation of Theorem $\ref{degf}$, and $[A/D, F]$ denotes the isomorphism class of a relative pair $D \subset A$ together with a locally free sheaf $F$ on $A$.
	\begin{theor}
		\label{apotqxc}
		The Quot scheme $\qxc \to \C^n$ admits a relative almost perfect obstruction theory of relative dimension zero.
	\end{theor}
	\begin{proof}
		The proof is analogous to the proof of Theorem \ref{qdarelapot} and the absolute case Theorem~\ref{quotsemipot}.
		The main idea is to adjust the standard relative obstruction theory  
		$$(R\sheafhom_{\pi_\X}(\sss_\X, \qq_\X))^\vee \to \LL_{\qxc/ \C^n}$$
		constructed from the universal sequence on $\X \times_{\C^n} \qxc$ and the reduced Atiyah class \cite{reducedatclstacks}, to make it $2$-term étale locally on an affine scheme. The almost perfect obstruction theory is then constructed by considering any étale affine cover of $\qxc$ with perfect obstruction theories on it.
	\end{proof}
	\begin{rem}
		As the map $\C^n \to \C$ is étale, Theorem~\ref{apotqxc} also gives a relative almost perfect obstruction theory on $\qxc$ over $\C$. 
	\end{rem}
	\begin{prop}
		\label{gysinc}
		Let $c \neq 0$ and $i_c \colon c \into C$ be the inclusion. Then 
		$$i_c^![\qxc]^{\vir} = [\Quot_{X_c}(F, n)]^{\vir}. $$
	\end{prop}
	\begin{proof}
		Consider the following fibered diagram
		\begin{center}
			\begin{tikzcd}
				{\Quot_{X_c}(F, n)} \arrow[r, hook] \arrow[d] & \qxc \arrow[d] \\
				c \arrow[d] \arrow[r, "\tilde{i_c}", hook]    & \C \arrow[d]   \\
				c \arrow[r, "i_c", hook]                      & C             
			\end{tikzcd}
		\end{center}
		
		By the excess intersection formula \cite[Theorem 6.3]{Fulton_Intersection_Theory}, 
		$i_c^![\qxc]^{\vir} = \tilde{i_c}^![\qxc]^{\vir}$. The latter class then equals $[\Quot_{X_c}(F, n)]^{\vir}$ by \cite{Behrend_1997} as the restriction of the relative almost perfect obstruction theory on $\qxc \to \C$ agrees with the one on $\Quot_{X_c}(F, n)$.
	\end{proof}
	
	\begin{prop}
		\label{gysin0}
		Let $i_0 \colon 0 \into C$ be the inclusion of $0$. Then
		$$i_0^![\qxc]^{\vir} = \sum_{\n} \iota_{\n *} [\qxcm]^{\vir}, $$
		where $\iota_\n \colon \qxcm \into \qxcz$ is the natural inclusion.
	\end{prop}
	\begin{proof}
		Consider the fibered diagram
		\begin{center}
			\begin{tikzcd}
				\disj_{\n} \qxcm \arrow[r,"\tilde{\nu} \coloneqq (\iota_\n)_{\n}"] \arrow[d] & \qxcz = \bigcup_{\n}\qxcm \arrow[r, hook] \arrow[d]                             & \qxc \arrow[d, "\pi_n"] \\
				{\disj_{\n} \C_0^{+, \n}} \arrow[r, "\nu"]  & {\C^n_o = \bigcup_{\n} \C_0^{+, \n}} \arrow[d] \arrow[r, "i_0^n", hook] & \C^n \arrow[d]          \\
				& \C_0 \arrow[r, "\tilde{i_0}"] \arrow[d, "pr"']                               & \C \arrow[d]            \\
				& 0 \arrow[r, "i_0", hook]                                             & C.                     
			\end{tikzcd}
		\end{center}
		First note that $\qxcz$ defined from the fibered diagram above admits an almost perfect obstruction theory relative to $\C^n_0$ restricted from the one on $\qxc \to \C^n$ or, equivalently, it can be constructed from the universal sequence in a usual way. Note that the almost perfect obstruction theories on $\qxc$ and $\qxcz$ can be considered relatively to $\C$ and $\C_0$, respectively, because the map $\C^n \to \C$ is étale.
		Therefore by \cite{Behrend_1997}, 
		\begin{equation}
			\label{virtfiber}
			\tilde{i_0}^![\qxc]^{\vir} = [\qxcz]^{\vir} 
		\end{equation}
		Similarly, there is an almost perfect obstruction theory on $\qxcm \to \C^{+, \n}_0$ for each $\n \in \pn$. We next note that by Proposition~\ref{lbs}, $\C^n_0$ is a finite union of closed substacks $\C_0^{+, \n}$ which are smooth zero-dimensional Artin stacks. Therefore, the map $\nu$ in the diagram above is the normalization and is of Deligne--Mumford type of degree one. Moreover, by Proposition~\ref{lnsn}, $\qxcz$ is a finite union of its closed substacks $\qxcm$, and the relative APOT on 
		$$\disj_{\n} \qxcm \to {\disj_{\n} \C_0^{+, \n}}, $$
		which on each component is an APOT on $\qxcm \to \C^{+, \n}_0$, can be viewed as an APOT obtained as a pullback via $\tilde{\nu}$ of the APOT on $\qxcz \to \C^n_0$.
		Therefore, Costello's pushforward formula \cite[Proposition~5.29]{manolachevirtpullb} can be applied and hence
		\begin{equation}
			\label{virtdisj}
			[\qxcz]^{\vir} = \tilde{\nu}_* [\disj_{\n} \qxcm]^{\vir} = \sum_\n \iota_{\n *} [\qxcm]^{\vir}. 
		\end{equation}
		%
		%
		%
		Finally note that as $\rk N_{\C_0/\C} = \rk N_{0/C} = 1$, the excess intersection formula for Artin stacks \cite{vistoliinttheor} gives that
		$ i_0^![\qxc]^{\vir} = \tilde{i_0}^![\qxc]^{\vir}$.  Combining this with (\ref{virtfiber}) and (\ref{virtdisj}) finishes the proof.
	\end{proof}
	\begin{prop}
		For a splitting $\n = (n_1, n_2) \in \pn$ of $n$ there is the equality of virtual classes
		$$[\qxcm]^{\vir} = [\qdaone]^{\vir} \times [\qdbtwo]^{\vir}. $$
		Here, $\qdbtwo$ is the relative Quot scheme on the stack of expanded degenerations corresponding to the pair $(B, D)$.
	\end{prop}
	\begin{proof}
		In this proof we will for brevity use the following notation: $\qn \coloneqq \qxcm$, $\qa \coloneqq \qdaone$, $\qb \coloneqq \qdbtwo$.
		Let $\sss_\A \into F_\A \onto \qq_\A$ be the universal sequence corresponding to $\qa$ and let $\pi_\A \colon \A \times_\U \qa \to \qa$ be the second projection. We will use analogous notations for the universal sequences and the projection maps for $\qb$ and $\qn$. Note that by Theorem \ref{lnsn}, $\qn \simeq \qa \times \qb$. Also note that by construction $\X = \A \cup_\D \mathscr{B}$.
		
		Denote by $\widetilde{\sss_\A}$ and $\widetilde{\qq_\A}$ the pullbacks to 
		$$\bigchi_\A \coloneqq (\A \times_\U \qa) \times_\qa \qn \simeq \A \times_\U \qn \subset \X \times_\cz \qn \eqqcolon \bigchi_\n $$
		of $\sss_\A$ and $\qq_\A$, respectively. We will use similar notation for $\widetilde{\sss_{\mathscr{B}}}$ and $\widetilde{\qq_{\mathscr{B}}}$.
		Note that $\qq_\n \simeq \widetilde{\qq_\A} \oplus \widetilde{\qq_\B}$ as the supports of the universal quotients do not meet the divisor $\D$. Therefore also $ R\sheafhom(\qq_\n, \sss_\n) \simeq R\sheafhom(\widetilde{\qq_\A}, \widetilde{\sss_\A}) \oplus R\sheafhom(\widetilde{\qq_\B}, \widetilde{\sss_\B})$ and 
		\begin{align*}
			R\sheafhom_{\pi_\n}(\qq_\n, \sss_\n) &\simeq R\sheafhom_{\widetilde{\pi_\A}}(\widetilde{\qq_\A}, \widetilde{\sss_\A}) \oplus R\sheafhom_{\widetilde{\pi_\B}}(\widetilde{\qq_\B}, \widetilde{\sss_\B}) \\ &\simeq 
			R\sheafhom_{\pi_\A}(\qq_\A, \sss_\A) \boxplus R\sheafhom_{\pi_\B}(\qq_\B, \sss_\B).
		\end{align*}
		Here, $\widetilde{\pi_\A} \colon \bigchi_\A \to \qn$ denotes the projection, we use similar notation for $\widetilde{\pi_\B}$. Taking now third cohomology yields the relation on obstruction sheaves
		$$Ob_\qn \simeq Ob_\qa \boxplus Ob_\qb. $$
		We also have $c_\qn \simeq c_\qa \times c_\qb$ which follows from \cite{Behrend_1997}. Therefore,
		\begin{align*}
			[\qn]^{\vir} &= 0^!_{Ob_\n}[c_\qn] = 0^!_{Ob_\qa \boxplus Ob_\qb}([c_\qa] \times [c_\qb]) \\ &= 0^!_{Ob_\qa}[c_\qa] \times 0^!_{Ob_\B}[c_\qb] = [\qa]^{\vir} \times [\qb]^{\vir},
		\end{align*}
		where the compatibilities for the generalized Gysin pullback follow from the corresponding compatibilities for the usual Gysin pullback \cite{Fulton_Intersection_Theory}. 
	\end{proof}
	
	Using now that $\deg i_c^![\qxc] = \deg i_0^![\qxc]$ for $c \neq 0$ 
	together with Proposition~\ref{gysinc} and Proposition~\ref{gysin0} yields
	$$\deg[\Quot_{X_c}(F, n)]^{\vir} = \deg[\qxcz]^{\vir} = \sum_{\n \in \pn} \deg[\qdaone]^{\vir} \cdot \deg[\qdbtwo]^{\vir}, $$ which can be rewritten as 
	$$\DT^n_{X_c, F} = \sum_{\n \in \pn} \DT^{n_1}_{A/D, F} \cdot \DT^{n_2}_{B/D, F}, $$
	yielding the desired formula $$\DT_{X_c, F}(q) = \DT_{A/D, F}(q) \cdot \DT_{B/D, F}(q)$$
	completing the proof of Theorem \ref{degf}.

	\appendix
	\section{}
	
	\begin{lem}
		\label{huybrthom}
		Let $C^\bullet = (C^i \to \dots \to C^j)$ be a perfect complex on a scheme $X$ such that for each closed point $\iota_x \colon x \into X$, we have $H^i(L\iota_x^*C^\bullet) = 0$ for $i \neq [a,b]$. Then $C^\bullet$ is quasi-isomorphic to a $(b-a+1)$-term complex of locally free sheaves in degrees $[a, b]$.
	\end{lem}
	\begin{proof}
		Restricted to each closed fiber $C^\bullet$ has cohomology in degrees from $a$ to $b$. Hence, if $j>b$ is maximal with $C^j$ non-zero, then on each fiber
		$$C^{j-1} \to C^j$$ is surjective. Hence, by Nakayama Lemma on each stalk over a closed point this map is surjective and is therefore surjective globally with locally free kernel. Replacing $C^{j-1}$ by this kernel and $C^j$ by zero we can inductively assume that $j=b$. Similarly, if $i<a$ is minimal with $C^i$ non-zero, then $C^i \to C^{i+1}$ is injective on fibers and hence injective globally, as $C^i$ and $C^j$ are locally free, with locally free cokernel. We can replace $C^{i+1}$ by this cokernel and $C^i$ by zero and assume inductively that $i=a$. Hence, $C^\bullet$ is isomorphic to a $(b-a+1)-$term complex. 
	\end{proof}
	
	\begin{lem}
		\label{finite}
		Let $X$ be a smooth scheme of dimension $d$ with a coherent sheaf $F$ on it. Assume the sequence
		$$ 0 \to E_d \to E_{d-1} \to \dots \to E_0 \to F \to 0$$ is exact with $E_i$ being locally free for $i \leq d-1$. Then $E_d$ is also locally free.
	\end{lem}
	\begin{proof}
		It suffices to check that for each closed point $x \in X$ the localization $(E_d)_x$ is free. The statement is local, so we can assume $E_d = \widetilde{M}$ on $X = Spec A$ such that $M$ is a finitely generated $A$-module with $A$ a local ring. We need to check that 
		$M$ is free. It suffices to show that its projective dimension is zero as a finitely generated projective module over a local ring is free. This follows from \cite[\href{https://stacks.math.columbia.edu/tag/00NE}{Tag 00NE}]{stacks-project} together with Auslaender--Buchsbaum formula \cite[\href{https://stacks.math.columbia.edu/tag/090U}{Tag 090U}]{stacks-project} and the fact that for a regular ring its depth coincides with its dimension.
	\end{proof}
	
	\begin{lem}
		\label{flatperfect}
		If a coherent sheaf $E$ on a product $X \times B$ with $X$ smooth of dimension $n$ is flat over $B$ then ${E \in \mathcal{D}(X \times B)}$ is perfect.
	\end{lem}
	\begin{proof}
		As $E$ is coherent, there exists a locally free resolution ${E^\bullet \to E}$ on $X \times B$. Consider its truncation
		$$\tau^{\geq -n}(E^\bullet) = (0 \to E^n \to \dots \to E^1 \to E^0 \to 0) $$
		so that 
		\begin{equation}
			\label{truncex}
			0 \to E^n \to \dots \to E^1 \to E^0 \to E \to 0
		\end{equation}
		is exact with $E^i$ being locally free for $i = 0, 1, \dots , n-1$. Hence, these sheaves are flat over $X \times B$ and hence, as the projection to $B$ is flat, also flat over $B$. As $E$ is also flat over $B$, the remaining sheaf $E^n$ is flat over $B$ as well. 
		
		We next show $E^n$ is flat over the product $X \times B$  which would imply it is locally free as $E^n$ is coherent. This would finish the proof. By \cite[Proposition~7.39]{goerzwedhornag} the restriction of $(\ref{truncex})$ to all fibers of the projection to $B$ remains exact. Hence, by Lemma \ref{finite} as these fibers are isomorphic to $X$ and hence smooth the restriction 
		$$E^n|_{X_b}$$
		to the fiber $X_b$ over $b \in B$ is locally free for all $b \in B$. In particular the restriction $E^n|_{X_b}$ is flat over the fiber $X_b$. As $E^n$ is flat over $B$ \cite[\href{https://stacks.math.columbia.edu/tag/039C}{Tag 039C}]{stacks-project} can be applied (taking the map $f$ to be id). Therefore $E^n$ is flat over $X \times B$ and hence locally free.
	\end{proof}
	
	\begin{lem}
		\label{exun}
		Let $A \xrightarrow{\alpha} B \xrightarrow{\beta} C \xrightarrow{+1}$ be a distinguished triangle in a derived category $\mathcal{D}$. Let ${\gamma \colon D \to B}$ be a morphism such that the composition $\beta \circ \gamma$ vanishes. Then there exists a morphism $\delta \colon D \to A$ so that $\alpha \circ \delta = \gamma$. Moreover, if ${\Hom_\mathcal{D}(D, C[-1]) = 0}$ then the morphism $\delta$ is unique.
	\end{lem}
	\begin{proof}
		The statement follows from the long exact sequence obtained by applying ${\Hom_\mathcal{D}(D, -)}$ to our distinguished triangle
		\begin{equation}
			\Hom_\mathcal{D}(D, C[-1]) \to \Hom_\mathcal{D}(D, A) \to \Hom_\mathcal{D}(D, B) \to \Hom_\mathcal{D}(D, C).
			\qedhere
		\end{equation}
	\end{proof}	
		\bibliographystyle{amsplain}
		\bibliography{bib}{}
	\end{document}